\documentclass[12pt]{amsart}
\usepackage{hyperref}
\usepackage{amsfonts,amssymb,amsmath,amsthm,amscd,euscript,array,mathrsfs,comment}
\usepackage[all]{xy}
\usepackage{bbm}
\usepackage{booktabs}
\usepackage{color}
\usepackage{makecell}
\input{mymacros-3.sty}

\renewcommand{\arraystretch}{1.2}
\title[Real modules of real Lie superalgebras]{Classification of irreducible real modules of  real Lie superalgebras}
\begin{document}

\author{Siddhartha Sahi}
\address{Department of Mathematics,    
    Rutgers University,paper
    110 Frelinghuysen Rd, 
    Piscataway, NJ 08854-8019,
    USA
  }
  \email{sahi@math.rutgers.edu}

  \author{Hadi Salmasian}

  \address{
Department of Mathematics and Statistics,
    University of Ottawa,
    585 King Edward Ave,
    Ottawa, Ontario,
    Canada K1N 6N5%
  }

  \email{hadi.salmasian@uottawa.ca}

\author{Vera Serganova}
\address{
Department of Mathematics,
    University of California at Berkeley,
    969 Evans Hall,
    Berkeley, CA 94720,
    USA
    } 
\email{serganov@math.berkeley.edu}


\maketitle

\begin{abstract}
We classify irreducible finite-dimensional modules of a collection of real Lie superalgebras that includes the simple ones, their classical variants, complex Lie superalgebras after restriction of scalars, and all real Lie algebras. 
Our strategy is to reduce this classification to determining the orbits of the parity and conjugation functors on  irreducible modules of the complexifications of the aforementioned algebras. Then we provide explicit results for the computation of these orbits.
 For Lie superalgebras of basic type or of type $\mathbf{Q}(n)$, our classification applies to any highest-weight parametrization of irreducible complex modules with respect to an arbitrary Borel subalgebra. 

As a consequence, in the special case of real simple Lie algebras we obtain  a new perspective on the classification  of real simple modules  and  establish a conceptual connection with Kostant’s cascade of strongly orthogonal roots.
\end{abstract}

\section{Introduction}

While the study of finite dimensional modules of  
complex simple Lie (super)algebras has been a mainstream research direction in algebra, much less attention has been paid to such modules over real numbers. Still, the problem of characterizing modules of real simple Lie algebras  has a long history, going back to the works of \'{E}. Cartan, Iwahori, and Karpelevich~\cite{Cartan,Iwahori,Karpelevich}. The monograph by Onishchik~\cite{Onishchik} offers a simplified exposition of the classification from the viewpoint of~\cite{Karpelevich}. We refer the reader to the introduction of~\cite{Onishchik} for a meticulous examination of the history of this problem. For the analogous problem over more general fields (but in the framework of representations of reductive groups), we refer the reader to~\cite{Tits}.

In the super setting, the only work on classification of representations over general fields that we are aware of is~\cite{Hayashi}. The qualitative result of Hayashi's paper (see~\cite[Thm 1.2]{Hayashi}) is in the  setting of representations of affine supergroup schemes, and its more quantitative result (see~\cite[Thm 1.7]{Hayashi}) is 
stated for quasi-reductive algebraic supergroups. Despite working in a general setting, Hayashi's paper has three 
drawbacks. First, the characterization given in~\cite[Thm 1.7]{Hayashi} is only applicable when irreducible modules 
are parametrized by their highest weights with respect to \emph{special} choices of a Borel in the quasi-reductive supergroup $\mathbb G$.  When the base field is $\R$, the condition on the Borel  $\mathbb B$ is that
$\tau(\mathbb B)$ and $\mathbb B$ are conjugate under the Weyl group of the even part, where $\tau$ denotes the anti-holomorphic involution of $\mathbb G$ that determines its real form. However, the majority of Borels $\mathbb B$ do not satisfy this property.   For example, for the unitary supergroups $U(p,q|r,s)$, the Borel subalgebra has to be of a special  palindromic form given in~\cite[Example 5.25]{Hayashi}; see also  Remark~\ref{rmk:specialBor} below.  
Second, the results of~\cite{Hayashi} only apply to integrable modules  of real simple Lie superalgebras, i.e., those modules that can be realized as representations of the corresponding  algebraic supergroup. However, finite dimensional modules of $\g{sl}(m|n)$ are generically  \emph{not} integrable to the supergroup $\mathrm{SL}(m|n)$.  Third, Hayashi's paper does not address the case of Lie superalgebras of Cartan type.

In the present paper, we give a concrete characterization of irreducible finite dimensional modules of 
a collection  of real Lie superalgebras 
that includes  real simple Lie superalgebras and their classical variants (see Definition~\ref{dfn:classM} for a precise description). 
Furthermore,  
our characterization is explicit for any parametrization of highest weights by an \emph{arbitrary} Borel subalgebra. As such, our paper addresses the three drawbacks of~\cite{Hayashi}.

\subsection{Main results}
 Let $\g g$ be a  complex Lie superalgebra, and let  $\tau:\g g\to\g g$ be an antilinear involution of $\g g$ corresponding to a real form $\g g^\R$ of $\g g$ (throughout this paper, all Lie superalgebras are finite dimensional). Let $\catR_{\g g,\tau}$ denote the category of finite dimensional 
$\g g^\R$-modules (over real vector spaces). Then $\catR_{\g g,\tau}$ is an $\R$-linear abelian category, equipped with an $\R$-linear involutive endofunctor $\Pi$, namely the parity change.  Using $\Pi$, we can define enriched endomorphism superalgebras 
\[
\wEnd_{\catR_{\g g,\tau}}(V):=
\Hom_{\catR_{\g g,\tau}}(V,V)\oplus 
\Hom_{\catR_{\g g,\tau}}(V,\Pi V),
\]
for any object $V\in\Obj(\catR_{\g g,\tau})$; see Remark~\ref{rmk-superhoms} for the connection to superized endomorphism algebras.
If $V$ is irreducible, then $\wEnd_{\catR_{\g g,\tau}}(V)$ is a division superalgebra over $\R$. Up to isomorphism, there are 10 such division superalgebras (sometimes physicists call this fact the \emph{10-fold way}). 
By an abstract construction from category theory, in Section~\ref{sec:cat-gener} we associate to  $\catR_{\g g,\tau}$  a $\C$-linear abelian category whose objects are those of $\catR_{\g g,\tau}$ that possess a complex structure. 
Not surprisingly, the resulting category is equivalent to the category $\catC_{\g g}$ of finite dimensional $\g g$-modules (over complex vector spaces). The importance of $\catC_{\g g}$ is that it is automatically equipped with a complex conjugation functor $\funcB_{\g g,\tau}$. 

Let  
$\Irr(\catX)$ denote the set of isomorphism classes of simple objects of a given small category $\catX$. We write $[M]$ for the  isomorphism class of any $M\in\Obj(\catX)$. Given $[W]\in \Irr(\catC_{\g g})$, by restriction of scalars we can consider $W$ as a $\g g^\R$-module. This restriction either remains irreducible, or is a direct sum of two isomorphic irreducible $\g g^\R$-modules.  Denoting an irreducible $\g g^\R$-summand of $W$ by $V$, we obtain a map \[
\Phi:\Irr(\catC_{\g g})\to\Irr(\catR_{\g g,\tau}),
\] given by $\Phi([W])=[V]$; see~\eqref{eq:themapPhi}.
In 
Theorem~\ref{thm:Irr=Irr}, we show that $\Phi$ induces a bijection from $\funcB_{\g g,\tau}$-orbits in $\Irr(\catC_{\g g})$ onto  $\Irr(\catR_{\g g,\tau})$. 
Furthermore, in Theorem~\ref{thm:Table} we show that we can decide between $W=V$ and $W=V\oplus V$ 
if we can compute the \emph{endotype} of $W$, defined as 
$\EType_{\g g,\tau}(W):=\wEnd_{\catR_{\g g,\tau}}(V)$. 
Theorems~\ref{thm:Irr=Irr} and~\ref{thm:Table} reduce the classification of irreducible real representations of $\g g^\R$ to the computation of the endotypes $\EType_{\g g,\tau}(W)$; see 
Remark~\ref{rmk:Endo-to-Class}. We solve the  problem of computing 
$\EType_{\g g,\tau}(W)$
 for $(\g g,\tau)$ corresponding to real Lie superalgebras  in  the collection $\cM$ described in Definition~\ref{dfn:classM}. This is accomplished in two steps: first, in  Sections~\ref{sec:reduction-E0} and~\ref{sec:realified} we reduce the problem to the case of real forms of Lie superalgebras of Definition~\ref{dfn:Liestypes}. Next, explicit results for the Lie superalgebras of 
Definition~\ref{dfn:Liestypes} 
  are given in Theorems~\ref{thm-endotypes-basic},~\ref{thm-endotypes-Q(n)}, and~\ref{cor:Etype}. 
  Finally, in Section~\ref{sec-examples} we give explicit examples that exhibit how the latter theorems can be used to compute endotypes. 

We now define the collections of Lie superalgebras mentioned in the previous paragraph.

\begin{dfn}
\label{dfn:Liestypes}
We say a Lie superalgebra $\g g$ is of \emph{basic type} if it is isomorphic to one of the following Lie superalgebras:
\begin{itemize}
\item[(i)] A simple Lie algebra. 
\item[(ii)] One of the Lie superalgebras $\g{sl}(m|n)$, $\g{gl}(m|n)$, $\g{psl}(m|n)$, or $\g{osp}(m|2n)$  for $m,n\geq 1$.
\item[(iii)] An exceptional Lie superalgebra $G(2|1)$, $F(3|1)$, or $D(2|1,\alpha)$. 
\end{itemize}
We say $\g g$ is of \emph{type $\mathbf Q(n)$} if it is isomorphic to one of $\g q(n)$, $\g{pq}(n)$, $\g{sq}(n)$, or $\g{psq}(n)$ for $n\geq 1$. We say $\g g$ is of \emph{type $\mathbf P(n)$} if it is isomorphic to  $\g{pe}(n)$ or $\g{spe}(n)$ for $n\geq 1$. Finally, we say $\g g$ is of \emph{Cartan type} if it is isomorphic to one of the Lie superalgebras $\mathbf W(n)$, $\mathbf H(n)$, $\mathbf S(n)$ or $\tilde{\mathbf{S}}(2n)$ for $n\geq 1$. 
\end{dfn}
Given a complex Lie superalgebra $\g g$, we denote the intersection of kernels of irreducible finite dimensional $\g g$-modules by $\g f_{\g g}$. For a real Lie superalgebra $\g g^\R$, we define $\g f_{\g g^\R}$ analogously.  We now give the definition of the collection $\cM$ of Lie superalgebras whose irreducible real modules are classified in this paper.
\begin{dfn}
\label{dfn:classM} Let $\cM$ be the smallest collection of finite dimensional real Lie superalgebras that is closed under taking direct sums and satisfies  the following properties.
\begin{itemize}
\item[(i)] $\cM$ contains real forms of the four types of Lie superalgebras of Definition~\ref{dfn:Liestypes}, every real Lie algebra, 
and every complex Lie superalgebra (considered as a real Lie superalgebra via restriction of scalars).

\item[(ii)] If $\g g^\R$ is a  real Lie superalgebra
such that either   $\g g^\R/\g f_{\g g^\R}\in\cM$, or 
 $\g g^\R$  has a hyperbolic type triangular decomposition $\g g^\R=\oline{\g u}^\R\oplus\g l^\R\oplus \g u^\R$  (see Subsection~\ref{subsec:hyperb}) such that $\g l^\R\in\cM$, then $\g g^\R\in\cM$.

\end{itemize}\end{dfn}

 When $\g g$ is either of basic type or of type $\mathbf Q(n)$, our method for the  computation of $\EType_{\g g,\tau}(W)$  from the highest weight of $W$ 
 (which is carried out in  Theorems~\ref{thm-endotypes-basic} and~\ref{thm-endotypes-Q(n)}) is quite different from the one used in~\cite{Hayashi}: rather than using tools from Galois cohomology in the spirit of~\cite{Tits}, we use the Harish-Chandra projection of products of root vectors of a sequence of isotropic roots associated to the highest weight. 

For Lie superalgebras of basic type or of type $\mathbf Q(n)$, 
in Section~\ref{sec:tau-compatible-parab} we establish a slightly different method for computing $\EType_{\g g,\tau}(W)$ that relies on the  evaluation of the highest weight  $\lambda$ of $W$ on the sum of coroots associated to Kostant's cascade of strongly orthogonal roots in a Levi subalgebra of $\g g_\eev$; see Theorem~\ref{thm:alternateclam}. 
Even in the special case of real reductive Lie algebras, this viewpoint leads to a classification of irreducible modules that (in our opinion) is substantially simpler than the one given in~\cite{Onishchik}; see Theorem~\ref{thm-reductivespecial}. 
In~\cite{Onishchik}, the characterization of irreducible modules of real simple Lie algebras reduces to the calculation of an  invariant of complex modules known as the Karpelevich index. However, the arguments to arrive at this result are rather intricate. Furthermore, the conceptual connection with Kostant's cascade has, to our best knowledge, remained unnoticed in~\cite{Onishchik}. 

Some of the proofs of the assertions in Sections~\ref{sec:E-types-module} and~\ref{sec:realified} require routine but probably  tedious sign calculations. The reader interested in these calculations can find them in the ``\texttt{details}'' environments in the source \LaTeX\ file.

\subsection*{Acknowledgement} Some of the ideas of this work were conceived during the series of AiM SQuaRE meetings on \emph{Symmetric spaces and Capelli operators for Lie superalgebras}. The authors thank the 
American Institute of Mathematics for sponsoring these SQuaRE meetings.  
The research of S.S. was partially supported by Simons Foundation Grant 00006698. 
The research of H.S. was partially supported by  Discovery Grant 
RGPIN-2024-04030 from the Natural Sciences and Engineering Research Council of
Canada.
The research of V.S. was partially supported by NSF Grant 2001191.

\section{Categorical generalities}

\label{sec:cat-gener}
The goal of this section is to lay out some abstract categorical foundations that will be used later in the setting of  module categories. Let $\catR$ be an essentially small $\R$-linear abelian category with the property  that $\End_{\catR}(V)$ is a finite dimensional real vector space for every simple object $V\in\mathrm{Obj}(\catR)$. 
 Consider the category $\catC$  defined as follows: 
\begin{itemize}
\item[(i)] Objects of $\catC$ are pairs $(V,\iota)$ where $\iota\in\Hom_{\catR}(V,V)$ and $\iota^2=-1$. 
\item[(ii)] For objects $W_1:=(V_1,\iota_1)$ and $W_2:=(V_2,\iota_2)$ of $\catC$ we set
\[
\Hom_{\catC}(W_1,W_2):=\{\phi\in\Hom_{\catR}(V_1,V_2)\,:\,\phi\iota_1=\iota_2\phi\}.
\] 
\end{itemize}
We will need the following three functors that relate the categories $\catR$ and $\catC$.
\begin{itemize}

\item[(i)]$\funcB:\catC\to\catC$, where $\funcB(V,\iota):=(V,-\iota)$ and  $\funcB(\alpha):=\alpha$ for every $\catC$-morphism $\alpha$.

\item[(ii)]  $\funcE:\catR\to\catC$, defined by 
\[
\funcE(V):=(V\oplus V,-\epsilon_{1}\pi_{2}+ \epsilon_{2}\pi_{1})
,\]
where the morphisms $\pi_i=\pi_{V,i}:V\oplus V\to V$ 
and 
 $\epsilon_i=\epsilon_{V,i}:V\to V\oplus V$ for $i=1,2$ are 
the canonical projections and embeddings, and  \[
\funcE(\phi):=\epsilon'_{1}\phi\pi_{1}^{}+ \epsilon'_{2}\phi\pi_{2}^{}
\quad\text{ for }\phi\in\Hom_{\catR}(V,V'),
\]  
where $\epsilon_i':=\epsilon_{V',i}^{}$ for $i=1,2$.  

\item[(iii)] $\funcF:\catC\to\catR$, where $\funcF$ maps $(V,\iota)$ to $V$ and  $\funcF(\alpha):=\alpha$ for every $\catC$-morphism $\alpha$.

\end{itemize}
We denote the identity functors on $\catR$ and $\catC$ by $\funcI_\catR$ and $\funcI_\catC$, respectively.

The category $\catC$ is a $\mathbb C$-linear abelian category,
with zero object $\mathbf{0}_{\catC}:=\funcE(\mathbf{0}_{\catR})$
and the biproduct of   $W_1:=(V_1,\iota_1)$ and $W_2:=(V_2,\iota_2)$ defined to be $(V_1\oplus V_2,\iota_1\oplus \iota_2)$.
This can be verified directly 
or using the 
Freyd-Mitchell embedding theorem:
if we realize $\catR$ as a  category of $R$-modules for a ring $R$, then  $\catC$ is  a  category of $S$-modules where $S:=R[x]/\lag x^2+1\rag$.

The functor $\funcB$ induces an involution $\Theta_\funcB$ on $\Irr(\catC)$ because $\funcB^2=\funcI_\catC$ (see Remark~\ref{rmk:WsimpleBWsimple}). We denote the set of $\Theta_\funcB$-orbits by $\Irr(\catC)/\funcB$. 
We now define three subsets of $\Irr(\catC)$:
\begin{itemize}
\item $\Irr_+^{}(\catC):=\left\{[W]\in\Irr(\catC)
\,:\,\funcF(W)\cong V\oplus V\text{ for some }V\in\Irr(\catR) \right\}$,
\item
$\Irr_\circ^{}(\catC):=\left\{[W]\in\Irr(\catC)\,:\,\funcF(W)\text{ is simple and }W\not\cong\funcB(W)\right\}$,
\item
$\Irr_-^{}(\catC):=\left\{[W]\in\Irr(\catC)\,:\,\funcF(W)\text{ is simple and }W\cong\funcB(W)\right\}$.

\end{itemize}
By Schur's Lemma in $\catR$, we have $\End_\catR(V)\in\{\R,\C,\qH\}$ whenever $[V]\in \Irr(\catR)$. We define three subsets of $\Irr(\catR)$ 
as follows: 
\begin{itemize}
\item $\Irr_+^{}(\catR):=\left\{
[V]\in\Irr(\catR)\,:\,
\End_\catR(V)\cong\R\right\}$.

\item $\Irr_\circ^{}(\catR):=\left\{
[V]\in\Irr(\catR)\,:\,
\End_\catR(V)
\cong \C\right\}$.

\item $\Irr_-^{}(\catR):=\left\{
[V]\in\Irr(\catR)\,:\,
\End_\catR(V)\cong \qH\right\}$.

\end{itemize}
By Lemma~\ref{lem:F2}, for any $[W]\in\Irr(\catC)$ either $\funcF(W)$ is simple or $\funcF(W)\cong V\oplus V$ for a simple $V\in\Obj(\catR)$ where $[W]$ uniquely determines $[V]$. 
Thus, we have a map
\begin{equation}
\label{eq:themapPhi}
\Phi:\Irr(\catC)\to\Irr(\catR)
\quad,\quad
\Phi([W]):=
\begin{cases}
[\funcF(W)]&\text{ if }\funcF(W)\text{ is simple},\\
[V]&\text{ if }\funcF(W)\cong V\oplus V.
\end{cases}
\end{equation}
We denote the set of $\Theta_\funcB$-orbits of $\Irr(\catC)$ by $\Irr(\catC)/\funcB$.
Then $\Phi([\funcB(W)])=\Phi([W])$ because
$\funcF\funcB=\funcF$. Thus,
 $\Phi$ 
 induces a map \[
 \bar \Phi:\Irr(\catC)/\funcB\to \Irr(\catR).
 \] 
The following theorem will be proved in Appendix~\ref{sec:proofsthms}.
\begin{thm}
\label{thm:Irr=Irr}
The following statements hold.
\begin{itemize}
\item[\rm (i)]
The sets $\Irr_\star(\catR)$, for $\star\in\{-,\circ,+\}$, form a partition of $\Irr(\catR)$.
\item[\rm (ii)] The  
sets $\Irr_\star(\catC)$, for $\star\in\{-,\circ,+\}$, form a partition of $\Irr(\catC)$. 
\item[\rm (iii)] The sets 
$\Irr_\star(\catC)$ are $\Theta_\funcB$-stable. Furthermore, the elements of $\Irr_+(\catC)$ are $\Theta_\funcB$-fixed. 
\item[\rm (iv)]
$\Phi(\Irr_\star(\catC))=\Irr_\star(\catR)$ for $\star\in\{+,\circ,-\}$. Furthermore,  the  map $\bar\Phi$ is a  bijection from $\Irr(\catC)/\funcB$ onto $\Irr(\catR)$. 
\end{itemize}
\end{thm}

\subsection{Enriched endomorphism algebras in $\catR$ and $\catC$}

In the rest of this section we assume that $\catR$ is equipped with an $\R$-linear  endofunctor $\Pi$ such that $\Pi^2=\funcI_\catR$. 
From naturality of the family of isomorphisms 
\begin{equation}
\label{eq:HomVPiV'}
\Hom_\catR(V,\Pi (V'))\cong \Hom_\catR(\Pi (V),V')\quad,\quad \phi\mapsto\Pi(\phi)
\end{equation}
it follows that $\Pi$ is left and right adjoint to itself and hence it is also exact.  
For $V,V'\in\Obj(\catR)$ we define a $\Z_2$-graded space of 
enriched \footnote{The category with the same objects as $\catR$ but with morphisms defined as in~\eqref{eq:enrHom} is enriched over the category of super vector spaces. 
} morphisms
\begin{equation}
\label{eq:enrHom}
\wHom_\catR(V,V'):=\Hom_\catR(V,V')\oplus \Hom_\catR(V,\Pi(V')).
\end{equation}
The isomorphism~\eqref{eq:HomVPiV'} allows us to define composition of enriched morphisms.
%
%
%
%
%
%
%
%
%
We set \[
\wEnd_\catR(V):=\wHom_\catR(V,V),
\] and note that $\wEnd_\catR(V)$ is an $\R$-superalgebra. 

The functor $\Pi$ induces a functor on $\catC$ which  we also denote by $\Pi$.  Clearly $\Pi\funcB=\funcB\Pi$.
 We  define $\wHom_\catC(W,W')$ and $\wEnd_\catC(W,W')$ for $W,W'\in\Obj(\catC)$ analogously. 
 
Recall that an associative superalgebra $A:=A_{\eev}\oplus A_{\ood}$ is called a \emph{division superalgebra} if its nonzero homogeneous  elements are invertible. By Schur's Lemma,  
if $V\in\Obj(\catR)$ is simple then $\wEnd_\catR(V)$ is a division superalgebra over $\R$. Similarly, if $W\in\Obj(\catC)$ is simple then  $\wEnd_\catC(W)$
is a division superalgebra over $\C$.

Let $\phi\in\wHom_\catC(W,\funcB (W))$ be  homogeneous. Since  $\funcB(\phi)=\phi$ and $\funcB^2(W)=W$, by a slight abuse of notation we can write 
$\phi^2$ instead of $\funcB(\phi)\phi$ and consider it  as an element of  $\End_\catC(W)$. In particular, if $W$ is simple then $\phi^2=\bfc(\phi) 1_{W}$ for a scalar $\bfc(\phi)\in\C$. 
\begin{lem}
\label{lem:sig}
Let $W\in\Obj(\catC)$ be simple.
\begin{itemize}
\item[\rm (i)]
Suppose that $\phi\in\wHom_\catC(W,\funcB(W))$ is homogeneous. Then $\bfc(\phi)\in\R$.
\item[\rm(ii)]
For nonzero $\phi,\phi'\in\wHom_\catC(W,\funcB(W))_i$ where $i\in\{\eev,\ood\}$, we have $\bfc(\phi)\bfc(\phi')>0$.  
\end{itemize}
\end{lem}
\begin{proof}
(i) Follows from $
\bfc(\phi)\phi=(\phi^2)\phi=
\phi^3=\phi(\phi^2)=\phi \circ(\bfc(\phi)1_W))=\oline{\bfc(\phi)}\phi$.

(ii) By Schur's Lemma $\phi'=z\phi$ for some $z\in\C^*$. The assertion follows from \[
\bfc(\phi')1_W=(\phi')^2=(z\phi)(z\phi)=z\bar z\phi^2=|z|^2\bfc(\phi)1_W.\qedhere
\] 
\end{proof}

\begin{dfn}
Let $W\in\Obj(\catC)$ be simple. \begin{itemize}
\item[(i)] Let $\phi\in\Hom_\catC(W,W')$ where $W':=\funcB(W)$ or $W':=\Pi\funcB(W)$.  We define \[
\sgn(\phi):=\begin{cases}
+&\text{ if }\bfc(\phi)>0,\\
-&\text{ if }\bfc(\phi)<0.
\end{cases}
\] 
\item[(ii)] The symmetry datum of $W$ is the set \[
\mathfrak S_W\sseq \left\{\Pi,(\funcB,+),(\funcB,-),(\Pi\funcB,+),(\Pi\funcB,-)\right\},
\] defined by  the following properties:
\begin{itemize}
\item[(ii-a)] $\Pi\in \mathfrak S_W$ if and only if $W\cong \Pi(W)$.
\item[(ii-b)] $(\funcB,\pm)\in \mathfrak S_W$ if and only if $W\cong \funcB(W)$ and the isomorphism has sign $\pm$. 
\item[(ii-c)]
$(\Pi\funcB,\pm)\in \mathfrak S_W$ if and only if $W\cong \Pi\funcB(W)$ and the isomorphism has sign $\pm$.
  
\end{itemize} 
\end{itemize}
\end{dfn}

\subsection{The Brauer--Wall monoid}
Isomorphism classes of finite dimensional central division superalgebras over a field $\F$ form a group $\BRm_\F$ with the binary operation uniquely determined by the property that  $[D]\cdot [D']=[D'']$ if and only if $D\otimes_\F D'$ is Morita equivalent to $D''$. 
The disjoint union 
$\BRm:=\BRm_\R\cup\BRm_\C$ naturally inherits a monoid structure that extends the group structures of $\BRm_\R$ and $\BRm_\C$: the product of $[D]\in\BRm_\R$ and $[D']\in\BRm_\C$ is defined to be $[D\otimes_\C (\C\otimes_\R D')]$.  This monoid is called  the \emph{Brauer--Wall monoid} of the field $\R$ (see~\cite{Wall} for a theory  over general fields). It  has 10 elements: 2 division superalgebras with supercenter $\mathbb C$ that form a group isomorphic to $\Z_2$, and 8 division superalgebras with supercenter $\R$ that form a group isomorphic to $\Z_8$. 
In what follows,  we denote the former 2 division superalgebras by $k_\C$ for $0\leq k\leq 1$ and the  latter 8 by $k_\R$ for $0\leq k\leq 7$. In this notation, $k_\F$ is Morita equivalent to the Clifford $\F$-algebra on $k$ square roots of unity.

\begin{dfn}
The endotype of $[V]\in\Irr(\catR)$ is the element of $\BRm$  that corresponds to $\wEnd_{\catR}(V)$. The endotype of $[W]\in\Irr(\catC)$ is defined to be the endotype of $[\Phi(W)]$. We denote the endotype of $[W]$ by $\EType_\catR(W)$. 
\end{dfn} 

The following theorem, which will be proved in Appendix~\ref{sec:proofsthms}, reduces the characterization of endotypes in  $\Irr(\catR)$ to the computation of the symmetry datum 
in $\Irr(\catC)$. 
\begin{thm}
\label{thm:Table}
Let $[V]\in\Irr(\catR)$ and 
let $W$  be an irreducible summand of $\funcE(V)$, so that  $\Phi([W])=[V]$. Then
the endotype of $V$  is uniquely determined by the symmetry datum of $W$, according to Table~\ref{Table-1}. 
\begin{table}[ht]
\begin{tabular}{|c|c|c|c|c|c|c|c|}
\hline
& $\wEnd_\catR(V)$ & $\wEnd_\catC(\funcE(V))$ & 
$\funcE(V)$ & 
Symmetry datum of $W$
\\
\hline
$0_\C$ 
& $\C$ 
& $\C\oplus\C$ 
& 
$W\oplus \funcB(W)$
&  
\\
\hline
$1_\C$ 
&   
$\mathrm Q(1)\cong \C\oplus \C \iota$, $\iota^2=1$, $\iota i=i\iota$
&
$\mathrm Q(1)\oplus \mathrm Q(1)$
&
$W\oplus \funcB(W)$
&
$\Pi$
\\

\hline
$0_\R$
&
$\R$ 
&
$\C$ 
&
$W$

& 

$(\funcB,+)$
\\
\hline 
$1_\R$ &
$\R\oplus \R\iota$, $\iota^2=-1$
&
$\mathrm Q(1)$
&
$W$
&
$\Pi$, $(\funcB,+)$,  
$(\Pi\funcB,-)$\\
\hline
$2_\R$ &
$\C\oplus \C\iota$, $\iota^2=-1$, $\iota i=-i\iota$ & 
$\Mat_{1|1}(\C)$ & $W\oplus \funcB(W)$
& 
{
$(\Pi\funcB,-)$
}

\\
\hline
$3_\R$ & $\qH\oplus \qH\iota$, $\iota^2=1$, $\iota i=i\iota$, etc.
& $\Mat_{2}(\mathrm Q(1))$ & 
$W\oplus \funcB(W)$ & 
{
$\Pi$, $(\funcB,-)$, 
$(\Pi\funcB,-)$}\\
\hline
$4_\R$ & $\qH$ & $\Mat_{2}(\C)$ & 
$W\oplus \funcB(W)$ & 
{
$(\funcB,-)$
} 
\\
\hline $5_\R$ & $\qH\oplus \qH\iota$, $\iota^2=-1$, $\iota i=i\iota$, etc.
& 
$\Mat_{2}(\mathrm Q(1))$ & $W\oplus \funcB(W)$ & 
{
$\Pi$, $(\funcB,-)$, $(\Pi\funcB,+)$}
\\

\hline $6_\R$ & $\C\oplus\C\iota$, $\iota^2=1$, $\iota i=-i\iota$ & 
{$\Mat_{1|1}(\C)$} & $W\oplus \funcB(W)$ & 
{$(\Pi \funcB,+)$} \\

\hline $7_\R$ & $\R\oplus\R\iota$, $\iota^2=1$ & $\mathrm Q(1)$ & $W$ & $\Pi$, $(\funcB,+)$, 
$(\Pi\funcB,+)$
\\
\hline
\end{tabular}
\vspace{3mm}

\caption{\label{Table-1} Enriched endomorphism algebra vs. symmetry datum}
\vspace{-8mm}
\end{table}
\end{thm}

\begin{rmk}
Let us describe the content of Table~\ref{Table-1} more explicitly. The $1^\mathrm{st}$  column represents an element  of
the Brauer--Wall monoid.  The $2^\mathrm{nd}$ column describes the explicit realization of $\wEnd_\catR(V)$ for a simple $V\in\Obj(\catR)$. The   
complexification of $\wEnd_\catR(V)$ is given in the $3^\mathrm{rd}$ column. The $4^\mathrm{th}$ column
indicates if $\funcE(V)$ is simple or not, according to  Lemma~\ref{lem:E2}; thus in every row  $[W]\in\Irr(\catC)$. 
Finally, the $5^\mathrm{th}$ column provides the symmetry datum of $W$.  \end{rmk}

\begin{rmk}
\label{rmk:End(W)}
From Table~\ref{Table-1} it follows that in rows $k_\R$ or $k_\C$ with $k$ even we have $\wEnd_\catC(W)\cong \C$, and in other rows we have  
$\wEnd_\catC(W)\cong \mathrm{Q}(1)$.
\end{rmk}

\begin{rmk}
\label{rmk:endoDM2D}
Let  
$[W]\in\Irr(\catC)$. 
By Lemma~\ref{lem:F2} we have  $\wEnd_{\catR}(\funcF(W))\cong \mathbb D$ or $\Mat_2(\mathbb D)$, where  
$\mathbb D:=\EType_\catR(W)$. 
\end{rmk}

\section{Endotypes for tensor products of modules}
\label{sec:E-types-module}
In this section we investigate endotypes of  irreducible modules of  tensor products of  associative superalgebras. We begin by introducing some notation that will be used in the rest of the paper.

Throughout the paper, we denote the parity of a vector $v$ in a superspace by $|v|$.
For a (Lie or associative) $\F$-superalgebra $\cA=\cA_\eev\oplus\cA_\ood$, we use $\Rep(\mathcal A)$ (respectively, $\Rep_\mathrm{fin}(\mathcal A)$) to denote the category of $\cA$-modules on $\F$-vector superspaces (respectively, finite dimensional $\F$-vector superspaces). The morphisms in these categories are assumed to respect the $\mathbb Z_2$-grading.

\begin{rmk}
\label{rmk-superhoms}
For $W,W'\in\Obj(\Rep_\mathrm{fin}(\mathcal A))$ we define the super endomorphism spaces \[
\Hom_\cA^s(W,W'):=\{T:W\to W'\,:\,T\text{ is linear and }Ta=(-1)^{|T|\cdot |a|}aT\text{ for }a\in\cA\},
\]
where, as usual, $|T|,|a|$, etc. denote parity. 
As in~\eqref{eq:enrHom}, 
one can also define 
\[
\Hom_\cA^s(W,W'):=\Hom_{\Rep_\mathrm{fin}(\mathcal A)}(W,W')\oplus \Hom_{\Rep_\mathrm{fin}(\mathcal A)}(W,\Pi{^\sigma W'}),
\] where $^\sigma W'$ is the twist of $W'$ by the automorphism $\sigma$ of $\cA$ given by $\sigma(a):=(-1)^{|a|}a$. There is an isomorphism $\Hom_\cA^s(W,W')\to \wHom_{{\Rep_\mathrm{fin}(\mathcal A)}}(W,W')$ given by $T\mapsto T^\diamond$ where 
$T^\diamond w:=(-1)^{|T|\cdot |w|} Tw$ for $w\in W$. 
Setting $\End_\cA^s(W):=\Hom_\cA^s(W,W)$, the latter map induces an isomorphism of superalgebras
\begin{equation}
\label{eq:soops}
\End_\cA^s(W)\cong ((\wEnd_{\Rep_\mathrm{fin}(\mathcal A)}(W))^\mathrm{op})^\mathrm{sop},
\end{equation} 
where $(X^\mathrm{op})^\mathrm{sop}$ denotes the  super opposite of the opposite of a superalgebra $X$.  
\end{rmk}

Let
$\cA$ be a (Lie or associative) $\C$-superalgebra and let $\cA^\R$ be a real form of $A$, obtained as fixed points of an antilinear involution $\tau:\cA\to \cA$. 
Henceforth we set \[
\catR:=\catR_{\cA,\tau}:=\Rep_\mathrm{fin}(\cA^\R).
\]  
The category $\catC$ associated to $\catR$ as in Section~\ref{sec:cat-gener} is $\catC_\cA:=\Rep_\mathrm{fin}(\cA)$. 
The yoga of functors 
in Section~\ref{sec:cat-gener} applied to these choices of categories $\catR$ and $\catC$ results in functors
$\funcB_{\cA,\tau},\funcE_{\cA,\tau},\funcF_{\cA,\tau}$. When the choice of $\cA$ is clear, we simplify our notation to 
$\funcB_{\tau},\funcE_{\tau},\funcF_{\tau}$. For simplicity we define
\begin{equation}
\label{eq:EtypetauA}
\EType_{\tau}(W):=
\EType_{\cA,\tau}(W)
:=\EType_{\catR_{\cA,\tau}}(W)\quad\text{ for }W\in\Obj(\catC_\cA). 
\end{equation}

\begin{rmk}
\label{rmk:Abar=}Given $[W]\in\Irr(\catC_A)$,
if $\oline{\cA}$ denotes the image of $\cA$ in $\End_{\cA}^s(W)$ then by the classification of primitive $\C$-superalgebras we have $\oline{\cA}= \End_\mathbb D^s(W)$ where $\mathbb D:=\End_{\cA}(W)\cong\C$ or $\mathrm Q(1)$. (In particular, it happens that  $\wEnd_{\catC_\cA}(W)\cong\End_\cA^s(W)$.) 
We equip $W$ with a right $\mathbb D$-module structure with the action $w\cdot d:=(-1)^{|d|\cdot |w|}d\cdot w$ for $d\in\mathbb D$, $w\in W$. 
\end{rmk}

In the rest of this section we assume that  $A$ and $B$ are associative $\C$-superalgebras. Given
$[W_A]\in \Irr(\catC_A)$ and $[W_B]\in\Irr(\catC_B)$, we set
 $\mathbb D_A:=\End_{A}^s(W_A)$ and $\mathbb D_B:=\End_{B}^s(W_B)$. From the general rules about external tensor products of irreducible modules of types $\mathtt M$ and $\mathtt Q$ (for example see~\cite[Sec. 3.1.3]{ChengWang2012}), it follows that $[W_A\otimes_{\mathbb D} W_B]\in\Irr(\catC_{A\otimes B})$, where $\mathbb D:=\mathbb D_{W_A,W_B}$ is defined as follows: $\mathbb D=\C$ if 
either $\mathbb D_A\cong \C$ or $\mathbb D_B\cong \C$, and $\mathbb D=\mathrm{Q}(1)$ otherwise.  
In what follows, tensor products over unspecified rings are always construed to be over $\C$. 

\begin{prp}
\label{prp:W=WAWB}
Let $[W]\in\Irr(\catC_{A\otimes B})$. 
Then 
there exist $[W_A]\in\Irr(\catC_A)$ and $[W_B]\in\Irr(\catC_B)$  such that $W\cong W_A\otimes_\mathbb D W_B$, where $\mathbb D=\mathbb D_{W_A,W_B}$.  
\end{prp}

\begin{proof}
For any irreducible $B$-submodule $W_B\sseq W$, the canonical $A\otimes B$-module map 
\[\Hom^s_B(W_B,W)\otimes W_B\to W\ ,\ T\otimes w\mapsto Tw\] 
is a surjection. Thus, as a $B$-module $W$ is a direct sum of copies of $W_B$ and $\Pi W_B$. 
Similarly, if $W_A\sseq \wHom_B(W_B,W)$ is an irreducible $A$-submodule, then the restriction map $W_A\otimes W_B \to W$ is a surjection, so that $W$ as $A$-module is a direct sum of copies of $W_A$ and $\Pi W_A$.

\begin{details}
$A$-module structure on $\Hom^s_B(W_B,W)$ is:
\[
(a\cdot T)v:=(a\otimes 1)(Tv).
\]
By our assumption, $W_B$ is a left $\mathbb D$-module:  we have $d(bw)=(-1)^{bd}b(dw)$ for $b\in B$, $d\in \mathbb D$, $w\in W_B$.

The left $\mathbb D$-module structure on $\Hom^s_B(W_B,W)$ is: 
\[
(d\cdot T)w=(-1)^{dT}T(dw)\quad\text{ for }w\in W_B.
\]
This action is well-defined because
\[
(d\cdot T)(bw)=(-1)^{dT}T(dbw)=(-1)^{dT+bd}T(bdw)=(-1)^{dT+bd+bT}(1\otimes b)T(dw)=(-1)^{b(d+T)}(1\otimes b)(d\cdot T)(w)
\]
The corresponding right $\mathbb D$-module structure on $\Hom^s_B(W_B,W)$ is given by $(T\cdot d)(w):=T(dw)$.

The left action of $\mathbb D$ on $\Hom^s_B(W_B,W)$ supercommutes with the $A$-module structure because:
\[
(d\cdot (a\cdot T))v=(-1)^{d(a+T)}(a\cdot T)(dv)=
(-1)^{d(a+T)}(a\otimes 1)(T(dv)),\]
and
\[
a\cdot (d\cdot T)v=(a\otimes 1)((d\cdot T)(v))=(a\otimes 1)(-1)^{dT}T(dv),
\]
hence $d\cdot (a\cdot T)=(-1)^{ad}a\cdot (d\cdot T)$.

We also have 
\[
a\cdot (T\cdot d)=(a\cdot T)\cdot d=(a\otimes 1)(T(dv)).
\]
The map $\Hom^s_B(W_B,W)\otimes W_B\to W$, $T\otimes w\mapsto Tw$ is an $A\otimes B$-module homomorphism:
\begin{align*}
(a\otimes b)(T\otimes w)&=(-1)^{bT}a\cdot T\otimes bw\mapsto (-1)^{bT}(a\cdot T)(bw)\\
&=(-1)^{bT}(a\otimes 1)(T(bw))=(a\otimes 1)(1\otimes b)Tw=(a\otimes b)(Tw).
\end{align*}
The map $\Hom^s_B(W_B,W)\otimes W_B\to W$ factors to a map
$\Hom^s_B(W_B,W)\otimes_\mathbb D W_B\to W$:
\[
Td\otimes w\mapsto (T\cdot d)w=T(dw)\quad{ and }
T\otimes (dw)\mapsto T(dw).
\]
Furthermore, the action factors through $\Hom^s_B(W_B,W)\otimes_{\mathbb D} W_B\to W$. This follows from:
\[
(a\otimes b)(T\cdot  d\otimes w)=(-1)^{b(T+d)}a\cdot (T\cdot d)\otimes bw,
\]
and 
\begin{align*}
(a\otimes b)(T\otimes dw)&=(-1)^{bT}a\cdot T\otimes b(dw)=
(-1)^{bT+bd}a\cdot T \otimes d(bw)\\&
=
(-1)^{bT+bd}(a\cdot T)\cdot d \otimes bw
=
(-1)^{bT+bd}a\cdot (T\cdot d) \otimes bw.
\end{align*}
\end{details}

\begin{details}
If $W_A$ is an irreducible module with $\End_\mathbb D^s(W_A)=\mathbb D$, then the $A\otimes B$-action is consistent with the tensor product $W_A\otimes_\mathbb D W_B$:
\[
(a\otimes b)(v\cdot d\otimes w)=(-1)^{b(v+d)}a\cdot (v\cdot d)\otimes bw
\]
and 
\[
(a\otimes b)(v\otimes d\cdot w)=(-1)^{bv}a\cdot v\otimes b(d\cdot w)=(-1)^{bv+bd}a\cdot v\otimes d\cdot (bw).
\]
But we have 
\[
a\cdot (v\cdot d)=(-1)^{vd}a\cdot (d\cdot v)=(-1)^{ad+vd}d\cdot (a\cdot v)=(a\cdot v)\cdot d. 
\]
This proves that the span of $v\cdot d\otimes w-v\otimes d\cdot w$ is invariant under the action of $A\otimes B$. 
\end{details}

If  $\mathbb D_A\cong\C$, then $\End_{\mathbb D_A}^s(W_A)\cong \mathrm{Mat}_{m|n}(\C)$ where $(m|n)$ is the graded dimension of $W_A$. Similarly, if $\mathbb D_A\cong\mathrm{Q}(1)$, then $\End_{\mathbb D_A}^s(W_A)\cong \mathrm{Q}(n)$ where $n=\dim (W_A)_\eev=\dim (W_A)_\ood$. From 
Remark~\ref{rmk:Abar=} and 
well-known descriptions of tensor products of superalgebras $\mathrm{Mat}_{m|n}(\C)$ and $\mathrm{Q}(n)$ it follows that if either $\mathbb D_A\cong\C$ or $\mathbb D_B\cong\C$, then $W_A\otimes W_B$ is irreducible and hence isomorphic to $W$, whereas if $\mathbb D_A\cong \mathbb D_B\cong \mathrm{Q}(1)$, then 
$W_A\otimes W_B\cong W\oplus \Pi W$. Since the map $W_A\otimes W_B\to W$ factors through the proper quotient $W_A\otimes_\mathbb D W_B$, we have $ W_A\otimes_{\mathbb D} W_B\cong W$.
\end{proof}

\begin{details}
Factoring through the proper quotient $W_A\otimes_{\mathbb D} W_B$ is inherited from the factoring of the map $\Hom^s_B(W_B,W)\otimes  W_B\to W$ to tensor product over $\mathbb D$. 

The tensor product of two modules of type $\mathsf Q$ is of the form $E\oplus E$ where $E$ is an irreducible module of type $\mathsf M$. Thus $W_A\otimes W_B=E\oplus E$, and from surjectivity of $W_A\otimes W_B\to W$ it follows that $W\cong E$. Now $W_A\otimes_\mathbb D W_B$ is a proper quotient of $W$, hence $E\cong W_A\otimes_\mathbb D W_B$. 
\end{details}

\begin{prp}
\label{prp:EtypeofWWAWB}
Let $W$, $W_A$, $W_B$, and $\mathbb D$ be as in Proposition~\ref{prp:W=WAWB} and let $A^\R$ and $B^\R$ be real forms of $A$ and $B$, corresponding to involutions $\tau_A:A\to A$ and $\tau_B:B\to B$, respectively. 
Set $C:=A\otimes B$ with real form $C^\R:=A^\R\otimes_\R B^\R$ corresponding to  $\tau_C:=\tau_A\otimes \tau_B$. 
Then $\EType_{{\tau_C}}(W)=\EType_{{\tau_A}}(W_A)\bullet\EType_{{\tau_B}}(W_B)$, where $\bullet$ denotes the product of $\BRm$. 
\end{prp}

\begin{proof}
Let $\mathbb D_A,\mathbb D_B,\mathbb D_C$ be division superalgebras corresponding to endotypes of $W_A,W_B,W_C$, respectively. Our goal is to prove that 
$[\mathbb D]=[\mathbb D_A]\bullet [\mathbb D_B]$. 
The $C^\R$-module $\funcF_{\tau_A}(W_A)\otimes_\R \funcF_{\tau_B}(W_B)$ is semisimple. 
Furthermore, $\funcF_{\tau_C}(W)$ is a quotient of $\funcF_{\tau_A}(W_A)\otimes_\R \funcF_{\tau_B}(W_B)$.

\begin{details}
$\funcF_{\tau_A}(W_A)$ is semisimple, hence the image ${\bar A}^\R$ of $A^\R$ in $\End_\R(\funcF_{\tau_A}(W_A))$ is a semisimple algebra. Same holds for $\funcF_{\tau_B}(W_B)$. It follows that  
the image of $A^\R\otimes_\R B^\R$ in $\End_\R(\funcF_{\tau_A}(W_A)\otimes_\R\funcF_{\tau_B}(W_B))\cong \End_\R(\funcF_{\tau_A}(W_A))\otimes_\R\End_\R(\funcF_{\tau_B}(W_B))$, which  is ${\bar A}^\R\otimes_\R {\bar B}^\R$, is also semisimple.
This is because over a perfect field, tensor products of semisimple superalgebras are semisimple. 
\end{details}

\begin{details}
Concretely, $\funcF_{\tau_A}(W_A)\otimes_\R \funcF_{\tau_B}(W_B)$ is $W_A\otimes_\R W_B$ whereas $\funcF_{\tau_C}(W)$ is $W_A\otimes_\mathbb D W_B$. The latter is a quotient of $W_A\otimes_\R W_B$ by definition. 
\end{details}

For any superalgebra $X$, we denote $(X^\mathrm{op})^\mathrm{sop}$ by $X^\diamond$.
By Remark~\ref{rmk:endoDM2D} and~\eqref{eq:soops}  we obtain  
\begin{align*}
\End^s_{C^\R}(\funcF_{\tau_A}(W_A)\otimes_\R \funcF_{\tau_B}(W_B))&
\cong 
\End^s_{A^\R}(\funcF_{\tau_A}(W_A))\otimes_\R^{}
\End^s_{B^\R}(\funcF_{\tau_B}(W_B))\\
&\cong
\Mat_{d_1}(\mathbb D_A^\diamond)\otimes_\R \Mat_{d_2}(\mathbb D_B^\diamond)\cong\Mat_{d_1d_2}(\mathbb D_A^\diamond\otimes_\R\mathbb D_B^\diamond),
\end{align*}
where $d_1,d_2\in\{1,2\}$. Furthermore, $\mathbb D_A^\diamond\otimes \mathbb D_B^\diamond\cong (\mathbb D_A\otimes_\R \mathbb D_B)^\diamond $. \begin{details}Note that for an irreducible module $E$ over $A^\R$ we have $\wEnd_{A^\R}(E)\cong \End_{A^\R}^s(E)^\diamond$. Also note that the map $\mathbb D\to\mathbb D^\diamond$ is an automorphism of the Wall-Brauer monoid.  
\end{details}It suffices to verify that $\mathbb D_A\otimes_\R \mathbb D_B$ is a direct product of central simple algebras that are Brauer equivalent to $[\mathbb D_A]\bullet [\mathbb D_B]$. If both $\mathbb D_A$ and $\mathbb D_B$ are $\R$-central then $\mathbb D_A\otimes_\R\mathbb D_B$ is also an $\R$-central simple superalgebra and we have $[\mathbb D_A\otimes_\R\mathbb D_B]=[\mathbb D_A]\bullet [\mathbb D_B]$. If exactly one of $\mathbb D_A$ and $\mathbb D_B$ (say $\mathbb D_A$) is a $\C$-central superalgebra, then 
\[
[\mathbb D_A\otimes_\R\mathbb D_B]= 
[\mathbb D_A\otimes_\C(\C\otimes_\R\mathbb D_B)]= [\mathbb D_A]\bullet [\mathbb D_B].
\]
Finally, if both of $\mathbb D_A$ and $\mathbb D_B$ are $\C$-central, then from $\C\otimes_\R\C\cong \C\times \C$ we obtain 
\[
\mathbb D_A\otimes_\R\mathbb D_B\cong 
\mathbb D_A\otimes_\C(\C\otimes_\R\C)\otimes_\C\mathbb D_B
\cong 
(\mathbb D_A\otimes_\C\mathbb D_B)\times 
(\mathbb D_A\otimes_\C\mathbb D_B),
\]
hence $[\mathbb D_A\otimes_\C\mathbb D_B]=[\mathbb D_A]\bullet [\mathbb D_B]$. 
\begin{details}
This last bit requires that the base field is $\mathbb R$: we are using the fact that $\mathbb C\otimes_\mathbb R\mathbb C$ is a direct product of isomorphic fields; this is not true for a general tensor product $L\otimes_K L$.  
\end{details}
\end{proof}

\section{Irreducible modules of $\g g^\R$, endotypes, and reductions}
\label{sec:reduction-E0}
In the rest of this paper we utilize the general framework of Section~\ref{sec:cat-gener} in the context of categories of modules for real and complex Lie superalgebras. 
Given a finite dimensional complex Lie superalgebra $\g g$, we use $\g g^\R$ to denote a real form of $\g g$ obtained as fixed points of an antilinear involution $\tau:\g g\to \g g$.
We set
$\catR:=
\Rep(\g g^\R)$. The corresponding category $\catC$ is 
$\Rep(\g g)$. 
As in Section~\ref{sec:cat-gener}, we have 
functors $\funcB_{\g g,\tau}$, $\funcE_{\g g,\tau}$, $\funcF_{\g g,\tau}$ relating  $\catR$ and $
\catC$, that  descend to functors relating the categories $\catR_{\g g,\tau}:=\Rep_\mathrm{fin}(\g g^\R)$ and $\catC_\g g:=\Rep_\mathrm{fin}(\g g)$. Recall from~\eqref{eq:EtypetauA} that for $[W]\in\Irr(\catC_\g g)$ we set
\[
\EType_{\g g,\tau}(W):=\EType_{\catR_{\g g,\tau}}^{}(W).
\]

\begin{rmk}
\label{remark:antilin-intertwiner}
Given an object $W$ of $\catC_\g g$, it will be helpful to think of $\funcB_{\g g,\tau}W$ in the following concrete way. As a vector (super)space, $\funcB_{\g g,\tau}W$ is just $W$ equipped with the conjugate complex structure. The action of any $x\in\g g$ on $\funcB_{\g g,\tau}W$ 
is the same as the action of $\tau(x)$ on $W$. 
Any homogeneous $\phi\in\wEnd_{\catC_\g g}(W,\funcB_{\g g,\tau}W)$ can be thought of as an (either even or odd) antilinear map $\phi:W\to W$ such that $\phi (x\cdot w)=\tau(x)\cdot \phi(w)$ for $x\in \g g$ and $w\in W$.
\end{rmk}

\begin{rmk}
\label{rmk:Endo-to-Class} Concrete realizations  of irreducible real modules of $\g g^\R$ can be obtained from the complete description of endotypes $\EType_{\g g,\tau}(W)$ for $[W]\in\Irr(\catC_\g g)$. Indeed by Lemmas~\ref{lem:E2} and~\ref{lem:EFandEF}(i) it follows that  for any row in Table~\ref{Table-1}, if we have $\funcE(V)=W\oplus \funcB W$,  then $\funcF(W)$ is irreducible and $\Phi(W)=[\funcF(W)]$. In other words, if $\EType_{\g g,\tau}(W)\not\in\{0_\R,1_\R,7_\R\}$, then $W$ itself is an irreducible real $\g g^\R$-module. 
If 
$\EType_{\g g,\tau}(W)\in\{0_\R, 1_\R,7_\R\}$,
  then $\funcF(W)=V\oplus V$ for a simple $V$ and we have $[\Phi(W)]=[V]$. In this case  $(\funcB,+)\in\mathfrak S_W$, i.e., there exists an even antilinear map $T:W\to W$ such that $Tx=xT$ for $x\in\g g^\R$. Having such a $T$, we can obtain a direct sum decomposition $W=V\oplus V$ by considering the two subspaces $\{w\pm Tw\,:\,w\in W\}$. When $W$ is a highest weight module, an explicit $T$ can also be computed by first defining it on the highest weight space of $W$ and then extending it to all of $W$.  
\end{rmk}

In the rest of this section we  show that  computing endotypes for modules of  $\g g^\R\in\cM$ can be reduced to the case of modules of the Lie superalgebras in Definition~\ref{dfn:Liestypes}.

\subsection{The direct sum reduction}
Suppose that $\g g=\g a\oplus\g b$ and $\tau=\tau_\g a\oplus \tau_\g b$ where $\tau_\g a:\g a\to \g a$ and $\tau_\g b:\g b\to \g b$ are antilinear involutions associated to real forms $\g a^\R$ and $\g b^\R$. 
The results of Section~\ref{sec:E-types-module} have the following immediate consequence. 
\begin{thm}
\label{thm:DirectSum}
Let $W\in\Irr(\catC_\g g)$. Then the following statements hold.
\begin{itemize}
\item[\rm(i)] There exist $[W_\g a]\in\Irr(\catC_\g a)$ and  $[W_\g b]\in\Irr(\catC_\g b)$ such that $W\cong W_\g a\otimes_\mathbb D W_\g b$, where $\mathbb D$ is defined from $\mathbb D_\g a:=\End^s_{\g a}(W_\g a)$ and  
$\mathbb D_\g b:=\End^s_{\g b}(W_\g b)$ by
\begin{equation}
\label{eq:DWAWB=}
\mathbb D:=\begin{cases}
\C& \text{if either $\mathbb D_{\g a}\cong \C$ or $\mathbb D_{\g b}\cong \C$},\\
\mathrm{Q}(1)& \text{otherwise}.
\end{cases}
\end{equation}
\item[\rm (ii)] $\EType_{\g g,\tau}(W)=\EType_{\g a,\tau_\g a}(W_\g a)\bullet \EType_{\g b,\tau_\g b}(W_\g b)$, where $\bullet$ denotes the product of $\BRm$. 
\end{itemize}
\end{thm}
Since abelian Lie algebras occur frequently as direct summands, we record the following simple fact. 
\begin{rmk}
\label{rmk-abbel}
If $\g g$ is an abelian Lie algebra, for $[W]\in\Irr(\catC_\g g)$ we have $\EType_{\g g,\tau}(W)\in\{0_\R,0_\C\}$. Furthermore,  
$\EType_{\g g,\tau}(W)=0_\R$
if and only if  elements of $\g g^\R$ act on $W$ by real scalars. To prove this, first note that every irreducible $\g g$-module $W$ is one-dimensional and corresponds to a linear functional $\mu\in\g g^*$. The existence of an antilinear map $\phi:W\to W$  as in Remark~\ref{remark:antilin-intertwiner} is equivalent to $\mu(\tau(x))=\oline{\mu(x)}$ for $x\in\g g$, or in other words to $\mu(\g g^\R)\sseq \R$. Furthermore, antilinear maps on a one-dimensional vector space are of the form $z\mapsto \alpha \bar z$, hence $\phi^2(z)=|\alpha|^2z$. This implies that $(\funcB,-)\not\in\mathfrak S_W$. 
\end{rmk}

\subsection{The triangular reduction}
\label{subsec:hyperb}
Recall that a triangular decomposition of $\g g$ is a vector space decomposition 
\begin{equation}
\label{eq:triang-decompo}
\g g=\oline{\g u}\oplus \g l\oplus \g u
\end{equation}
such that $\g u,\oline{\g u},\g l$ are subalgebras satisfying $[\g l,\g u]\sseq \g u$ and $[\g l,\oline{\g u}]\sseq \oline{\g u}$. We say this triangular decomposition is \emph{$\tau$-stable} when  $\tau(\g u)=\g u$, $\tau(\oline{\g u})=\oline{\g u}$, and $\tau(\g l)=\g l$. 

For any $\g g$-module $W$, the subspace  $W^{\g u}:=\left\{w\in W\,:\,x\cdot w=0\text{ for all }x\in\g u\right\}$ is an $\g l$-module. Set $\g q:=\g l\oplus\g u$ and let $
\mathrm{Ind}_{\g q}^{\g g}:\Rep(\g l)\to\Rep(\g g)$ denote the induction functor given by the assignment $
W\mapsto \g U(\g g)\otimes_{\g U(\g q)} W$. Given $[W]\in\Irr(\catC_\g g)$, if $\Ind_{\g q}^{\g g}(W^\g u)$  has a unique maximal proper $\g g$-submodule $M$, then the kernel of the canonical surjection 
$\Ind_{\g q}^{\g g}(W^\g u)\to W$ must be $M$, so that  $(\Ind_{\g q}^{\g g}(W^\g u))/M\cong W$.
For an $\g l$-module $W_\g l$, 
we denote the unique irreducible quotient of  
$\mathrm{Ind}_{\g q}^{\g g}
(W_\g l)$, if it exists, 
by 
$\underline{\mathrm{Ind}}_{\g q}^{\g g}
(W)$.

\begin{thm}
\label{thm:EtypeCplx1}
Assume that $\g g$ has a  $\tau$-stable triangular decomposition as in~\eqref{eq:triang-decompo}. Furthermore, assume that 
for every $W\in\Irr(\catC_\g g)$, the $\g l$-module $W^\g u$ is irreducible and the induced module $\Ind_{\g q}^{\g g}W^\g u$ has a unique maximal proper $\g g$-submodule. 
Then 
for every $[W]\in\Irr(\catC_{\g g})$ we have $\EType_{\g g,\tau}(W)= \EType_{\g l,\tau}(W^{\g u})$. 
\end{thm}
\begin{proof}
According to Theorem~\ref{thm:Table}, we need to verify that the symmetry data of $W$ and $W^{\g u}$ are the same. 
From $\tau(\g u)=\g u$ it follows that $(\funcB_{\g g,\tau}W)^\g u=\funcB_{\g l,\tau}(W^{\tau(\g u)})=\funcB_{\g l,\tau}\left(W^{\g u}\right)$. Since  $W^{\g u}$ generates $W$ as a $\g g$-module,  we have injective linear maps 
\begin{equation}
\label{eq:1sthom}
\wHom_{\catC_{\g g,\tau}}(W,W)\to 
\wHom_{\catC_{\g l,\tau}}(W^{\g u},W^{\g u})
\end{equation}
and
\begin{equation}
\label{eq:2ndhom}
\wHom_{\catC_{\g g,\tau}}(W,\funcB_{\g g,\tau}W)\to 
\wHom_{\catC_{\g l,\tau}}(W^{\g u},
\funcB_{\g l,\tau}
(W^{\g u})).
\end{equation}
To complete the proof, it suffices to prove that~\eqref{eq:1sthom} and~\eqref{eq:2ndhom} are surjections. We only give the argument for~\eqref{eq:2ndhom}, which is slightly less trivial.
By Remark~\ref{remark:antilin-intertwiner}, any homogeneous  $\phi\in \wHom_{\catC_{\g l,\tau}}(W^{\g u},
\funcB_{\g g,\tau}
W^{\g u})$ corresponds to an (even or odd) antilinear map 
$T:W^{\g u}\to W^\g u$ such that $T x=\tau(x)T$ for $x\in \g l$. From $T$ we obtain  an antilinear map
\[
U(\g g)\otimes_{U(\g q)}W^{\g u}\to 
U(\g g)\otimes_{U(\g q)}W^{\g u}\ ,\
D\otimes w\mapsto \tau(D)\otimes Tw,
\]
corresponding  to a $\g g$-module homomorphism
$U(\g g)\otimes_{U(\g q)}W^{\g u}\to \funcB_{\g g,\tau}(U(\g g)\otimes_{U(\g q)}W^{\g u})$. 
Passing to the quotients on both sides by the unique maximal proper invariant subspaces, we obtain a $\g g$-module homomorphism ${\tilde T}:W\to \funcB_{\g g,\tau}(W)$ that restricts to $T$, hence it corresponds to $\phi$ under~\eqref{eq:2ndhom}.    
\end{proof}


In the rest of this subsection, we describe two kinds of triangular decompositions that satisfy the assumption of Theorem~\ref{thm:EtypeCplx1}. See Definitions~\ref{dfn:hyper} and~\ref{dfn:gradedtyp}. 

One way to obtain a $\tau$-stable triangular decomposition is to choose  $x_\circ\in\g g^\R_\eev$ such that $\ad_{x_\circ}:\g g^\R\to\g g^\R$ is diagonalizable with real eigenvalues (i.e., a \emph{hyperbolic} element of $\g g^\R$) and decompose $\g g^\R$ as 
\begin{equation}
\label{eq:gR=}
\g g^\R=
\oline{\g u}^{\R}\oplus\g l^\R\oplus{\g  u}^{\R},
\end{equation} 
where $\g u^\R:=\oplus_{c>0}\g g^{\R,c}$,
$\bar{\g u}^\R:=\oplus_{c<0}\g g^{\R,c}$,
and $\g l^\R:=\left\{y\in\g g^\R\,:\,\ad_{x_\circ}(y)=0\right\}$.
Here $\g g^{\R,c}$ denotes the $c$-eigenspace of $\ad_{x_\circ}$ for $c\in\R$. Passing to complexifications, we obtain a $\tau$-stable triangular decomposition. 
\begin{dfn}
\label{dfn:hyper}
A triangular decomposition~\eqref{eq:triang-decompo} is said to be of  \emph{hyperbolic type} if it is  obtained as above from a hyperbolic $x_\circ\in\g g^\R_\eev$.  
\end{dfn}

\begin{rmk}
\label{rmk:eigx0}
Let $W$ be a $\g g$-module that is generated by an irreducible finite dimensional $\g l$-module $W'\sseq W$. Then by Schur's Lemma, 
 $x_\circ$ acts on $W'$ by a scalar $c\in\C$, and since $\g g$ is a direct sum of $\ad_{x_\circ}$-eigenspaces it follows that $W$ is a  direct sum of $x_\circ$-eigenspaces as well.
\end{rmk}

\begin{lem}
\label{lem:uniquemaxx0}
Assume that either $\oline{\g u}\sseq \g g_\ood$ or  the triangular decomposition~\eqref{eq:triang-decompo} is obtained by a hyperbolic element $x_\circ$. 
Then
for every 
$[W]\in\Irr(\Rep(\g l))$ the induced module
 $
\mathrm{Ind}_{\g q}^{\g g}(W)$ has a unique maximal proper $\g g$-submodule.
\end{lem} 
\begin{proof}
A proper submodule $M$ of  $\Ind_{\g q}^{\g g}(W)=\mathfrak U(\oline{\g u})\otimes_\C W$ satisfies $M\cap (1\otimes W)=\{0\}$. Our strategy is to show that there is a proper subspace of $\Ind_{\g q}^{\g g}(W)$ that contains  every proper submodule; hence  the sum of proper submodules remains proper.

First assume that $\oline{\g u}\sseq \g g_\ood$. We follow an argument originally due to Shapovalov. It suffices to  prove that every proper submodule of $\Ind_{\g q}^{\g g}(W)$ is contained in  $\mathfrak U(\oline{\g u})_+\otimes_\C W$, where $\mathfrak U(\oline{\g u})_+$ is the augmentation ideal of $\mathfrak U(\oline{\g u})$. To this end, we prove that any element of $\Ind_{\g q}^{\g g}(W)$ outside $\mathfrak U(\oline{\g u})_+\otimes_\C W$ generates $\Ind_{\g q}^{\g g}(W)$.  Take such an element 
$w\in \Ind_{\g q}^{\g g}(W)$. 
From $\oline{\g u}\sseq \g g_\ood$ it follows that $\g U(\oline{\g u})$ is isomorphic to the exterior algebra on $\oline{\g u}$, hence $\g U(\oline{\g u})$ has a natural $\Z$-grading.
By irreducibility of $W$ we can express $w$ as $w=\sum_{k=0}^N w_k$ where $w_0\neq 0$ and $w_k\in \mathfrak U(\oline{\g u})^{(k)}\mathfrak U(\g l)w_0 $
for $k\geq 1$. Thus \[
w=(1+u_1+\cdots +u_N)w_0,
\] where $u_i\in \mathfrak U(\oline{\g u})^{(k)}\mathfrak U(\g l)$. Now $u:=u_1+\cdots+u_N$ is nilpotent in $\mathfrak U(\g l\oplus \oline{\g u})$, hence $1+u$ is invertible. Thus, the $\g g$-submodule generated by $w$ contains $w_0$, and the latter generates $\Ind_{\g q}^{\g g}(W)$.  

Next assume that the triangular decomposition is obtained by a hyperbolic $x_\circ$. From Remark~\ref{rmk:eigx0} It follows immediately that $\Ind_{\g q}^{\g g}(W)$ is a direct sum of eigenspaces for $x_\circ$.   By Schur's Lemma, $x_\circ$ acts on  the irreducible $\g l$-module $1\otimes W\sseq \Ind_{\g q}^{\g g}(W)$ by a scalar $c\in\C$.
The action of $\oline{\g u}$ moves each eigenspace $\Ind_{\g q}^{\g g}(W)_{c'}$ 
of $\Ind_{\g q}^{\g g}(W)$
with eigenvalue $c'$ into another one with eigenvalue $c''$ where $\Re(c'')<\Re(c')$. Thus, we have $1\otimes W=\Ind_{\g q}^{\g g}(W)_c$. Consequently, 
 every proper submodule  $M\sseq \Ind_{\g q}^{\g g}(W)$ must be a subspace of the sum of $x_\circ$-eigenspaces other than $c$, i.e.,
$
\textstyle M\sseq \bigoplus_{c'\neq c}\Ind_{\g q}^{\g g}(W)_{c'}$. 
The right hand side is a proper subspace of  $\Ind_{\g q}^{\g g}(W)$.
\end{proof}

\begin{dfn}
\label{dfn:gradedtyp}
The triangular decomposition~\eqref{eq:triang-decompo} is called of \emph{graded type} if the following hold:
\begin{itemize}
\item[(i)] $\g g=\bigoplus_{i\in \Z}\g g^{(i)}$ is $\Z$-graded, i.e., each $\g g^{(i)}$ is $\mathbb Z_2$-graded and 
$[\g g^{(i)},\g g^{(j)}]\sseq \g g^{(i+j)}$.

\item[(ii)] $\g l=\g g^{(0)}$, $\g u=\bigoplus_{i\geq 1}\g g^{(i)}$ and 
$\oline{\g u}=\bigoplus_{i\leq -1}\g g^{(i)}$.

\item[(iii)] For some Borel subalgebra $\g b^{(0)}\sseq (\g g^{(0)})_\eev$, if we set $\g b:=\g b^{(0)}\oplus \g u$ then 
$
[\g b_\ood,\g b_\ood]\sseq [\g b_\eev,\g b_\eev]
$ and 
$\left[\g b^{(0)},(\g g^{(1)})_\eev\right]=(\g g^{(1)})_\eev$.
\item[(iv)] $\g g^{(1)}$ generates $\g u$ and $\g g^{(-1)}$ generates $\oline{\g u}$. 

\end{itemize}

\end{dfn}

\begin{prp}
\label{prp:Wq+-inv} Let $\mathrm{Res}_{\g l}^{\g g}:\catC_\g g\to \catC_{\g l}$ be the functor $W\mapsto W^{\g u}$. 
\begin{itemize}
\item[\rm (i)] Assume that the triangular decomposition~\eqref{eq:triang-decompo} is of graded type. Then $\mathrm{Res}_{\g l}^{\g g}$ induces a bijection $\Irr(\catC_\g g)\to \Irr(\catC_\g l)$.  

\item[\rm (ii)] Assume that the triangular decomposition~\eqref{eq:triang-decompo}  is of hyperbolic type.
Then $\mathrm{Res}_{\g l}^{\g g}$ induces an injection $\Irr(\catC_\g g)\to \Irr(\catC_\g l)$.  
\end{itemize}
\end{prp}
\begin{proof}
(i) This is one of the steps in Kac's  classification of irreducible modules of Cartan-type and periplectic Lie superalgebras; see~\cite[Prop 5.2.5]{KacAdvances}. We remark that Kac's arguments hold under the slightly weaker assumptions stated in Definition~\ref{dfn:gradedtyp}. 

(ii) 
This is standard, but for the reader's convenience we sketch the argument. 
Assuming that $W^{\g u}$ is irreducible, injectivity of  $\Irr(\catC_\g g)\to\Irr(\catC_\g l)$ follows from Lemma~\ref{lem:uniquemaxx0} and 
the isomorphism $W\cong \underline{\mathrm{Ind}}_{\g l}^{\g g}(W^{\g u})$. 
Next we prove that $W^\g u$ is an irreducible $\g l$-module. 
Suppose that there exists a proper irreducible $\g l$-submodule $W'\sseq W^{\g u}$. By Schur's Lemma, $x_\circ$ acts on $W'$ by a scalar $c\in\C$. By Remark~\ref{rmk:eigx0}, $W$ (and hence $W^{\g u}$) is a direct sum of $x_\circ$-eigenspaces. 
For  any $c'\in\C$, let $W_{c'}$ be the $c'$-eigenspace of $x_\circ$ on $W$.
Every $W'':=W_{c'}\cap W^{\g u}\sseq W^{\g u}$  is an $\g l$-module and by the PBW Theorem if $W''\neq 0$ then  we have 
$W=\g U(\g g)W''=\g U(\oline{\g u})W''$. It follows that the set of $x_\circ$-eigenvalues on $W$ is a subset of $\left\{c'-t\,:t\in\R^{\geq 0}\right\}$ that contains $c'$. From this, it follows that $x_\circ$ acts on all of $W^{\g u}$ by the same scalar $c$. But then $
\mathfrak U(\g g)W'=\mathfrak U(\oline{\g u})W'\sseq W'\oplus\bigoplus_{t>0}W_{c-t}$,
and the right hand side is a proper subspace of $W$. This contradicts irreducibility of $W$. 
 \end{proof}

\begin{cor}
\label{thm:EtypeCplx}
Assume that the $\tau$-stable triangular decomposition~\eqref{eq:triang-decompo} is either of hyperbolic type, or of  graded type and  $\oline{\g u}\sseq \g g_\ood$.
 Then 
for $[W]\in\Irr(\catC_{\g g})$, we have $\EType_{\g g,\tau}(W)= \EType_{\g l,\tau}(W^{\g u})$. 
\end{cor}
\begin{proof}
By Proposition~\ref{prp:Wq+-inv} and Lemma~\ref{lem:uniquemaxx0},  the hypotheses of 
Theorem~\ref{thm:EtypeCplx1} are satisfied.
\end{proof}

\subsection{The reductive reduction}
\label{subsec:Liered}
Recall that $\g f_\g g\sseq\g g$ denotes the common kernel of the irreducible finite dimensional $\g g$-modules, and set $\g g':=\g g/\g f_\g g$. We denote the image of $[W]\in\Irr(\catC_\g g)$ under the canonical 
 map 
$\Irr(\catC_\g g)\to\Irr(\catC_{\g g'})$ 
 by $[W]$. We have  $\tau(\g f_\g g)=\g f_\g g$ and thus $\tau$ induces an  antilinear involution $\tau':\g g'\to\g g'$. The following assertion is trivial.
\begin{prp}
\label{thm:triv-reduction} For $[W]\in\Irr(\catC_\g g)$ we have $\EType_{\g g,\tau}(W)=\EType_{\g g',\tau'}(W)$. 
\end{prp}
In the rest of this subsection we assume that $\g g=\g g_\eev$. Let $\g r$ denote the solvable radical of $\g g$, i.e., the sum of solvable ideals of $\g g$. Then $\g r$ is $\tau$-stable and $\g g/\g r$ is semisimple. Furthermore, $[\g g,\g r]$ is a $\tau$-stable ideal of $\g g$. 
\begin{lem}
$\g g/[\g g,\g r]$ is reductive.
\end{lem}
\begin{proof}
By Levi's theorem, we have a semidirect sum decomposition $\g g=\g s\ltimes \g r$ where $\g s$ is semisimple. Thus  $\g g/[\g g,\g r]\cong \g s\ltimes (\g r/[\g g,\g r])$. Clearly $\g r/[\g g,\g r]$ is an abelian central ideal of $\g g/[\g g,\g r]$. Hence $\g g/[\g g,\g r]$ is reductive.  
\end{proof}

\begin{lem}
\label{lem:[r,r]triv}
Let $W$ be an irreducible $\g g$-module. Then $[\g g,\g r]$ acts on $W$ trivially.   
\end{lem}
\begin{proof}
 By Lie's theorem, every irreducible $\g r$-module is one-dimensional. Thus, there exists $\lambda\in\g r^*$ such that the subspace $
W_\lambda:=\{w\in W\,:\,x\cdot w=\lambda(x)w\text{ for all }x\in\g r\}$
 is nonzero. For $x\in\g r$, we have $\tr(x|_{W_\lambda})=\lambda(x)\dim W_\lambda$. By a  well-known lemma from the proof of Lie's theorem (for example see \cite[Lemma 9.13]{FultonHarris1991}), $W_\lambda$ is $\g g$-invariant. But then the elements of $[\g g,\g r]$ are sums of commutators of linear maps on $W_\lambda$, hence their traces (as linear maps on $W_\lambda$) vanish. This proves that $\lambda(x)\dim W_\lambda=0$ for $x\in[\g g,\g r]$, hence $\lambda(x)=0$. In other words, $[\g g,\g r]$ acts by zero on $W_\lambda$. Next set $W':=\{w\in W\,:\,x\cdot w=0\text{ for all }x\in [\g g,\g r]\}$. Then $W_\lambda\sseq W'$, and $W'$ is a $\g g$-submodule (since $[\g g,\g r]$ is an ideal of $\g g$), hence $W'=W$.     
\end{proof}
Since $[\g g,\g r]$ is $\tau$-stable, $\tau$ induces an antilinear involution on $\g g/[\g g,\g r]$, which we will also denote by $\tau$. The following statement is an immediate consequence of Lemma~\ref{lem:[r,r]triv}. 
\begin{thm}
Set $\g g_\mathrm{red}:=\g g/[\g g,\g r]$. For $[W]\in\Irr(\catC_\g g)$ we have $\EType_{\g g,\tau}(W)=\EType_{\g g_\mathrm{red},\tau}(W)$. 
\end{thm}

\section{Endotypes for realified complex Lie superalgebras}
\label{sec:realified}

The goal of this section is to prove Theorem~\ref{thm:W'W''D}, which classifies endotypes of real superalgebras that have a complex structure. Given a complex Lie superalgebra $\g g$,  
its complex conjugate $\oline{\g g}$  is the Lie superalgebra with the same superbracket but equipped with the conjugate structure as a complex vector space.  If $W$ is a $\g g$-module, then $\oline{W}$ is a $\oline{\g g}$-module, and vice versa.  
We have a $\C$-linear isomorphism
\[
\g g\otimes_\R\C\to \g g\oplus \oline{\g g}\ ,\ 
x\otimes z\mapsto (zx,z \bar x),
\]
where to avoid confusion we denote elements of $\oline{\g g}$ by $\bar x$ for $x\in\g g$. 
The Lie superalgebra $\g g$ is the real form of $\g g\oplus\oline{\g g}$ corresponding to $\tau(x,\bar y)=(y,\bar x)$.

\begin{details}
Given an irreducible $\g g$-module $W$ with $\mathbb D:=\End_{\g g}^s(W)$, for clarity we always assume that the latter isomorhism is $\C$-linear. This is related to the following subtlety: if we set $\mathbb D':=\End_{\oline{\g g}}^s(\oline W) $, then $\mathbb D\cong \mathbb D'$ but this isomorphism is \emph{antilinear}: the point is that the scalars $z\in \mathbb D_\eev\cong \C$ act on $\oline{W}$ as scalar multiplication by $\bar z$, not $z$. 
\end{details}

\begin{rmk}
\label{rmk:(-1)|x|isom}Recall that for a $\g g$-module $W$, we use $^\sigma W$ to denote the twist of $W$ by the automorphism $x\mapsto (-1)^{|x|} x$ of $\g g$. Then  $^\sigma W\cong W$ as $\g g$-modules via the map $w\mapsto (-1)^{|w|}w$.
\end{rmk}
\begin{rmk}
\label{rmk:PiWW}Let $W'$ and $W''$ be modules for Lie superalgebras $\g a$ and $\g b$, and let $\mathbb D$ be defined as in~\eqref{eq:DWAWB=}. Then we have  $\g a\oplus\g b$-module isomorphisms
$W'\otimes_\mathbb D W''\cong \Pi W'\otimes_\mathbb D\Pi W''$ given by $v\otimes w\mapsto (-1)^{|w|}v\otimes w$, and  $\Pi(W'\otimes_{\mathbb D} W'') \cong W'\otimes_{\mathbb D} \Pi W''$ given by the identity map. 
\end{rmk}

\begin{details}
For the first isomorphism:\[
\xymatrix{\ar[d] v\otimes w\ar[r]& 
 (-1)^w v\otimes w\ar[d]\\
(-1)^{yv}xv\otimes yw \ar[r] & (-1)^{yv+y+w}xv\otimes yw 
 }
\]
For the second isomorphism, we note that since $\Pi$ is on $W''$, there is no effect on sign changes. 
\end{details}

In the rest of this section we assume that $[W]\in\Irr(\catC_{\g g\oplus\oline{\g g}})$.  By Proposition~\ref{prp:W=WAWB} we have  $W= W'\otimes_{\mathbb D}W''$ where $W'$ and $W''$ are irreducible modules of $\g g$ and $\oline{\g g}$, respectively. 

\begin{lem}
\label{lem:PiSigmaWW}
$\Pi\in\mathfrak S_W$ if and only if $\Pi$ belongs to exactly one of $\mathfrak S_{W'}$ and $\mathfrak S_{W''}$. 
\end{lem}
\begin{proof}
Note that  $\Pi\in\mathfrak S_W$ if and only if 
$
\End_{\g g\oplus {\oline{\g g}}}(W'\otimes_{\mathbb D}W'')\cong \mathrm{Q}(1)$.
The assertion follows from  general facts about tensor products of modules of types $\mathtt{M}$ and $\mathtt{Q}$~\cite[Sec. 3.1.3]{ChengWang2012}.
\end{proof}
\begin{lem}
\label{lem:TT0T1}
Assume that $W \cong \funcB_{\g g\oplus\oline{\g g},\tau}W$, and the isomorphism corresponds to the even antilinear map $T: W'\otimes_\mathbb D W''\to W'\otimes_\mathbb D W''$. Fix a nonzero $w_0\in W''_\eev$. Then there exist $v_0\in W'_\eev$, $v_1\in W'_\ood$, an even antilinear map $T_0:W'\to W''$, and an odd antilinear map $T_1:W'\to W''$ satisfying the following:
\begin{itemize}
\item[\rm (i)] $T(v\otimes w_0)=v_0\otimes_\mathbb D T_0v+v_1\otimes_\mathbb D T_1v$ for $v\in W'$.
\item[\rm (ii)] $T_0x=\bar xT_0$ and $T_1x=(-1)^{|x|}\bar xT_1$ for $x\in\g g$.
\end{itemize}
If $\mathbb D\cong \mathrm Q(1)$, then we can assume that $T_1=0$.
\end{lem}
\begin{proof}
Fix a homogeneous $\mathbb D$-basis $v_1,\ldots,v_p,v_{\bar 1},\ldots,v_{\bar{q}}$ of $W'$, where the first $p$ are in $W'_\eev$ and the last $q$ are in $W'_\ood$ (and if $\mathbb D=\mathrm Q(1)$ then we assume $q=0$). Then we have
\[
T(v\otimes_\mathbb D w_0)=\sum_{i=1}^p v_i\otimes_\mathbb D T_iv+
\sum_{i=1}^p v_{\bar i}\otimes_\mathbb D T_{\bar i}v\quad \text{for }v\in W',
\]
where the $T_i:W'\to { W''}$ are even antilinear maps that satisfy $T_ix=\bar xT_i $ for $x\in\g g$, whereas the $T_{\bar i}:W'\to {W''}$ are odd  antilinear maps that satisfy $T_{\bar i}x=(-1)^{|x|}\bar xT_{\bar i}$.

 \begin{details}This follows from the commutative diagram
\[
\xymatrix{v\otimes w_1 \ar[r]^T\ar[d]&  \sum_i v_i\otimes T_iv\ar[d]\\
xv\otimes w_1\ar[r]& \sum_i v_i\otimes T_i(xv)=\sum_i (-1)^{|x|\cdot |v_i|}v_i\otimes \bar x T_iv
}\]
\end{details} 

By Schur's Lemma, any two $T_i$ are related by a scalar in $\mathbb D_\eev$. 
This allows us to combine the summands $v_i\otimes_\mathbb D T_iv$ into one term $v_0\otimes_\mathbb D T_0v$. 
The same argument applies to the summands $v_{\bar i}\otimes_\mathbb DT_{\bar i}v$. 
\end{proof}

\begin{prp}
\label{prp:funcB+}We have $(\funcB,+)\in \mathfrak S_W$ if and only if either $W'\cong\oline{W''}$ or $W'\cong\Pi \oline{W''}$ as $\g g$-modules. Furthermore, $(\funcB,-)\not\in \mathfrak S_W$.   
\end{prp}
\begin{proof}
Suppose that $\mathfrak S_W\cap\{(\funcB,+),(\funcB,-)\}\neq\varnothing$. 
By Remark~\ref{rmk:PiWW} we can assume that $W''_\eev$ is nonzero.
Now in  Lemma~\ref{lem:TT0T1}, 
if $T_0\neq 0$ then $\oline{W''}\cong W'$, and if $T_1\neq 0$ then ${{\oline {W''}}}\cong \Pi W'$ (for the latter, we use the isomorphism $W''\cong {^\sigma W''}$; see Remark~\ref{rmk:(-1)|x|isom}).  

\begin{details}
If some $T_i\neq 0$, we must have $\oline{W'}\cong W$ and if some $T_{\bar i}\neq 0$, it yields as isomorphism $W\cong\Pi{^\sigma\oline{ W'}}\cong \Pi \oline{W'}$.
\end{details}

Conversely, if $\oline{W''}\cong W'$ then define
\[
T:W'\otimes_\mathbb D \oline{W'}\to W'\otimes_\mathbb D \oline{W'}\ ,\
v\otimes_\mathbb D w\mapsto (-1)^{|v|\cdot |w|}w\otimes_\mathbb D v,
\]

\begin{details}

Recall that the left actions on $W'$ and $\oline{W'}$ are conjugate to each oteher: we have $d\bar\cdot w=\bar d\cdot w$, where $\bar \cdot$ denotes the action of $d$ on $\oline{W'}$. Well-definedness of $T$ follows from
\[
v\cdot d\otimes w\mapsto (-1)^{(d+v)w}w\otimes v\cdot d=(-1)^{(d+v)w+vd}w
\otimes \bar d\bar\cdot v,
=
\]
and
\[v\otimes d\bar\cdot w\mapsto 
(-1)^{(d+w)v}\bar d\cdot w\otimes v=
(-1)^{(d+w)v+dw}w\cdot \bar d\otimes  v.
\]
Commuting with action follows from
\[
\xymatrix{v\otimes w \ar[d]_{a\otimes b} \ar[r]& (-1)^{vw}w\otimes v\ar[d]^{(-1)^{ab}b\otimes a}\\
(-1)^{bv}av\otimes bw \ar[r]& (-1)^{bv+(a+v)(b+w)}bw\otimes av}
\]

\end{details}

and if $\oline{W''}\cong\Pi W'$ then define 
\[
T:W'\otimes_\mathbb D \Pi\oline{W'}\to W'\otimes_\mathbb D \Pi\oline{W'}\ ,\
v\otimes_\mathbb D w\mapsto (-1)^{|v|\cdot |w|+|v|}w\otimes_\mathbb D v.
\]

\begin{details}
For well-definedness of $T$ first note that the left actions of $\mathbb D$ on vectors of $W'$ and $\Pi \oline{W'}$ are the \emph{same} up to conjugation, but the right actions are \emph{different} up to a sign. More precisely,  
considering $w$ as a vector in $W'$, we have \[
w\cdot d:=(-1)^{dw}d\cdot w
\]
but
if $\cdot'$ denotes the (left or right) actions on $\Pi\oline{W'}$, then 
 considering $w\in \Pi\oline{W'}$, we have \[
w\cdot' d:=
(-1)^{(w+1)d}d\cdot' w=
(-1)^{(w+1)d}\bar d\cdot w=(-1)^d w\cdot \bar d.
\]
Now we have 
\[
v\cdot d\otimes w\mapsto (-1)^{(v+d)(w+1)}
w\otimes v\cdot d=
(-1)^{(v+d)(w+1)+d}
w\otimes v\cdot' \bar d
=(-1)^{(v+d)(w+1)+vd+d}w\otimes \bar d\cdot' v,
\]
and also 
\[
v\otimes d\cdot' w\mapsto 
(-1)^{v(d+w+1)}\bar d\cdot w\otimes v
=
(-1)^{v(d+w+1)+dw}w\cdot \bar d\otimes v,
\]
which complete the proof of well-definedness.

For compatibility with actions, we have 
\[
\xymatrix{v\otimes w \ar[r]\ar[d]_{a\otimes b}& (-1)^{v(w+1)}w\otimes v\ar[d]^{(-1)^{ab}b\otimes a}\\
(-1)^{bv}av\otimes bw\ar[r] &(-1)^{bv+(a+v)(b+w+1)}bw\otimes av}
\]
Note that in the latter case we have 
\[
v\otimes w\mapsto (-1)^{vw+v}w\otimes v
\mapsto (-1)^{vw+v}(-1)^{(v+1)(w+1)+(w+1)}v\otimes w=v\otimes w. 
\]

\end{details}

In both cases, $T$ is even, antilinear, and satisfies $T(x,\bar y)=(x,\bar y)T$  for $(x,\bar y)\in\g g\oplus\oline{\g g}$. Furthermore, $T^2$ is the identity map. This proves that in both cases we have $(\funcB,+)\in \mathfrak S_W$, and thus $(\funcB,-)\not\in\mathfrak S_W$ (this follows from Table~\ref{Table-1}).
\end{proof}

\begin{prp}
\label{prp:PiB+PiB-}
$\{(\Pi\funcB,+),(\Pi\funcB,-)\}\cap \mathfrak S_W=\varnothing$.
\end{prp}
\begin{proof}
Suppose that $W\cong \Pi\funcB W$ and the isomorphism corresponds to the odd, antilinear map $
T:W'\otimes_\mathbb D W''\to W'\otimes_\mathbb D W''$.  
By Remark~\ref{rmk:PiWW}, we can assume that $W''_\eev$ is nonzero. As in Lemma~\ref{lem:TT0T1}, for a nonzero $w_0\in W''_\eev$ we can write $T$ as
\begin{equation}
\label{eq:Tvw0}
T(v\otimes_\mathbb D w_0)=v_\eev\otimes_\mathbb D T_1 v+v_\ood\otimes_\mathbb D T_0v,
\end{equation}
where $T_0:W'\to \oline{W''}$ is an even linear map that satisfies $T_0x=(-1)^{|x|}xT_0$ for $x\in\g g$, while $T_1:W'\to \oline{W''}$ is an odd linear map that  satisfies $T_1x= xT_1$ for $x\in\g g$. 
(If $\mathbb D=\mathrm Q(1)$, we can assume $T_1=0$.)

\begin{details}
For two even maps $T_0$ and $T'_0$ satisfying $T_0x=(-1)^{|x|}xT_0$ and $T'_0x=(-1)^{|x|}xT'_0$ we obtain an even map $S=T_0(T'_0)^{-1}:W\to W$ such that $Sx=xS$ for $x\in\g g$. By Schur's lemma, $S$ is a scalar in $\mathbb D_\eev=\C$. Thus, $T_0$ is unique up to a scalar. The same argument works for $T_1$. 
\end{details}

Assume that $T_0\neq 0$. Then $W'\cong {^\sigma\oline{W''}}\cong \oline{W''}$ and by Proposition~\ref{prp:funcB+} we have $(\funcB,+)\in\mathfrak S_W$. 
Since $T$ is odd, this also proves that $W'_\ood$ and $W''_\ood$ are nonzero. But now choosing a nonzero $v_1\in W'_\ood$ we obtain
\begin{equation}
\label{eq:Tjkdj}
T(v_1\otimes_\mathbb D w)=S_0 w\otimes_\mathbb D w''_0+S_1w\otimes_\mathbb D w''_1,
\end{equation}
where 
$w_0''\in W''_\eev$, $w''_1\in W''_\ood$, 
$S_0:W''\to \oline{W'}$ is even, $S_1:W''\to \oline{W'}$ is odd, and we have $S_0x=(-1)^{|x|}xS_0$ and $S_1x=(-1)^{|x|}xS_1$ for $x\in \g g$. 
By comparing~\eqref{eq:Tjkdj} with~\eqref{eq:Tvw0} at $v_1\otimes w_0$, it follows that  $S_1\neq 0$, hence  $W'\cong \Pi{^\sigma \oline{W''}}\cong \Pi \oline{W''}$.

\begin{details}
Here we use the isomorphism $W\cong {^\sigma W}$ given by $w\mapsto (-1)^{|w|}w$.
\end{details} 

From the above discussion it follows that in particular $W'\cong \Pi W'$, hence  $\mathbb D=\mathrm{Q}(1)$. Since the external tensor product (over $\C$) of two irreducible modules of type $\mathtt Q$ is a direct sum of two copies of an irreducible module of type $\mathtt M$, it follows that $\Pi\not\in\mathfrak S_W$. Furthermore, from $W'\cong \oline{W''}$ and Proposition~\ref{prp:funcB+}  we have $(\funcB,+)\in\mathfrak S_W$.  Finally, from  Table~\ref{Table-1} it follows that $(\Pi\funcB,\pm )\not\in\mathfrak S_W$. 

The argument for the case where  $T_1\neq 0$ is similar: in this case $W'\cong\Pi\oline{W''}$, and 
from~\eqref{eq:Tjkdj} it follows that $S_0\neq 0$, hence $W'\cong {^\sigma\oline{W''}}\cong \oline{W''}$.
\end{proof}
In the following theorem, we denote the symmetric difference of two sets $E$ and $E'$ by $E\Delta E'$, i.e., \[
E\Delta E':=(E\backslash E')\cup (E'\backslash E).
\] 
\begin{thm}
\label{thm:W'W''D}Let
$W=W'\otimes_\mathbb D W''$
be an irreducible  $\g g\oplus\oline{\g g}$-module.
\begin{itemize}
\item[\rm (i)] If $\Pi\in\mathfrak S_{W'}\Delta \mathfrak S_{W''}$, then $\EType_{\g g\oplus\oline{\g g}}(W)=1_\C$.
\item[\rm (ii)] If $\Pi\not\in \mathfrak S_{W'}\Delta \mathfrak S_{W''}$ and either $W''\cong\oline{W'}$ or $W''\cong \Pi\oline{W'}$, then $\EType_{\g g\oplus\oline{\g g}}(W)=0_\R$.
\item[\rm (iii)] If $\Pi\not\in \mathfrak S_{W'}\Delta \mathfrak S_{W''}$ and neither $W''\cong\oline{W'}$ nor $W''\cong \Pi\oline{W'}$, then $\EType_{\g g\oplus\oline{\g g}}(W)=0_\C$.

\end{itemize}
\end{thm}
\begin{proof}
This follows from Lemma~\ref{lem:PiSigmaWW}, Proposition~\ref{prp:funcB+}, and Proposition~\ref{prp:PiB+PiB-}
\end{proof}

\section{Computing $\EType_{\g g,\tau}(W)$ when $\g g$ is of basic type or of type $\mathbf Q(n)$}
\label{sec:ETypeWbasic-Q}

Sections~\ref{sec:reduction-E0} and~\ref{sec:realified} reduce the characterization of irreducible real representations of $\g g^\R$ and their endotypes to the case of real forms of Lie superalgebras of Definition~\ref{dfn:Liestypes}. In the rest of this paper, we give explicit answers to this characterization problem for the latter types of Lie superalgebras. For the classification of real forms of simple Lie superalgebras, see~\cite{Serganova,Parker}.

\subsection{Endotypes for Clifford-type Lie superalgebras}  
In  this subsection we assume that  $\g h=\g h_\eev\oplus\g h_\ood$ is a Lie superalgebra such that $[\g h_\eev,\g h]=0$. Let $\g h^\R$ be a real form of $\g h$ corresponding to an antilinear involution $\tau:\g h\to \g h$. Given an irreducible $\g h$-module $(\rho,W)$, by Schur's Lemma there exists $\lambda\in\g h_\eev^*$ such that  $\rho(h)v=\lambda(h)1_W$ for every $h\in \g h_\eev$.
We define a symmetric bilinear form
\[
B_\lambda:\g h_\ood\times \g h_\ood\to \C\ ,\
B_\lambda(x,y):=\lambda([x,y]).
\]
Set $R_\lambda:=\mathrm{Rad}(B_\lambda)$. Then $W$ is an irreducible module for $U(\g h)/I_\lambda$, where $I_\lambda$ is the ideal of $U(\g h)$ generated by $R_\lambda$ and the elements $h-\lambda(h)1$ for $h\in\g h_\eev$. Note that $U(\g h)/I_\lambda$ is isomorphic to the  complex Clifford superalgebra $\mathrm{Cl}_r(\C)$ of rank $r:=\dim(\g h_\ood/R_\lambda)$.  

%

The action of $\g h_\eev$ on both $\funcB_{\g h,\tau} W$ and $\Pi\funcB_{\g h,\tau} W$ is by the $\C$-linear functional $\oline{\lambda\circ\tau}$. Thus, if $\wHom_{\catC_\g h}(W,\funcB_{\g h,\tau}W)\neq 0$ then 
we must have $\lambda=\oline{\lambda\circ\tau}$. 
Indeed it follows from Proposition~\ref{prp:ETClp,q} below that the converse also holds. Assume that $\lambda=\oline{\lambda\circ\tau}$. Then
\[
B_\lambda(\tau(x),\tau(y))=\lambda\circ\tau([x,y])=\oline{\lambda([x,y])}=\oline{B_\lambda(x,y)}.
\]
In particular, the restriction of $B_\lambda$ to $\g h_\ood^\R\times \g h_\ood^\R$  is real-valued, with radical equal to 
$R_\lambda\cap \g h_\ood^\R$. 
Thus, $B_{\lambda}$ induces  a non-degenerate real symmetric bilinear form 
\begin{equation}
\label{eq:Blatau}
B_{\lambda,\tau}:\g h_{\ood,\lambda}\times\g h_{\ood,\lambda}\to\R\quad,\quad \g h_{\ood,\lambda}:=\g h_\ood^\R/R_\lambda\cap \g h_\ood^\R.
\end{equation}
Suppose that the signature of $B_{\lambda,\tau}$ is $(p,q)$, i.e., the quadratic form ${B}_{\lambda,\tau}(x,x)$ is equivalent to $\sum_{k=1}^px_k^2-\sum_{k=1}^q x_{p+k}^2$. Then 
$\tau$ induces an antilinear involution on $\mathrm{Cl}_r(\C)$, whose fixed subalgebra is isomorphic to the real Clifford algebra $\mathrm{Cl}_{p,q}(\R)$.

\begin{prp}
\label{prp:ETClp,q}
Let $[W]\in\Irr(\catC_\g h)$ and assume that $W$ has  central character $\lambda\in\g h_\eev^*$.
\begin{itemize}
\item[\rm (i)] If $\lambda\neq \oline{\lambda\circ\tau}$, then $\EType_{\g h,\tau}(W)=r_\C$ for $r:=\mathrm{rank}(B_\lambda)\mod 2$. 
\item[\rm (ii)]
If $\lambda=\oline{\lambda\circ\tau}$, then $\EType_{{\g h,\tau}}(W)=t_\R$ for 
$t:=q-p\mod 8$, where $(p,q)$ is the signature of $B_{\lambda,\tau}$.  
\end{itemize}
\end{prp}
\begin{proof}
Set $s:={\lfloor\frac{r-1}{2}\rfloor}$. As superalgebras we have 
\begin{equation}
\label{eq:Cl_r}
\mathrm{Cl}_r(\C)\cong \mathrm{Cl}_1(\C)^{\otimes r}\cong\mathrm{Q}(1)^{\otimes r}\cong
\begin{cases}
\mathrm{Mat}_{2^s|2^s}(\C)&r\text{ even},\\
\mathrm{Q}(2^s)&r\text{ odd}.
\end{cases}
\end{equation}
It follows that $W\cong \Pi W$ if and only if $r$ is odd. This proves (i).

In the rest of the proof we assume that $W:=\C^{1|1}$ is the unique irreducible $\mathrm{Cl}_1(\C)$-module. Note that
$\mathrm{Cl}_1(\C)\cong \mathrm{Q}(1)$ and  
$
\mathrm{Cl}_{p,q}(\R)\cong \mathrm{Cl}_{1,0}(\R)^{\otimes p}\otimes \mathrm{Cl}_{0,1}(\R)^{\otimes q}
$. The real forms $\mathrm{Cl}_{1,0}(\R)$ and $\mathrm{Cl}_{0,1}(\R)$ of $\mathrm{Q}(1)$ correspond to the antilinear involutions
\[
\tau_\pm:\mathrm{Q}(1)\to\mathrm{Q}(1)
\ ,\
a+b\iota\mapsto \bar a\pm \bar b\iota.
\]
In both cases we have $\Pi,(\funcB,+)\in\mathfrak S_{W}$, with the corresponding isomorphisms $\phi_\pm : W\to \funcB W$ given by  $z\oplus w\mapsto \bar z\oplus\pm  \bar w$. Now let $\pi:W\to \Pi W$ be the parity switch and set $\phi:=\pi\circ \phi_\pm$. We denote the action map by $\rho:\g h\to \gl(W)$. Then $\phi:W\to \Pi \funcB W$ is an antilinear map that satisfies $\phi\rho(\tau(x))=\rho(x)\phi$ for $x\in\mathrm{Q}(1)$, and we need to compute $\bfc(\phi)$. We have $\pi\circ \phi_\pm \circ \pi=\pm \phi_\pm$, hence $\phi^2=\pm \phi_\pm \circ \phi_\pm=\pm 1_W$, where the sign $\pm $ agrees with  the subscript of $\tau_\pm$.  This proves that   
\begin{equation}
\label{eq:ETypeWCliff}
\EType_{\mathrm{Q}(1),\tau}(W)=
\begin{cases}
7_\R& \text{ if }\tau=\tau_+,\\
1_\R&\text{ if }\tau=\tau_-.
\end{cases}
\end{equation}
The assertion of the proposition follows from~\eqref{eq:ETypeWCliff} and Proposition~\ref{prp:EtypeofWWAWB}. 
\end{proof}

\subsection{Some general Lie theory notation}
\label{subsec:generalLie}
 Given a Cartan subalgebra $\g h\sseq \g g$, if $\g g$ has a root space decomposition
with respect to $\g h_\eev$ then  we denote the corresponding root system of $\g g$ by $\Phi_\g g$. 
We suppress $\g h $ in this notation as it will always be clear from the context. The root space of $\g g$ associated to $\alpha\in\Phi_\g g$ will be denoted $\g g_\alpha$.

If $\g g_\eev$ is reductive, then  $\Gund$ denotes the connected and simply connected complex semisimple algebraic group with Lie algebra $\gund:=[\g g_\eev,\g g_\eev]$, and 
$\Hund$ denotes the Cartan subgroup of $\Gund$ with Lie algebra $\g h_\eev\cap \gund$. The restriction of 
$\tau$  to $\gund$ can be integrated to an  anti-holomorphic
involution of $\Gund$, which we will still denote by $\tau$. 
Every finite dimensional $\gund$-module $(\rho,W)$ can be integrated to a representation of $\Gund$, which we still denote by $\rho$. Every $\lambda\in\hund^*$ that occurs as a weight in $(\rho,W)$ is then integrated to a multiplicative character $\lambda:\Hund\to \C^*$. The integrated adjoint representation of $\gund$ on $\g g$ is denoted $\Ad:\Gund\to\mathrm{GL}(\g g)$. 

Suppose that $\g g$ is either of basic type or of type $\mathbf Q(n)$. Any positive system $\Phi_\g g^+$ of $\Phi_\g g$ corresponds to a Borel subalgebra   $\g h\sseq \g b\sseq \g g$. 
We denote the fundamental system of $\Phi_\g g^+$ by $\Pi_\g b$. 
For  $w\in N_{\Gund}(\Hund)$, we denote the positive system of the Borel $\Ad_w\g b$ by $w\cdot \Pi_\g b$. 

Suppose that $\g g$ is of basic type (so that $\g h=\g h_\eev$). If $\alpha\in\Pi_\g b$ is isotropic, we use 
$r_{\g g,\alpha}\cdot \Pi_{\g b}$ to denote the fundamental system obtained from $\Pi_\g b$ by the odd reflection associated to  $\alpha$. 
Thus, $r_{\g g,\alpha}\cdot \Pi_{\g b}$ corresponds to the positive system $(\Phi_\g g^+\cup\{-\alpha\})\backslash\{\alpha\}$. \begin{rmk}
\label{rmk:hwodd-sl22}
We continue with $\g g$ of basic type. Let $\g b_{\alpha}$ denote the Borel subalgebra corresponding to $r_{\g g,\alpha}\Pi_\g b$, where $\alpha$ is an isotropic root.
If $\g g$ is not isomorphic to $\g{sl}(2|2)$ or $\g{psl}(2|2)$, then the subalgebra $\g g_\alpha\oplus \g g_{-\alpha}\oplus [\g g_\alpha,\g g_{-\alpha}]$ is isomorphic to $\g{sl}(1|1)$. This implies that  for  an irreducible $\g g$-module $W$ with $\g b$-highest weight $\lambda\in\g h^*$,  the $\g b_\alpha$-highest weight of $W$ is
\[
\lambda_\alpha:=\begin{cases}
\lambda& \text{ if }\lambda([\g g_\alpha,\g g_{-\alpha}])=0,\\
\lambda-\alpha&\text{ otherwise.}
\end{cases}
\] 
However,  for $\g g=\g{sl}(2|2)$ or $\g{psl}(2|2)$, the subalgebra 
$\g g_\alpha\oplus \g g_{-\alpha}\oplus [\g g_\alpha,\g g_{-\alpha}]$ is isomorphic to $\g{sl}(1|1)\oplus\g{sl}(1|1)$ or $(\g{sl}(1|1)\oplus\g{sl}(1|1))/\C$.
In these cases, the computation of the $\g b_\alpha$-highest weight of $W$ is explained in Proposition~\ref{prp:hwsl22}.\end{rmk}

\subsection{Odd reflections and the Harish-Chandra image}
In this subsection $\g g$ will be a Lie superalgebra that is either of basic type or of type $\mathbf Q(n)$. 
\begin{rmk}
\label{rmk-smallcases}
The Lie superalgebras $\g{sl}(1|1)$, $\g{psl}(1|1)$, $\g{q}(1)$, $\g{pq}(1)$ and  $\g{sq}(1)$  are of Clifford type and each of them has a unique real form (namely, the trivial one). Thus, in these  cases endotypes of irreducible modules are determined by Proposition~\ref{prp:ETClp,q}.   
\end{rmk}
In the rest of this section, we assume that $\g g$ is \emph{not} isomorphic to the Lie superalgebras of Remark~\ref{rmk-smallcases}.  
Let $\g h$ be a Cartan subalgebra of $\g g$.
Henceforth we assume that  $\tau(\g h)=\g h$, or equivalently, $\g h$ is the complexification of a Cartan subalgebra of $\g g^\R$. It is possible to modify our results slightly to be applicable to all Cartan subalgebras $\g h$ (all one needs is to take conjugacy of $\tau(\g h)$ and $\g h$ into consideration), but this will be at the expense of more complicated final statements.

Let 
$\g b=\g h\oplus\g n$ be a Borel subalgebra of $\g g$. We have $\Phi_{\g g}\neq \varnothing$. In Lemma~\ref{lem:hwB(W)}, $\oline{\lambda\circ\tau}$ means the linear map obtained by complex conjugation of the antilinear map $\lambda\circ\tau$.
\begin{lem}
\label{lem:hwB(W)}
Let $[W]\in\Irr(\catC_\g g)$ and assume that the $\g b$-highest weight of $W$ is $\lambda\in\g h_\eev^*$. Then the $\tau(\g b)$-highest weight of $\funcB_{\g g,\tau}(W)$ is $\oline {\lambda\circ\tau}$.
\end{lem}
\begin{proof}
Follows from Remark~\ref{remark:antilin-intertwiner}.\end{proof}

By Theorems~\ref{thm:Irr=Irr}
 and~\ref{thm:Table}, 
the computation  of endotypes in $\Irr(\catR_{\g g,\tau})$ is tantamount to the following two steps:
\medskip
\begin{description}\setlength\itemsep{.5em}
\item[Step 1] Computing the $\g b$-highest weights 
that lie in the same $\Theta_{\funcB_{\g g,\tau}}$-orbit.
\item[Step 2] Computing the symmetry datum of each $[W]\in \Irr(\catC_{\g g})$ from its $\g b$-highest weight. 
\end{description}
\medskip
Step 1 is fairly straightforward and will be accomplished in Proposition~\ref{prp:Step1}. Step 2 is substantially more delicate and will be carried out in 
Theorems~\ref{thm-endotypes-basic} and~\ref{thm-endotypes-Q(n)}.

In the following lemma (as well as  the other assertions of this subsection), when  $\g g$ is of type $\mathbf Q(n)$ we assume that  $k=0$, i.e., no odd reflections are involved. 
\begin{lem}
\label{lem:seqisotr}
There exist $w\in N_{\Gund}(\Hund)$ and isotropic roots $\alpha_1,\ldots,\alpha_k\in\Phi_\g g$ such that   $w^{-1}\tau(w^{-1})\in \Hund$ and
\begin{equation}
\label{eq:taub-and-b-again}
\Pi_{\tau(\g b)}=
w\cdot (r_{\g g,\alpha_k}\cdot (\cdots( r_{\g g,\alpha_1}\cdot \Pi_{\g b}))).
\end{equation}
\end{lem}
\begin{proof}
$\tau(\g b_\eev)$ and $\g b_\eev$ are Borel subalgebras of $\g g_\eev$, therefore they are conjugate by an  element of the Weyl group $N_{\Gund}(\Hund)/\Hund$. This proves that 
\begin{equation}
\label{eq:ADw-1}
\Ad_{w^{-1}}(\tau(\g b_\eev))=\g b_\eev\quad\text{for some }w\in N_{\Gund}(\Hund).
\end{equation}
The two Borel subalgebras $\Ad_{w^{-1}}(\tau(\g b))$ and $\g b$ of $\g g$ have identical even parts, therefore they are related by a sequence of odd reflections. This completes the proof of~\eqref{eq:taub-and-b-again}. 

Next we prove $w^{-1}\tau(w^{-1})\in\Hund$. Taking $\tau$ of both sides of~\eqref{eq:ADw-1} 
yields
$
\Ad_{\tau(w^{-1})}(\g b_\eev)=\tau(\g b_\eev)
$, 
hence
\[
\Ad_{w^{-1}\tau(w^{-1})}(\g b_\eev)=\Ad_{w^{-1}}(\Ad_{\tau(w^{-1})}\g b_\eev)=\g b_\eev.
\]
We have $\tau(\Hund)=\Hund$, hence $w^{-1}\tau(w^{-1})\in N_{\Gund}(\Hund)$. Since the Weyl group acts simply transitively on positive systems, we must have $w^{-1}\tau(w^{-1})\in \Hund$. 
\end{proof}
When $\g g$ is of basic type and $\g g_\ood\neq 0$, we will need 
the Harish-Chandra projection
\[
\HC_{\g b}:\mathfrak U(\g g)\to\mathfrak  U(\g h)
\] 
corresponding to the triangular decomposition
$\g g=\bar{\g n}\oplus\g h\oplus \g n$, where $\bar{\g n}$ denotes the nilradical opposite to $\g n$.  Thus $\HC_\g b(\mathfrak  U(\g g)\g n+\bar{\g n}\mathfrak  U(\g g))=0$. 

Our next goal is to define a weight $\lambda_\funcB\in \g h^*$ and a scalar $c_\lambda\in \R^*$ that are needed in Theorem~\ref{thm-endotypes-basic}. Because of 
Remark~\ref{rmk:hwodd-sl22}, first we address the case where $\g g$ is \emph{not} isomorphic to $\g{sl}(2|2)$ or $\g{psl}(2|2)$, and then we explain the necessary modifications for the latter two cases in Remark~\ref{rmk-lamBcsl22}.

Suppose that $\g g$ is  not isomorphic to $\g{sl}(2|2)$ or $\g{psl}(2|2)$. 
 For every isotropic root $\alpha\in\Phi_\g g$ (if they exist), we choose a root vector   $e_{\alpha}\in\g g_{\alpha}$. Suppose that $\alpha_1,\ldots,\alpha_k$ is the sequence of isotropic roots in Lemma~\ref{lem:seqisotr} (note that $k\geq 1 $ if and only if  $\g g$ is of basic type).  
Set $h_{\alpha_i}:=[e_{\alpha_i},e_{-\alpha_i}]$ for $1\leq i\leq k$. 
Given $\lambda\in \g h^*$, we define a subsequence $\alpha_{i_1},\ldots,\alpha_{i_r}$, where $1\leq i_1<\cdots< i_r\leq k$, inductively as follows. Having chosen $i_1<\ldots<i_{j-1}$,  we set \begin{equation}
\label{eq:lambdaseq}
i_j:=\min\left\{ i\,:\,
i_{j-1}<i\leq k\text{ and }
\lambda\left(h_{\alpha_i}\right)\neq \sum_{s=1}^{j-1}\alpha_{i_s}\left(h_{\alpha_{i}}\right)\right\},
\end{equation}
with the convention that $i_0=0$. Having found  the subsequence $i_1,\ldots,i_r$, we define 
\begin{equation}
\label{eqDlam}
D_\lambda:=e_{\alpha_{i_1}}\cdots e_{\alpha_{i_r}}(\Ad_{w^{-1}}^{}\tau(e_{\alpha_{i_1}}))\cdots (\Ad_{w^{-1}}^{}\tau(e_{\alpha_{i_r}}))\in \mathfrak  U(\g g).
\end{equation}
We extend $\lambda$ canonically to an algebra homomorphism $S(\g h)\to\C$. 
Finally, we define
\begin{equation}
\label{eq:lambdaB-clam}
\lambda_\funcB:=
\Ad^*_w\left(\lambda-\sum_{j=1}^r\alpha_{i_j}
\right)\quad\text{and}\quad
c_\lambda:=\lambda(\HC_\g b(D_\lambda))\lambda(w^{-1}\tau(w^{-1})),
\end{equation}
where $(\Ad^*_w\mu)
(h)=\mu(\Ad_{w^{-1}}h)$ for $\mu\in\g h_\eev^*$ and $h\in\g h_\eev$. Note that $\lambda(w^{-1}\tau(w^{-1}))$ means the evaluation of the exponentiated character $\lambda:\Hund\to\C^*$ at $w^{-1}\tau(w^{-1})$.
\begin{rmk}
We remind the reader that whenever $r=0$ (e.g., $\g g=\g g_\eev$ or  $\g g$ is of type $\mathbf Q(n)$),  we have $\lambda_\funcB=\Ad_w^*(\lambda)$ and $c_\lambda=\lambda(w^{-1}\tau(w^{-1}))$.
\end{rmk}

\begin{rmk}
\label{rmk-lamBcsl22}
When $\g g$ is isomorphic to $\g{sl}(2|2)$ or $\g{psl}(2|2)$, 
in~\eqref{eq:taub-and-b-again} we have $k\leq 1$. 
 If $k=1$, we define $e_i,f_i,h_i$ for $i\in\{1,2\}$ and $\omega_{\lambda,\alpha_1}$ as in Appendix~\ref{App-sl22}. We set $\lambda_\funcB:=\Ad_w^*\left(\lambda-\mathrm{rank}(\omega_{\lambda,\alpha_1})\alpha_1\right)$ and we define $e_{\alpha_1}\in \mathfrak U(\g g)$ by 
 \[
e_{\alpha_1}=\begin{cases}
1 & \text{ if }\mathrm{rank}(\omega_{\lambda,\alpha_1})=0,\\
e_i & \text{ if }\mathrm{rank}(\omega_{\lambda,\alpha_1})=1\text{ and }\lambda(h_i)\neq 0,\\
e_1e_2& \text{ if }\mathrm{rank}(\omega_{\lambda,\alpha_1})=2.
\end{cases}
 \] 
Correspondingly, we set $e_{-\alpha_1}=1$ or $f_i$ or $f_1f_2$. We define $D_\lambda$ and $c_\lambda$ as in~\eqref{eqDlam} and~\eqref{eq:lambdaB-clam}.  \end{rmk}

Let $(\rho,W)$ be an irreducible $\g g$-module with $\g b$-highest weight $\lambda\in\g h_\eev^*$.  By Lemma~\ref{lem:hwB(W)},  the $\tau(\g b)$-highest weight of $\funcB_{\g g,\tau}(W)$ is $\oline{\lambda\circ\tau}$.
Furthermore, 
if $W_\lambda$ denotes the $\g b$-highest weight space of $W$, 
then the $\tau(\g b)$-highest weight space of $W$ is 
\begin{equation}
\label{eq:WlamW'lam}
W'_\lambda:=
\rho(w)\rho(e_{-\alpha_{i_r}})\cdots \rho(e_{-\alpha_{i_1}})W_\lambda,
\end{equation}
and the  $\tau(\g b)$-highest weight of $W$ is $\lambda_\funcB$.
Thus, $\lambda_\funcB=\oline{\lambda\circ\tau}$ if and only if $\wHom_{\catC_\g g}(W,\funcB_{\g g,\tau}W)\neq 0$. This immediately implies the following statement, which addresses Step 1 above.
\begin{prp}
\label{prp:Step1}
Let $W$ and $W'$ be irreducible finite dimensional $\g g$-modules with $\g b$-highest weights $\lambda,\lambda'\in\g h_\eev^*$. Then up to parity $W$ and $W'$ are in the same $\Theta_{\g g,\tau}$-orbit if and only if $\lambda_\funcB=\oline{\lambda'\circ\tau}$.  
\end{prp}
In Lemma~\ref{lem:psiphitrick} below, let $\psi\in\wEnd_{\SVect}(\funcB_{\g g,\tau}W)$ be defined by
\[
\psi:=\rho(e_{\alpha_{i_1}})\cdots \rho(e_{\alpha_{i_r}})\rho(w^{-1})
.
\] 
Thus, $\psi$ is a $\C$-linear endomorphism of the vector superspace $\funcB_{\g g,\tau}W$. In the rest of this section, for 
any $\phi\in\wHom_{\catC_\g g}(W,\funcB_{\g g,\tau}W)$ we consider
$\psi\phi$ as an element of 
$\wHom_\SVect(W,\funcB_{\g g,\tau}W)$.
\begin{lem}
\label{lem:psiphitrick}
Let $[W]\in\Irr(\catC_\g g)$  with  $\g b$-highest weight $\lambda$ such that $\lambda_\funcB=\oline{\lambda\circ\tau}$. Also, let $\phi\in\wHom_{\catC_\g g}(W,\funcB_{\g g,\tau}W)$ be homogeneous and nonzero. 
Then the following  hold.
\begin{itemize}
\item[\rm (i)] $\phi(W_\lambda)=W_\lambda'$, $\phi(W_\lambda')=W_\lambda$, and $\psi(W_\lambda')=W_\lambda$. 
\item[\rm (ii)] $(\psi\phi)^2\big|_{W_\lambda}=c_\lambda\phi^2\big|_{W_\lambda}$. Furthermore, $c_\lambda\in\R^*$.
\end{itemize}
\end{lem}

\begin{proof} 
As before let $\rho:\g g\to \g{gl}(W)$ denote the action map. By Remark~\ref{remark:antilin-intertwiner} we can think of $\phi$ as an antilinear map $\phi:W\to W$ satisfying 
\begin{equation}
\label{eq:phitaurho}
\phi\rho(x)=\rho(\tau(x))\phi\quad\text{ for }x\in\g g.
\end{equation} 
From~\eqref{eq:phitaurho} it follows that $\phi(W_\lambda)=W_\lambda'$ and 
$\phi(W_\lambda')=W_\lambda$. This proves (i).
Similar to~\eqref{eq:WlamW'lam}, 
we have
$
W_\lambda=\rho(e_{\alpha_{i_1}})\cdots \rho(e_{\alpha_{i_r}})\rho(w^{-1})W_\lambda'=\psi(W_\lambda')$.
Equality~\eqref{eq:phitaurho} holds for the integrated representation $\rho$ and 
$x\in \Gund$. Thus,
\begin{align*}
\phi\psi
&=\phi\rho(e_{\alpha_{i_1}})\cdots
\rho(e_{\alpha_{i_r}})
\rho(w^{-1})=
\rho(\tau(e_{\alpha_{i_1}}))
\cdots
\rho(\tau(e_{\alpha_{i_r}}))
\rho(\tau(w^{-1}))\phi.
\end{align*}
Hence
\begin{align*}(\psi\phi)^2&=
\psi(\phi\psi)\phi=
\rho(e_{\alpha_{i_1}})\cdots \rho(e_{\alpha_{i_r}})
\rho(w^{-1})
\rho(\tau(e_{\alpha_{i_1}}))
\cdots
\rho(\tau(e_{\alpha_{i_r}}))
\rho(\tau(w^{-1}))\phi^2\\
&=
\rho(e_{\alpha_{i_1}})\cdots \rho(e_{\alpha_{i_r}})
\rho(\Ad_{w^{-1}}\tau(e_{\alpha_{i_1}}))
\cdots
\rho(\Ad_{w^{-1}}\tau(e_{\alpha_{i_r}}))\rho(w^{-1})\rho(\tau(w^{-1}))
\phi^2\\
&=\rho(D_\lambda)\rho(w^{-1}\tau(w^{-1}))\phi^2.
\end{align*}
Note that $\phi^2(W_\lambda)=W_\lambda$
and $\rho(w^{-1}\tau(w^{-1}))\big|_{W_\lambda}=\lambda\big(w^{-1}\tau(w^{-1})\big)\cdot 1_{W_\lambda}$. 

If $\g g$ is of type $\mathbf Q(n)$ then $r=0$  and 
we have 
\[
\lambda(w^{-1}\tau(w^{-1}))=
\lambda\circ\tau(\tau(w^{-1})w^{-1})=
\oline{\lambda(w^{-1}(\tau(w^{-1})w^{-1})w)}
=\oline{\lambda(w^{-1}\tau(w^{-1}))},
\]
hence $\lambda(w^{-1}\tau(w^{-1}))\in \R^*$. This proves $c_\lambda\in\R^*$.  

If $\g g$ is not of type $\mathbf Q(n)$ then we express $D_\lambda$ as 
\[
D_\lambda=D_{\bar{\g  n}}+\HC_\g b(D_\lambda)+D_{\g n}\quad\text{where}\quad D_{\g n}\in\mathfrak  U(\g g)\g n\text{ and }
D_{\bar{\g  n}}\in \bar{\g n}\mathfrak  U(\g h).
\]
Since $W_\lambda$ is the $\g b$-highest weight space we must have $\rho(D_{\g n})W_\lambda=0$. Also, 
$\rho(D_{\bar {\g n}})W_\lambda$
is in the sum of weight spaces of $W$ with weight  $\lambda-\beta$ where $\beta\neq 0$ and $\beta=\sum_{\alpha\in \Pi_{\g b}}n_\alpha \alpha$ for some $n_\alpha\geq 0$. 
But we also have $\psi\phi(W_\lambda)=W_\lambda$. This proves that $\rho(D_{\bar{\g n}})W_\lambda=0$. 
Consequently, 
\[
(\psi\phi)^2|_{W_\lambda}=\rho(\HC_{\g b}(D_\lambda))\lambda(w^{-1}\tau(w^{-1}))\phi^2|_{W_\lambda}. 
\]
 But  by the definition of $\phi$ we have $\phi^2\big|_{W_\lambda}=\bfc(\phi)1_{W_\lambda}$ where $\bfc(\phi)\in\R^*$ (see Lemma~\ref{lem:sig}). Since $\dim(W_\lambda)=1$, as in Remark~\ref{rmk-abbel} for the antilinear map $\psi\phi:W_\lambda\to W_\lambda$  we obtain $(\psi\phi)^2=c'1_{W_\lambda}$ for some $c'\in \R^+$. In particular, we must have $c_\lambda\in\R^*$. 
\end{proof}
\begin{thm}
\label{thm-endotypes-basic}
Let $\g g$ be a Lie superalgebra of basic type, let $\tau:\g g\to\g g$ be an antilinear involution, and let $\g b=\g h\oplus \g n$ be a Borel subalgebra of $\g g$ such that $\tau(\g h)=\g h$. Let $[W]\in\Irr(\catC_\g g)$. Let the $\g b$-highest weight of $W$ be $\lambda\in\g h^*$, and let $\lambda_\funcB$, $c_\lambda$, and $r$  be  as in~\eqref{eq:lambdaB-clam}.
\begin{itemize}
\item[\rm (i)] If $\lambda_\funcB\neq \oline{\lambda\circ\tau}$, then $\EType_{\g g,\tau}(W)=0_\C$.
\item[\rm (ii)] If $\lambda_\funcB= \oline{\lambda\circ\tau}$, then $c_\lambda\in \R^*$ and $\EType_{\g g,\tau}(W)$ is determined from the sign of $c_\lambda$ and parity of $r$ according to Table~\ref{Table-ETBasic} below.

\end{itemize}
\begin{table}[ht]
\centering
\begin{tabular}{|c|c|c|}
\hline 
$r$ & $\sgn(c_\lambda)$ & $\EType_{\g g,\tau}(W)$\\
 \hline 
\rm{even} & $+$ & $0_\R$\\
 \hline \rm{odd} & $-$ & $2_\R$\\
 \hline \rm{even} & $-$ & $4_\R$\\
 \hline \rm{odd} & $+$ & $6_\R$\\
 \hline
\end{tabular}
\vspace{3mm}
\caption{\label{Table-ETBasic} $\EType(W)$ when $\g g$ is a basic Lie superalgebra. }
\vspace{-8mm}
\end{table}
\end{thm}

\begin{proof} 
If $\lambda_\funcB\neq\oline{\lambda\circ\tau}$, then $\wHom_{\catC_\g g}(W,\funcB_{\g g,\tau}W)=0$ and 
$\Pi\not\in \mathfrak S_W$ (because $\dim W_\lambda=1$). This implies
(i). For (ii), note that   
$\wHom_{\catC_\g g}(W,\funcB_{\g g,\tau}W)\cong \C^{1|0}$ or $\C^{0|1}$, again because $\Pi\not\in \mathfrak S_W$. 
Furthermore, the  parity of the (unique up to scalar) nonzero elements of 
$\wHom_{\catC_\g g}(W,\funcB_{\g g,\tau}W)$ is the same as the number of odd reflections that result in changing the parity of the highest weight, i.e, it equals the parity of $r$.
Now suppose that $\phi\in\wHom_{\catC_\g g}(W,\funcB_{\g g,\tau}W)$ is a nonzero homogeneous element. 
Since $\phi^2(W_\lambda)=W_\lambda$, to determine $\sgn(\phi)$ it suffices to consider its restriction to $W_\lambda$. 
Since $\dim W_\lambda=1$, the antilinear map $\psi \phi:W_\lambda\to W_\lambda$ is of the form $z\mapsto \alpha \bar z$ for some  $\alpha\in\C^*$. It follows that $\sgn(\psi\phi)=+$, hence by Lemma~\ref{lem:psiphitrick}(ii) we have $\sgn(\phi)=\sgn(c_\lambda)$. Putting together these facts and Table~\ref{Table-1}, we obtain (ii). 
\end{proof}
\begin{rmk}
\label{rmk:specialBor}
For special choices of the Borel subalgebra $\g b\sseq \g g$, the number of odd reflections   that relate $\Pi_{\g b}$ to $\Pi_{\tau(\g b)}$ can be small and, as a result, the calculation of  $\HC_\g b(D_\lambda)$ can be trivial. For instance, for real forms $\g{su}(p,q|r,s)$ of $\g{sl}(p+q,r+s)$ where $0\leq p\leq q$ and $0\leq r\leq s$, and $\lfloor \frac{q-p}{2}\rfloor \geq \lfloor\frac{s-r}{2}\rfloor$, we can consider the Borel subalgebra corresponding to the 
$\eps\delta$-sequence $\eps^{2p}\omega \delta^{2r}$ where
\[
\omega:=\begin{cases}
\eps^{k-l}(\eps\delta)^{l}
(\delta\eps)^{l}
\eps^{k-l} &\text{if }q-p=2k,s-r=2l.\\
\eps^{k-l}\delta(\eps\delta)^{2l}\eps^{k-l} &\text{if }q-p=2k,\ s-r=2l+1.\\
\eps^{k-l}(\eps\delta)^{2l}\eps^{k-l+1} &\text{if }q-p=2k+1,\ s-r=2l.\\
\eps^{k-l}(\eps\delta)^{l+1}(\delta\eps)^l\eps^{k-l}&\text{if }q-p=2k+1,\ s-r=2l+1.
\end{cases}
\]
For this $\g b$, in the first three cases no odd reflections are required (hence $D_\lambda=1$), and in the fourth case there is only one odd reflection corresponding to $\alpha=\eps_{k+1}-\delta_{l+1}$. Thus,  $D_\lambda=1$ if $\lambda([e_\alpha,\tau(e_\alpha)])=0$, and $D_\lambda=[e_\alpha,\tau(e_\alpha)]$ otherwise. Special Borel subalgebras of this kind have been considered in~\cite[Sec. 5]{Hayashi}.

\end{rmk}

We now characterize endotypes when $\g g$ is of type $\mathbf Q(n)$. 
Let $\tau_w:\g h\to\g h$ be defined by 
\[
\tau_w(x)=\Ad_{w^{-1}}(\tau(x))\quad\text{for }x\in\g h.
\] 
Note that $\tau_w$ is an antilinear involution of $\g h$ because by Lemma~\ref{lem:seqisotr}  we have
\begin{align*}
\tau_w^2(x)
=\Ad_{w^{-1}}\big(\tau(\Ad_{w^{-1}}(\tau(x))\big)
&=
\Ad_{w^{-1}}
\big(
\Ad_{\tau(w^{-1})}x
\big)
=\Ad_{w^{-1}\tau(w^{-1})}(x)=x
.\end{align*}

\begin{thm}
\label{thm-endotypes-Q(n)}
Let $\g g$ be of type $\mathbf Q(n)$,
let $\tau:\g g\to\g g$ be an antilinear involution, and let $\g b=\g h\oplus \g n$ be a Borel subalgebra of $\g g$ such that $\tau(\g h)=\g h$. Let $[W]\in\Irr(\catC_\g g)$  and let the  $\g b$-highest weight of $W$ be $\lambda\in\g h_\eev^*$.
Let $\lambda_\funcB$ and $c_\lambda$ be as in~\eqref{eq:lambdaB-clam}, and 
let $B_{\lambda,\tau_w}$ be  as in~\eqref{eq:Blatau}. 

\item[\rm (i)] Suppose that $\lambda\neq \oline{\lambda\circ\tau_w}$. Then $\EType_{\g g,\tau}(W)=t_\C$ for $t:=\mathrm{rank}(B_{\lambda,\tau_w})\mod 2$. 

\item[\rm (ii)] Suppose that $\lambda= \oline{\lambda\circ\tau_w}$. Then 
$\EType_{\g g,\tau}(W)=t_\R$ for $t:=(q-p+s)\mod 8$, where $(p,q)$ denotes the signature of $B_{\lambda,\tau_w}$, and $s$ is defined by
\[
s:=\begin{cases}
0 & \text{ if }c_\lambda>0,\\
4 & \text{ if }c_\lambda<0.
\end{cases}
\] 
\end{thm}
\begin{proof}
Part (i) is proved as in Theorem~\ref{thm-endotypes-basic}(i), using Proposition~\ref{prp:ETClp,q}(i). Note that the given constraint on $\lambda$ is equivalent to $\lambda_\funcB\neq\oline{\lambda\circ\tau}$.  For part (ii), we use Lemma~\ref{lem:psiphitrick}. Denoting (the exponentiation to $\Gund$ of) the action map  by $\rho$, we have $\psi=\rho(w^{-1})$ and from Remark~\ref{remark:antilin-intertwiner} it follows that 
$\psi\phi\rho(x)=\rho(\tau_w(x))\psi\phi$  for $x\in\g h$. Thus, 
we have a parity-preserving $\C$-linear isomorphism
\[
\wHom_{\catC_\g g}(W,\funcB_{\g g,\tau}W)
\cong \wHom_{\catC_\g h}(W,\funcB_{\g h,\tau}W)\to 
\wHom_{\catC_{\g h}}(W_\lambda,\funcB_{\g h,\tau_w}W_\lambda)
\ ,\ 
\phi\mapsto (\psi\phi)\big|_{W_\lambda}.
\]  
Lemma~\ref{lem:psiphitrick}(ii) implies that 
$
\sgn(\phi)=\sgn(c_\lambda)\sgn\big(\psi\phi\big|_{W_\lambda}\big)$. 
From this and Table~\ref{Table-1} we obtain
\begin{equation}
\label{eq:ETYpe4R}
\EType_{\g g,\tau}(W)=\begin{cases}
\EType_{\g h,\tau_w}(W_\lambda)&\text{ if }c_\lambda>0,\\
\EType_{\g h,\tau_w}(W_\lambda)\bullet 4_\R&\text{ if }c_\lambda<0.\\
\end{cases}
\end{equation}
Part (ii) follows from~\eqref{eq:ETYpe4R} and Proposition~\ref{prp:ETClp,q}(ii).
\end{proof}

\subsection{Further remarks on endotypes when  $\g g$ is of basic type}
When  $\g g=\g g_\eev$, for any $[W]\in\Irr(\catC_\g g)$ we have  
$\EType_{\g g,\tau}(W)\in\{0_\R,0_\C,4_\R\}$, corresponding to the division algebras $\R,\C,\qH$. 

For $\g g$ of basic type, the highest weight space of $W$ is one-dimensional and thus $\Pi\not \in \mathfrak S_W$.
Consequently,  $\EType_{\g g,\tau}(W)\in\{0_\R,2_\R,4_\R,6_\R, 0_\C\}$. However, the endotypes $2_\R$ and $6_\R$ can only occur for the following real forms:
\begin{itemize}
\item[(i)] Unitary real forms of $\g{gl}(m|n)$, $\g{sl}(m|n)$, and $\g{psl}(m|n)$, where $m$ and $n$ are both odd (i.e., the algebras $\g{u}(p,m-p|q,n-q)$ and their $\g{su}$ and $\g{psu}$ variants). 

\item[(ii)] Real forms of $\gl(n|n)$,
$\g{sl}(n|n)$, and $\g{psl}(n|n)$,  corresponding to $\tau(X)=(\bar X)^\Pi$ 
where 
\[
X^\Pi:=\begin{bmatrix}
D & C\\
B & A
\end{bmatrix}
\text{ for }X=\begin{bmatrix}
A& B\\C & D\end{bmatrix}\in\g{gl}(n|n).
\]
Due to the resemblance of the matrix realization of the latter real form of $\g{gl}(n|n)$  to the  Lie superalgebra $\g{q}(n)$,
we denote this real form by $\oline{\g q}(n)$.
\end{itemize}

\begin{table}[ht]
\renewcommand{\arraystretch}{1.4}
\begin{tabular}{|c|c|c|}
\hline
$\g g^\R$ & $\g l^\R$ & \makecell{$2_\R$ or $6_\R$\\ possible?}\\
\hline 
$\g{gl}(m|n,\R)$ & $\R^{m+n}$ & No\\

\hline 
\makecell{$\g{u}(p,q|r,s)$, $p\leq q$, $r\leq s$\\ $(p+q)(r+s)$ even} & $\g u(q-p,0|s-r,0)\oplus \R^{2p+2r}$ & No\\

\hline 
\makecell{$\g{u}(p,q|r,s)$, $p\leq q$, $r\leq s$\\ $(p+q)(r+s)$ odd} & $\g u(q-p,0|s-r,0)\oplus \R^{2p+2r}$ & Yes\\

\hline $\g u^*(2m|2n)$ & $\g{u}(2)^{m+n}$ & No\\

\hline $\oline{\g q}(n)$ &  $\oline{\g q}(1)^{\oplus n}$ & Yes\\

\hline 
$\oline{\g{pe}}(n)$, $n\geq 2$ & $\R^{n}$ & No\\

\hline 
$\g{osp}(p,q|2n)$, $p\leq q$ & $\g{so}(q-p,\R)\oplus \R^{p+n}$ & No\\

\hline 
$\g{osp}^*(2m|p,q)$, $p\leq q$, $m=2r$ & $\g{u}(2)^{\oplus (p+r)}\oplus \g{sp}(q-p,0)$ & No\\
\hline $\g{osp}^*(2m|p,q)$, $p\leq q$, $m=2r+1$ & $\g{u}(2)^{\oplus (p+r)}\oplus \g{osp}^*(2|q-p,0)$ & No\\
\hline 
$G(2|1)$, $\g g_\eev^\R=\g{sl}(2,\R)\oplus G(2)_\mathrm{split}$ & $\R^3$ & No\\
\hline 
$G(2|1)$, $\g g_\eev^\R=\g{sl}(2,\R)\oplus G(2)_\mathrm{compact}$ & $G(2)_\mathrm{compact}$ & No\\
\hline 
$F(3|1)$, $\g g_\eev^\R=\g{sl}(2,\R)\oplus\g{so}(7,\R)$ & $\g{so}(7,\R)$& No \\
\hline 
$F(3|1)$, $\g g_\eev^\R=\g{sl}(2,\R)\oplus\g{so}(3,4)$ & $\R^4$& No\\
\hline 
$F(3|1)$, $\g g_\eev^\R=\g{su}(2)\oplus \g{so}(2,5)$ & $\g{u}(2)\oplus \R\oplus \g{so}(3,\R)$& No\\
\hline 
$F(3|1)$, $\g g_\eev^\R=\g{su}(2)\oplus \g{so}(1,6)$ & $\g{u}(2)\oplus \g{so}(5,\R)$& No\\
\hline 
$D(2|1,\alpha)$, $\alpha\in\R$, $\mathrm{rank}_\R(\g g_\eev^\R)=3$ & $\R^3$& No\\
\hline   
$D(2|1,\alpha)$, $\alpha+\bar \alpha=-1$, $\mathrm{rank}_\R(\g g_\eev^\R)=2$ & $\R^3$& No\\
\hline 
$D(2|1,\alpha)$, $\alpha\in\R$, $\mathrm{rank}_\R(\g g_\eev^\R)=1$ & $\g{su}(2)\oplus \g u(2)$& No\\
\hline 
\end{tabular}
\vspace{3mm}

\caption{\label{Table-hyper} 
Hyperbolic reduction for Lie superalgebras of basic type}
\vspace{-8mm}
\end{table}

For all other Lie superalgebras of basic type, we have
$\EType_{\g g,\tau}(W)\in\{0_\R,0_\C,4_\R\}$. 
This is because in these cases we can find a hyperbolic reduction such that either $\g l=\g l_\eev$ or $\g l$ has a Borel subalgebra $\g b_\g l$ such that $\g b_\g l$ and $\tau(\g b_\g l)$ are conjugate by some $w\in N_{\Gund}(\Hund)$ (i.e., $k=0$ in Lemma~\ref{lem:seqisotr}). Then the claim follows from   
Corollary~\ref{thm:EtypeCplx} or Theorem~\ref{thm-endotypes-basic} (in the latter case, note that $r=0$ since $r\leq k$).  
In Table~\ref{Table-hyper}, we describe the subalgebras $\g l^\R$ of the hyperbolic triangular decomposition obtained by taking a generic element $x_\circ$ of a split maximal torus of $\g g_\eev^\R$. (For brevity, we have excluded real forms of $\g{sl}$ and $\g{psl}$, but the calculations are similar to $\gl$.) Indeed, 
the combination of the reduction to $\g l^\R$,  Corollary~\ref{thm:EtypeCplx}, and Proposition~\ref{prp:EtypeofWWAWB}  can be utilized as an alternate method to prove Theorem~\ref{thm-endotypes-basic}. However, this approach entails tedious  case-by-case analysis. For instance, sometimes $\g l^\R$ has a center that does not split as a summand. Furthermore, when abelian summands are present in $\g l^\R$, one has to determine whether the restriction of the highest weight of $W$ to those summands remains real, or could be complex. Our strategy to prove  Theorem~\ref{thm-endotypes-basic} circumvents such laborious details and provides a uniform, more elegant statement. 

Examples of real forms for which $\g g$ has a Borel $\g b$ such that $\g b$ and $\tau(\g b)$ are conjugate under $N_{\Gund}(\Hund)$ are the split forms (for which $\g b=\tau(\g b)$), the unitary forms  $\g u(p,q|r,s)$ with at least one of $p+q$ and $r+s$ even (see Remark~\ref{rmk:specialBor}) and $\g{osp}^*(2|n,0)$, where $\g b$ is the Borel with  the positive system 
$\{\delta_1-\delta_2,\cdots,\delta_{n-1}-\delta_n,\delta_n-\eps,\delta_n+\eps\}$. We caution the reader that $\g{osp}^*(2m|p,q)$ is a real form of $\g{osp}(2m|2p+2q)$, and  $\oline{\g{pe}}(n)$ denotes the real form of $\g{gl}(n|n)$ corresponding to $\tau(X)=((-\bar X)^\mathrm{str})^\Pi$.

\section{Families of  $\tau$-compatible Borel subalgebras}
\label{sec:tau-compatible-parab}
In this section  $\g g$ will be a Lie superalgebra that is either of basic type or of type $\mathbf Q(n)$. We assume that $\g g$ is not isomorphic to the Lie superalgebras of Remark~\ref{rmk-smallcases}.  
The goal of this section is to explicitly construct families of Borel subalgebras $\g b\sseq \g g$ for which we obtain an alternate formula for  $c_\lambda$ in~\eqref{eq:lambdaB-clam} that is purely Lie algebraic, i.e., it does not require the evaluation of the integrated character $\lambda$ on the  group element $w^{-1}\tau(w^{-1})\in\Hund$. This is established in Theorem~\ref{thm:alternateclam}. In particular, if $\g g=\g g_\eev$, then $c_\lambda$ is obtained by the evaluation of $\lambda$  at a semisimple element of  $[\g g,\g g]$. One interesting outcome  of Theorem~\ref{thm:alternateclam} is that it establishes a connection between computing endotypes and Kostant's cascade of strongly orthogonal roots.

We retain the notation of 
Section~\ref{sec:ETypeWbasic-Q}. 
We denote the centre of a Lie algebra $\g m$ by $\Centre(\g m)$, and the centralizer of a Lie algebra $\g s$ in the Lie algebra $\g m$ by $\Centre_\g m(\g s)$. 

\subsection{Cartan decomposition and compact Cartans}
We begin with some general facts about semisimple Lie algebras with a compact Cartan subalgebra. 
In this subsection $\g g$ will be a semisimple Lie algebra. 
To the  real form $\g g^\R$ of $\g g$ corresponding to the antilinear involution $\tau:\g g\to \g g$, one can associate a \emph{Cartan involution}, i.e., a $\C$-linear involution $\theta:\g g\to\g g$  with the following properties:
\begin{itemize}
\item[(i)] $\theta\tau=\tau\theta$. This implies that $\theta(\g g^\R)=\g g^\R$.
\item[(ii)] The antilinear involution $\theta\tau$ corresponds to the  compact real form of $\g g$. 
\item[(iii)] The symmetric bilinear form  
$\kappa_\theta:\g g^\R\times \g g^\R\to\R$ defined by $\kappa_\theta(x,y):=-\kappa(x,\theta(y))$ is positive definite. Here $\kappa(\cdot,\cdot)$ denotes the Killing form of $\g g^\R$.
\end{itemize}
Let 
$\g k^\R$ and $\g p^\R$ be the $+1$ and $-1$ eigenspaces of $\theta$, respectively. 
The decomposition 
\begin{equation}
\label{eq:Cartandecomp1}
\g g^\R=\g k^\R\oplus\g p^\R 
\end{equation}
is called the Cartan decomposition of $\g g^\R$. Complexifying~\eqref{eq:Cartandecomp1}, we obtain a decomposition
\[
\g g=\g k\oplus \g p.
\]
Let $\g h^\R$ be a $\theta$-stable  Cartan subalgebra of $\g g^\R$, i.e.,  $\theta(\g h^\R)=\g h^\R$. Then  $\g h^\R=\g t^\R\oplus \g a^\R$ where $\g t^\R:=\g h^\R\cap\g k^\R$ and $\g a^\R:=\g h^\R\cap \g p^\R$. 
\begin{dfn}
We say $\g h^\R$ is a \emph{compact Cartan subalgebra} if $\g h^\R=\g t^\R$.  
\end{dfn}
Recall that  $\Phi_\g g$ denotes the root system of $\g g$ (with respect to an a priori  chosen  Cartan subalgebra $\g h\sseq \g g$). 
\begin{lem}
\label{lem:kcompalpha-alpha}
  Let $\g g$ be a complex semisimple Lie algebra and let $\g g^\R$, $\tau$, $\theta$, $\g k^\R$, and $\g p^\R$  be as above.  Let $\g h^\R$ be a compact $\theta$-stable Cartan subalgebra of $\g g^\R$, with complexification $\g h$. Then  
$\tau(\g g_\alpha)=\g g_{-\alpha}$
 for every  $\alpha\in\Phi_\g g$.
 \end{lem}
\begin{proof}
For any $h\in\g h^\R$, the map $\ad_h:\g g^\R\to\g g^\R$ is skew-symmetric because
\[
\kappa_\theta(\ad_hx,y)=
-\kappa(\ad_hx,\theta(y))=
\kappa(x,\ad_h\theta(y))=\kappa(x,\theta(\ad_hy))=-\kappa_\theta(x,\ad_hy).
\]
It follows that the eigenvalues of $\ad_h$ are imaginary.
Now for $x\in\g g_\alpha$ and  $h\in\g h^\R$ we have
\[
\ad_h\tau(x)=[h,\tau(x)]=\tau([h,x])=\tau(\alpha(h)x)=\oline{\alpha(h)}\tau(x)=
-\alpha(h)\tau(x).
\]
Since $\alpha$ is $\C$-linear, it follows that $x\in\g g_{-\alpha}$. 
\end{proof}

\begin{lem}
\label{lem:tau(b)=bop}
Let $\g g$ be a complex reductive Lie algebra with a real form $\g g^\R$. Let $\theta$ be a Cartan involution of $[\g g^\R,\g g^\R]$. Let $\g h^\R$ be a Cartan subalgebra of $\g g^\R$, with complexification $\g h$, such that $\g h^\R\cap [\g g^\R,\g g^\R]$ is a compact $\theta$-stable Cartan subalgebra of $[\g g^\R,\g g^\R]$.
Then for every Borel subalgebra $\g b\sseq \g g$ that satisfies $\g h\sseq \g b$ we have $\tau(\g b)=\g b^\mathrm{op}$. 
 \end{lem}
\begin{proof}
This follows from Lemma~\ref{lem:kcompalpha-alpha}, noting that Borel subalgebras of $\g g$  are of the form $\g b'\oplus \Centre(\g g)$ for a Borel subalgebra $\g b'$ of $[\g g,\g g]$. 
\end{proof}

\begin{lem}
\label{lem:Z(a)has-cpct-Cartan}
Let $\g g$ be a semisimple Lie algebra with a real form $\g g^\R$, and let $\theta$ be a Cartan involution of $\g g^\R$.  Let $\g h^\R=\g t^\R\oplus\g a^\R$ be a $\theta$-stable Cartan subalgebra of $\g g$, and set $\g l^\R:=\Centre_{\g g^\R}(\g a^\R)$. Then $\theta|_{[\g l^\R,\g l^\R]}$ is a Cartan involution of $[\g l^\R,\g l^\R]$ and $\g h^\R\cap [\g l^\R,\g l^\R]$ is  a compact $\theta$-stable Cartan subalgebra of $[\g l^\R,\g l^\R]$. 
\end{lem}
\begin{proof}
As before, let $\g h=\g t\oplus\g a$ denote the complexification of $\g h^\R=\g t^\R\oplus\g a^\R$ and let $\Phi_\g g$ denote the $\g h$-root system of $\g g$. The complexification $\g l$ of  $\g l^\R$ is a Levi subalgebra of $\g g$, with Cartan subalgebra $\g h$, and with root system 
$
\Phi_\g l:=\{\alpha\in\Phi_\g g\,:\,\g a_\R\sseq \ker(\alpha)\}$. Henceforth we assume $\Phi_\g l\neq \varnothing$, as otherwise there is nothing to prove. 

Let $\kappa$ and $\kappa'$ denote the Killing forms of $\g g$ and $[\g l,\g l]$, respectively. The restriction of  $\kappa$ to $[\g l,\g l]$ is a non-degenerate invariant form, hence $\kappa'=c\kappa$ for a nonzero scalar $c$. Note that $c>0$ because for a coroot vector $h_\alpha$ with $\alpha\in\Phi_\g l$ we have 
$\kappa(h_\alpha,h_\alpha)=
\sum_{\beta\in\Phi_\g g}\beta(h_\alpha)^2>0$, and for a similar reason $\kappa'(h_\alpha,h_\alpha)>0$. Furthermore, $\theta$ restricts to an involution of $[\g l^\R,\g l^\R]$, and $\kappa_\theta$ is positive definite on $[\g l^\R,\g l^\R]$. These prove that $\theta$ induces a Cartan involution of $[\g l^\R,\g l^\R]$. 

We have  $\g l=[\g l,\g l]\oplus \Centre(\g l)$ where  
$\Centre(\g l)=\bigcap_{\alpha\in\Phi_\g l}\ker(\alpha)$, and  $\g h':=\sum_{\alpha\in\Phi_\g l}[\g g_\alpha,\g g_{-\alpha}]$ is a Cartan subalgebra of $[\g l,\g l]$ (see for example~\cite[Cor 5.94]{Knapp}). 
Clearly $\g h'\sseq \g h$, hence $\g h'=\g h\cap [\g l,\g l]$.
This in particular implies that $\tau(\g h')=\g h'$.
Consequently, the $\tau$-fixed real subalgebra of $\g h'$ is a Cartan subalgebra $(\g h')^\R$ of $[\g l^\R,\g l^\R]$.  
Since $\theta(\g g _\alpha)=\g g_\alpha$ for $\alpha\in\Phi_\g l$, 
it follows that $\g h'$ is $\theta$-stable. From $\theta\tau=\tau\theta$ it follows that $(\g h')^\R$ is $\theta$-stable as well.  
We have $\g h'\cap\g a\sseq [\g l,\g l]\cap \g a=\{0\}$, 
hence $\g h'\sseq \g t$. Taking $\tau$-fixed points we obtain $(\g h')^{\R}\sseq\g t^\R\sseq \g k^\R$. 
Thus, $(\g h')^\R=\g h^\R\cap[\g l^\R,\g l^\R]$ is a compact $\theta$-stable Cartan subalgebra of $[\g l^\R,\g l^\R]$. 
%
%
%

%
\end{proof}

\subsection{Construction of $\tau$-compatible Borel subalgebras}
\label{subsec:taucompat-construct}
In this subsection  $\g g$ will be a Lie superalgebra that  is either of basic type, or of type $\mathbf Q(n)$. We will use the notation introduced in Subsection~\ref{subsec:generalLie}. 
Recall that $\gund:=[\g g_\eev,\g g_\eev]$
and $\gund^\R:=\gund\cap \g g_\eev^\R$. We assume that    
\[
\gund^\R=\g k^\R_{}\oplus \g p^\R_{}
\quad\text{and}\quad
\hund^\R:=\g t^\R\oplus\g a^\R
\] 
are a Cartan decomposition and a $\theta$-stable Cartan subalgebra of $\gund^\R$. Taking complexifications, we obtain a Cartan subalgebra 
 $\hund=\g t\oplus \g  a$  of $\gund$. Set 
 \begin{equation}
 \label{eq:hzero}
 \g h_\eev:=\hund\oplus \Centre(\g g_\eev)\quad
 \text{and}\quad
 \g h_\eev^\R:=\g h_\eev\cap \g g_\eev^\R,
 \end{equation} 
 so that ${\g h}_\eev$ is a Cartan subalgebra of $\g g_\eev$. Let $\Phi_\g g$ and $\Phi_{\gund}$ denote the root systems   of $\g g$ and $\gund$ with respect to $\g h_\eev$ and $\hund$. Let $\g h=\g h_\eev\oplus\g h_\ood$ be the (unique) Cartan subalgebra of $\g g$ that contains ${\g h}_\eev$. 
 Set \begin{equation}
\label{lRdf}
\g l^\R:=\Centre_{\g g^\R}(\g a^\R),
\end{equation} 
and  let $\g l\sseq \g g$ be the complexification of $\g l^\R$. Note that $\g l_\eev=\Centre_{\g g_\eev}(\g a)$. Set
\begin{equation}
\label{eq:lunddfn}
\lund:=[\g l_\eev,\g l_\eev].
\end{equation} From Lemma~\ref{lem:Z(a)has-cpct-Cartan} it follows that $\g h^\R\cap\lund$ is a compact $\theta$-stable Cartan subalgebra of $\lund$. The $\g h_\eev$-root system of $\g l$ is 
\[
\Phi_\g l=\{\alpha\in\Phi_\g g\ ,\ \g a^\R\sseq \ker(\alpha)\}.
\]
\begin{lem}
\label{lem:roots-real-on-a}
 Every $\alpha\in\Phi_\g g$ is real-valued on $i\g t^\R\oplus \g a^\R$. 
\end{lem}
\begin{proof}
When $\alpha\in \Phi_{\gund}$, the argument is similar to the proof of Lemma~\ref{lem:kcompalpha-alpha}: again the map $\ad_h:\g g^\R\to\g g^\R$ is symmetric for $h\in\g a^\R$ and skew symmetric for $h\in\g t^\R$ (see also~\cite[Cor. 6.49]{Knapp}). Since the weight lattice of $\gund$ is contained in the $\Q$-span of $\Phi_{\gund}$, the assertion also holds for $\hund$-weights of finite dimensional $\gund$-modules. By considering  $\g g_\ood$  as a  $\gund$-module we obtain the assertion for all $\alpha\in\Phi_\g g$.
\end{proof}
We now construct the $\tau$-compatible Borel subalgebras of $\g g$.
\begin{dfn}
Let $\g h=\g h_\eev\oplus \g h_\ood$ be the Cartan subalgebra of $\g g$ corresponding to $\g h_\eev$ defined in~\eqref{eq:hzero}.  
Let $\g l$ be the complexification of $\g l^\R$, defined in~\eqref{lRdf}. 
We choose any positive system $\Phi_\g l^+$ for  $\Phi_{\g l}$.
Then we choose $h_\circ\in\g a^\R$ such that for every $\alpha\in\Phi_\g g\backslash \Phi_\g l$ we have $\alpha(h_\circ)\neq 0$, and 
  we extend $\Phi_{\g l}^+$  to a positive system $\Phi^+_{\g g}$ of  $\Phi_\g g$  by setting
\[
\Phi_{\g g}^+:=\Phi_{\g l}^+\cup\Phi^+_{\g u},\quad\text{ where}\quad\Phi^+_{\g u}:=\{\alpha\in\Phi_\g g\,:\,\alpha(h_\circ)>0\}.
\]
Let $\g b_\g l$ and $\g b$ be the Borel subalgebras of $\g l$ and $\g g$ that contain $\g h$ and correspond to $\Phi_{\g l}^+$ and $\Phi_{\g g}^+$, so that
\begin{equation}
\label{eq:b=bl+u}
\g b=\g b_{\g l}\oplus\g u\quad\text{ for   }\g u:=\bigoplus_{\alpha\in\Phi_\g u^+}\g g_\alpha.
\end{equation} 
We call $\g b$ the \emph{$\tau$-compatible} Borel subalgebra of $\g g$ associated to $\Phi^+_{\g l}$ and $h_\circ$.
We denote the fundamental systems of $\g b_\g l$ and $\g b$ by $\Pi_{\g b_{\g l}}$ and $\Pi_\g b$. Note that 
$\Pi_{\g b_{\g l}}\sseq \Phi_{\g l}^+\subseteq\Phi_{\g g}^+$.  
\end{dfn}

\begin{lem}
\label{lem:PiblPibr}
We have $\Pi_{\g b_{\g l}}=\Pi_{\g b}\cap\Phi_{\g l}$.
\end{lem}
\begin{proof}
By definition, the elements of $\Pi_{\g b_{\g l}}$ are the roots in $\Phi_{\g l}^+$ that cannot be expressed as a sum of two roots in $\Phi_{\g l}^+$. Let $\alpha\in\Pi_{\g b_{\g l}}$. If $\alpha\not\in\Pi_\g b$, then $\alpha=\beta+\gamma$ for some $\beta,\gamma\in\Phi_{\g g}^+$. At least one of $\beta,\gamma$ must be in $\Phi_{\g u}^+$. Thus,
$
0=\alpha(h_\circ)=\beta(h_\circ)+\gamma(h_\circ)>0
$,
which is a contradiction. This proves $\Pi_{\g b_{\g l}}\sseq \Pi_\g b\cap\Phi_\g l$. The reverse inclusion is trivial. 
\end{proof}
Let $\Gund$ and $\Hund$ be as in Subsection~\ref{subsec:generalLie}, and 
let $\Lund$ be the connected  algebraic subgroup of $\Gund$ with Lie algebra $\lund$.
\begin{rmk}
\label{rmk:Normalizers}
Let $L_\eev\sseq \Gund$ be the connected algebraic subgroup of $\Gund$ with Lie algebra $\g l_\eev$. Then  $[L_\eev,L_\eev]$ is connected and has Lie algebra $\lund$, hence 
$[L_\eev,L_\eev]=\Lund$.
As $L_\eev$ is reductive, we have   $L_\eev=\Centre(L_\eev)^\circ\Lund$ 
where  $\Centre(L_\eev)^\circ$ is the connected centre of $L_\eev$. Now $\Hund$ is a maximal torus of $L_\eev$, hence $\Centre(L_\eev)^\circ\sseq \Hund$.
Consequently, $\Hund=\Centre(L_\eev)^\circ(\Hund\cap \Lund)$. The Lie algebra of $\Hund\cap \Lund$ is $\hund\cap\lund$, which is a Cartan subalgebra of $\lund$. Thus, $(\Hund\cap \Lund)^\circ$ is a maximal torus of $\Lund$. But then  $(\Hund\cap \Lund)^\circ=\Hund\cap \Lund$, because a maximal torus in a connected reductive group is self-centralizing
(see \cite[Corollary 11.12]{Borel}).
Furthermore, 
$N_{\Lund}(\Hund\cap\Lund)=N_{\Lund}(\Hund)$.
\end{rmk}

As in Subsection~\ref{subsec:generalLie}, 
 for a fundamental system $\Pi_{\g b_{\g l}}$ of the $\g h$-root system $\Phi_{\g l}$ of $\g l$ corresponding to a Borel subalgebra $\g h\sseq \g b_{\g l}\sseq\g l$,  the actions of odd reflections (associated to $\alpha$) and elements $w$ of $N_{\Lund}(\Hund)=N_{\Hund}(\Hund\cap\Lund)$ on $\Pi_{\g b_{\g l}}$ are denoted by $r_{\g l,\alpha}\cdot \Pi_{\g b_{\g l}}$ and $w\cdot \Pi_{\g b_{\g l}}$, respectively. We denote the positive systems associated to $r_{\g g,\alpha}\cdot \Pi_\g b$ and $r_{\g l,\alpha}\cdot \Pi_{\g b_{\g l}}$ by $\Phi_{\g g,\alpha}^+$ and $\Phi_{\g l,\alpha}^+$, respectively. 

\begin{lem}
\label{lem:rlaPipb}
Let $\alpha\in\Pi_{\g b_{\g l}}$ be isotropic. Then  $r_{\g l,\alpha}\cdot \Pi_{\g b_{\g l}}=
\Phi_{\g l}\cap
(r_{\g g,\alpha}\cdot \Pi_\g b)$ and $\Phi_{\g g,\alpha}^+=\Phi_{\g l,\alpha}^+\cup \Phi_{\g u}^+$. 
\end{lem}
\begin{proof}
Recall that 
\[
r_{\g g,\alpha}\cdot \Pi_{\g b}=\{-\alpha\}
\cup
\{\beta\in\Pi_{\g b}\,:\,(\beta,\alpha)=0,\beta\neq \alpha\}
\cup
\{\beta+\alpha\,:\,\beta\in\Pi_{\g b},(\beta,\alpha)\neq 0\}.
\]
The first assertion is a consequence of  this, the analogous description of $r_{\g l,\alpha}\cdot \Pi_{\g b_{\g l}}$, and Lemma~\ref{lem:PiblPibr}. 
For the second assertion, note that  \[
\Phi^+_{\g g,\alpha}=(\Phi^+_{\g g}\backslash\{\alpha\})\cup\{-\alpha\}
=(\Phi_{\g l}^+\bls\{\alpha\}\cup \{-\alpha\})\cup\Phi_{\g u}^+=\Phi_{\g l,\alpha}^+\cup\Phi_{\g u}^+.\qedhere
\] \end{proof}

\begin{prp}
\label{prp:goodBorel}
Let $w\in N_{\Lund}(\Hund\cap\Lund)$ and 
$\alpha_1,\ldots,\alpha_k$ be  as in the statement of Lemma~\ref{lem:seqisotr} for $\g b_\g l$, so that 
$w^{-1}\tau(w^{-1})\in \Hund\cap\Lund$ and
\begin{equation}
\label{eq:Pitabll}
\Pi_{\tau(\g b_\g l)}=
w\cdot (r_{\g l,\alpha_k}\cdot (\cdots( r_{\g l,\alpha_1}\cdot \Pi_{\g b_\g l}))).
\end{equation}
Then the following assertions hold.
\begin{itemize}
\item[\rm (i)] The element $w$  corresponds to the longest element of the Weyl group 
\[
N_{\Lund}(\Hund\cap\Lund)/\Hund\cap\Lund
\]
of the root system $\Phi_{\lund}$.
\item[\rm (ii)]
We have $w\in N_{\Gund}(\Hund)$ and 
\begin{equation}
\label{eq:taub-and-b}
\Pi_{\tau(\g b)}=
w\cdot (r_{\g g,\alpha_k}\cdot (\cdots( r_{\g g,\alpha_1}\cdot \Pi_{\g b}))).
\end{equation}
\end{itemize}
\end{prp}
\begin{proof}
(i)
Recall from~\eqref{eq:ADw-1} that $\Ad_w((\g b_{\g l})_\eev)=\tau((\g b_{\g l})_\eev)$.  From Lemma~\ref{lem:kcompalpha-alpha} it follows that $\tau((\g b_{\g l})_\eev)$ is the Borel subalgebra opposite to $(\g b_{\g l})_\eev$. Hence $w$ must correspond to the longest element of the Weyl group of $\Phi_{\lund}$. 

(ii) 
From Remark~\ref{rmk:Normalizers} it follows that $w\in N_{\Gund}(\Hund)$. 
Since $\tau(\g u)=\g u$, we also have \begin{equation}
\label{eq:taub=taubl+}
\tau(\g b)=\tau(\g b_{\g l})\oplus \g u.
\end{equation} 
Successive application of Lemma~\ref{lem:rlaPipb} implies that \[
\alpha_i\in r_{\g g,\alpha_{i-1}}\cdot(\cdots(r_{\g g,\alpha_1}\cdot \Pi_{\g b}))
\quad\text{for }2\leq i\leq k.
\]
The positive system corresponding to the right hand side 
of~\eqref{eq:taub-and-b} is obtained from $\Phi_{\g g}^+$ by the following procedure: we remove $\alpha_i$ and include $-\alpha_i$ for $1\leq i\leq k$, and subsequently apply the map $\gamma\mapsto \Ad_w^*(\gamma)$. To complete the proof of~\eqref{eq:taub-and-b}, we must show that the latter positive system is identical to the positive system of $\tau(\g b)$. By~\eqref{eq:taub=taubl+}, the last assertion reduces to the fact that $\Ad_w(\g u)=\g u$. But 
note that $\g l_\eev=\Centre(\g l_\eev)\oplus \lund$, and since $\alpha(h_\circ)=0$ for every $\alpha\in\Phi_{\lund}$, we must have $h_\circ\in \Centre(\g l_\eev)$. As $\Lund$ is connected, we obtain $\Ad_g(h_\circ)=h_\circ$ for every $g\in \Lund$. In particular, for every $\gamma\in\Phi_{\g u}^+$ we have $\Ad_w^*(\gamma)(h_\circ)=\gamma(\Ad_{w^{-1}}(h_\circ))=\gamma(h_\circ) >0$, hence $\Ad_w^*(\gamma)\in \Phi_{\g u}^+$ as well. 
This completes the proof of~\eqref{eq:taub-and-b}. 
\end{proof} 
\subsection{Kostant's cascade and a Lie algebraic formula for $c_\lambda$}
In this subsection we continue with the notation of 
Subsection~\ref{subsec:taucompat-construct}. Thus, $\hund=\g t\oplus \g a$ is a $\theta$-stable Cartan subalgebra of $\gund$ such that $\tau(\g h)=\g h$, and $\g h_\eev=\hund \oplus \Centre(\g g_\eev)$. Let $\lund$ be defined as in~\eqref{eq:lunddfn}, and let $\g b$ be a $\tau$-compatible Borel subalgebra of $\g g$, so that~\eqref{eq:b=bl+u} holds. 

We recall the construction of Kostant's cascade of strongly orthogonal roots in the case of the root system $\Phi_{\lund}$. Let $\beta_1,\ldots,\beta_{N_1}$ denote the highest roots of the irreducible components 
$\Phi^{(k)}$, $1\leq k\leq N_1$, 
of $\Phi_{\lund}$, with respect to the positive systems induced from $\Phi_{\lund}^+=\Phi_{\lund}\cap\Phi_{\g l}^+$.
Next set \[
\Phi^{(k),\perp}:=\{\alpha\in\Phi^{(k)}\,:\,(\alpha,\beta_k)=0\} \quad\text{for }1\leq k\leq N_1.
\] Add the highest roots $\beta_{N_1+1},\ldots,\beta_{N_2}$ of the irreducible components of the $\Phi^{(k),\perp}$ to the sequence of the $\beta$'s in some order, and repeat this process until no more roots are left. The resulting sequence of roots $\beta_1,\ldots, \beta_N$ is a maximal set of strongly orthogonal roots in $\Phi_{\g l}$, that is, $\beta_i\pm\beta_j\not\in\Phi_{\lund}$ for $1\leq i,j\leq N$.

For the following theorem, recall that $c_\lambda$ is the real scalar defined in~\eqref{eq:lambdaB-clam}.
\begin{thm}
\label{thm:alternateclam}
Let $h_{\beta_1},\ldots,h_{\beta_N}\in\g h_\eev$ denote the coroots of $\beta_1,\ldots,\beta_N$, so that $\beta_i(h_{\beta_i})=2$ for $1\leq i\leq N$. Then for every $[W]\in\Irr(\catC_\g g)$  with $\g b$-highest weight $\lambda\in\g h_\eev^*$ we have 
 \[
 c_\lambda=(-1)^{\sum_{i=1}^N\lambda(h_{\beta_i})}\lambda(\HC_\g b(D_\lambda)).
 \]
\end{thm}
\begin{proof}
We have $\theta(\lund)=\lund$, and the Cartan decomposition of 
$\lund^\R:=\lund\cap\g g^\R$
 is 
\[
\lund^\R=(\lund^\R\cap\g k^\R)\oplus(\lund^\R\oplus \g p^\R)
.\] By Lemma~\ref{lem:Z(a)has-cpct-Cartan},  $\hund^\R\cap\lund^\R$ is a 
compact $\theta$-stable Cartan subalgebra of $\lund^\R$.  Let $\alpha\in\Phi_{\lund}$ and set
\[
\bfe:=e_\alpha\quad,\quad
\bff:=-\tau(e_\alpha)\quad,\quad
\bfh:=-[e_\alpha,\tau(e_\alpha)],
\] 
where the $e_\alpha\in\g l_\alpha$ are choices of root vectors. 
By Lemma~\ref{lem:kcompalpha-alpha} we have  $\tau(e_\alpha)\in\g g_{-\alpha}$, and since $\g a^\R\sseq \ker(\alpha)$ we have $\theta(\alpha)=\alpha$. Furthermore,  
 $\bfh\in i(\hund^\R\cap \lund^\R)$
because $\tau(\bfh)=-\bfh$. 
Lemma~\ref{lem:roots-real-on-a} implies that $\alpha(\bfh)\in\R$. Hence, by suitably scaling $e_\alpha$ and substituting $\alpha$ by $-\alpha$ if necessary, we can assume that $\bfe,\bff,\bfh$ is a  standard $\g{sl}_2$-triple. In matrix notation, the real form of the spanned $\g{sl}_2$ subalgebra that is induced by $\tau$ corresponds to the antilinear involution $X\mapsto -X^*$, and we have 
\[
\bfe=\begin{bmatrix}
0 & 1\\
0 & 0
\end{bmatrix}\quad,\quad
\bff=\begin{bmatrix}
0 & 0\\1 & 0
\end{bmatrix}\quad,\quad
\bfh=\begin{bmatrix}
1 & 0\\0 & -1
\end{bmatrix}.
\]
The algebraic subgroup of $\Gund$ corresponding to this $\g{sl}_2$ subalgebra of $\gund$ is isomorphic to either  $\mathrm{SL}_2(\C)$ or its quotient by $\{\pm I\}$, 
and the antiholomorphic involution of $\Gund$ restricts to the map $X\mapsto (X^*)^{-1}$. 
The simple reflection corresponding to $\alpha$ corresponds to 
\begin{equation}
\label{eq:w00}
w_\alpha=\begin{bmatrix}
0 & 1\\-1 & 0
\end{bmatrix}\in\mathrm{SL}_2(\C),
\end{equation}
and we have $w_\alpha\tau(w_\alpha)=w_\alpha^2=-I$, where $I\in\mathrm{SL}_2(\C)$ is the identity matrix. it follows that $\lambda(w_\alpha\tau(w_\alpha))=(-1)^{\lambda(\bfh)}$. 

It is a well known fact (originally due to Kostant) that the longest element of the Weyl group of a semisimple Lie algebra is equal to the product of reflections associated to the roots in the Kostant cascade. These reflections commute because of strong orthogonality. This fact together with the above $\mathrm{SL}_2(\C)$ calculation implies the assertion of the corollary. 
\end{proof}

\subsection{Revisiting the classification of modules  of real reductive Lie algebras}
Combining Theorems~\ref{thm-endotypes-basic} and~\ref{thm:alternateclam}, we obtain the classification of irreducible modules of real reductive Lie algebras, which we state below independently as it might be of independent interest. 

Given a reductive Lie algebra $\g g$, we have a decomposition 
$\g g=\Centre(\g g)\oplus \g g'$, where $\g g'=[\g g,\g g]$. 
For every $[W]\in\Irr(\catC_\g g)$, the elements of $\Centre(\g g)$ act by scalars (because of Schur's Lemma), hence $W$ is an irreducible $\g g'$-module.

For a real form $\g g^\R$ of $\g g$ that is the fixed subalgebra of the antilinear involution $\tau$, we have $\g g^\R=(\Centre(\g g)\cap \g g^\R)\oplus (\g g'\cap \g g^\R)$.  
Let $\theta$ be a Cartan involution of the real semisimple Lie algebra $\g g'\cap \g g^\R$, corresponding to the decomposition  
\[
\g g'\cap \g g^\R=\g k^\R\oplus\g p^\R
.
\] Let $\g h^\R=\g t^\R \oplus\g a^\R$ be a $\theta$-stable Cartan subalgebra of $\g g'\cap \g g^\R$, with complexification $\g h=\g t\oplus\g a$. Set $\g l:=\Centre_{\g g'}(\g a)$.  
As before, 
 $\Phi_{\g g'}$ and $\Phi_\g l$ denote the root systems of  $\g g'$ and $\g l$ with respect to $\g h$.

Choose a Borel subalgebra $\g b_\g l$ of $\g l$ that contains $\g h$, and 
an element $h_\circ\in\g a^\R$ such that $\alpha(h_\circ)\neq 0$ for every $\alpha\in\Phi_{\g g'}\bls \Phi_\g l$. Let $\g b$ be the Borel subalgebra of $\g g'$ obtained from $\g b_\g l$ and $h_\circ$ as follows:   $\g b_\g l\sseq \g b$ and for every $\alpha\in\Phi_{\g g'}\bls \Phi_{\g l}$ we have 
$
\g g'_\alpha\sseq \g b$ if and only if $\alpha(h_\circ)>0$.
 Let $w_\circ$ denote the longest element of the Weyl group of $\g h\cap [\g l,\g l]$, considered as an element of the Weyl group of $\g h$ in a canonical way (see Remark~\ref{rmk:Normalizers}). Finally, let $\beta_1,\ldots,\beta_N$ be Kostant's cascade of strongly orthogonal roots for the positive system of $\g b_\g l$, and let $h_{\beta_i}\in\g h$ denote the coroot corresponding to $\beta_i$ for $1\leq i\leq N$. 

\begin{thm}
\label{thm-reductivespecial}
Let $[W]\in\Irr(\catC_\g g)$ and let $\lambda\in\g h^*$ be the $\g b$-highest weight of $W$ as a $\g g'$-module.
\begin{itemize}
\item[\rm(i)] If $\Centre(\g g)\cap\g g^\R$ acts on $W$ by real scalars, $\Ad^*_{w_\circ}(\lambda)=\oline{\lambda\circ\tau}$, and $\sum_{i=1}^N\lambda(h_{\beta_i})$ is even, then $\EType_{\g g,\tau}(W)=0_\R$, i.e., it is isomorphic to $\R$.

\item[\rm(ii)] If $\Centre(\g g)\cap\g g^\R$ acts on $W$ by real scalars, $\Ad^*_{w_\circ}(\lambda)=\oline{\lambda\circ\tau}$, and $\sum_{i=1}^N\lambda(h_{\beta_i})$ is odd, then $\EType_{\g g,\tau}(W)=4_\R$, i.e., it is isomorphic to $\qH$.

\item[\rm(iii)] Otherwise, $\EType_{\g g,\tau}(W)=0_\C$, i.e., it is isomorphic to $\C$. 

\end{itemize} 
\end{thm}
\begin{proof}
From Theorem~\ref{thm:DirectSum}(ii) it follows that the exterior tensor product of two irreducible representations has endotype $0_\C$ if and only if at least one of them has endotype $0_\C$. Combining this  observation with Theorem~\ref{thm-endotypes-basic}, Remark~\ref{rmk-abbel}, Proposition~\ref{prp:goodBorel}, and Theorem~\ref{thm:alternateclam}, we obtain the assertions of the theorem. Note that in the notation of~\eqref{eq:lambdaB-clam} we have $r=0$, so that $\lambda_\funcB=\Ad_{w_\circ}(\lambda)$ and $D_\lambda=1$. 
\end{proof}

\section{Computing $\EType_{\g g,\tau}(W)$ when $\g g$ is of Cartan type or of type $\mathbf P(n)$}
  
In this section  $\g g$ will be  either 
of type $\mathbf P(n)$ or 
a  Cartan type Lie superalgebra.
Similar to Remark~\ref{rmk-smallcases}, we exclude $\g{pe}(1)$ and $\g{spe}(1)$ which are of Clifford type. 
Note that up to isomorphism, each of the Lie superalgebras $\mathbf{W}(n)$, $\mathbf{S}(n)$, and $\tilde{\mathbf S}(n)$ has a unique real form. (We remind the reader that  $\tilde{\mathbf S}(n)$ exists only for $n$ even). 

Our strategy is to use Corollary~\ref{thm:EtypeCplx} for a $\tau$-stable triangular decomposition of $\g g$ that is either of hyperbolic type, or of graded type  (see Definitions~\ref{dfn:hyper} and~\ref{dfn:gradedtyp}).
First we consider the case where $\g g$ is not isomorphic to 
$\mathbf{S}(n)$ or 
 $\tilde{\mathbf S}(n)$.
Then $\g g$ has a $\Z$-grading
\[
\g g=\bigoplus_{k\geq -1}\g g_k
\] 
that is consistent with the $\Z_2$-grading of $\g g$.  Furthermore,  for every real form of $\g g$,  we can assume that the corresponding  $\tau:\g g\to \g g$ preserves the $\mathbb Z$-grading. 
We have 
\[
\g g^\R=
\oline{\g u}^{\R}\oplus\g l^\R\oplus{\g  u}^{\R},
\]
where
 \[
\g u^\R:=\bigoplus_{k\geq 1}(\g g^\R)_k\quad,\quad
\g l^\R:=(\g g^\R)_0 
\quad,\quad
\oline{\g u}^\R:=(\g g^\R)_{-1}
.\]
Passing to complexifications, we obtain a direct sum decomposition
$\g g=\oline{\g u}\oplus\g l\oplus \g u$. This is Kac's triangular decomposition, and it  is of graded type. 

\begin{rmk}
\label{rmk:KacClassification}
When $\g g\neq \tilde{\mathbf S}(n)$, Kac gave a classification of irreducible $\g g$-modules by proving that  the functor $\catC_{\g g}\to  \catC_{\g l}$ defined by $W\mapsto W^{\g u}$ induces a bijection $\Irr(\catC_{\g g})\to \Irr(\catC_{\g l})$; see~\cite[Prop 5.2.5]{KacAdvances}. Note that $\g l$ is a reductive Lie algebra. Together with Corollary~\ref{thm:EtypeCplx}, this provides a complete answer to the computation of endotypes for irreducible $\g g$-modules. 
For $\g g=\tilde{\mathbf S}(n)$, we do not have a bijection $\Irr(\catC_{\g g})\to \Irr(\catC_{\g l})$, but we still have an injection for a hyperbolic triangular decomposition that will be described below; see Proposition~\ref{prp:Wq+-inv}(ii).
\end{rmk}

Next, suppose that  $\g g$ is a Lie superalgebra of type 
$\mathbf{S}(n)$ or 
$\tilde{\mathbf S}(n)$. The Lie superalgebra $\tilde{\mathbf S}(n)$ is not $\mathbb Z$-graded, but it has a subalgebra  isomorphic to 
\[
\underline{\g g}:=\bigoplus_{k\geq 0}\mathbf S(n)_k.
\] Furthermore, in both cases $\g g$ has a unique real form, the corresponding $\tau$ preserves the graded components of $\underline{\g g}$, and 
$(\underline{\g g}^\R)_{0}^{}\cong \g{sl}(n,\R)$. Let $\g h^\R$ be a  Cartan subalgebra of $(\underline{\g g}^\R)_{0}^{}$ and let $\g h$ be the complexification of $\g h^\R$. Then we have an $\g h^\R$-root decomposition of $\g g^\R$ with roots
\[
\eps_{i_1}+\cdots +\eps_{i_k}\quad\text{and}\quad
\eps_{i_1}+\cdots +\eps_{i_k}-\eps_j,
\]
where $\leq i_1<\cdots<i_k\leq n$ with $0\leq k\leq n-1$,  and $j\not\in\{i_1,\ldots,i_k\}$. 
Here as usual the $\eps_i$ are standard characters of the diagonal Cartan subalgebra of $\g{sl}(n,\R)$, so that \[
\eps_1+\cdots+\eps_n=0.
\] Following the notation of Section~\ref{subsec:generalLie}, let $\Phi_\g g$ denote the set of $\g h$-roots of $\g g$. 
We choose $x_\circ\in \g h^\R$ such that $\alpha(x_\circ)\neq 0$ for every $\alpha\in \Phi_\g g$, and we consider the 
corresponding $\tau$-stable triangular decomposition of hyperbolic type, as in Subsection~\ref{subsec:hyperb}. Then we  have 
\[
\oline{\g u}=\bigoplus_{\begin{subarray}{c}
\alpha\in\Phi_\g g\\
\alpha(x_\circ)<0\end{subarray}}
\g g_\alpha\quad,\quad
{\g l}=\g h
\quad,\quad
{\g u}=\bigoplus_{\begin{subarray}{c}
\alpha\in\Phi_\g g\\
\alpha(x_\circ)>0\end{subarray}}
\g g_\alpha
.\]

Recall that  in all cases $\g l$ is a reductive Lie algebra. 
We fix a Borel subalgebra $\g b_\g l\sseq \g l$. 
In Theorem~\ref{cor:Etype} below, by the  highest weight of an irreducible  $\g g$-module $W$ we mean  the $\g b_\g l$-highest weight of $W^{\g u}$. 
\begin{thm}
\label{cor:Etype}
Suppose that $\g g$ is of one of the types $\mathbf P(n)$, $\mathbf W(n)$, $\mathbf H(n)$,  $\mathbf S(n)$, or  $\tilde{\mathbf S}(n)$.  Let $[W]\in\Irr(\catC_{\g g})$. Then the following statements hold.
\begin{itemize}
\item[\rm (i)] If $\g g=\mathbf{S}(n)$ or $\tilde{\mathbf S}(n)$, then for the unique real form  of $\g g$ we always have $\EType_{\g g,\tau}(W)=0_\R$.  

\item[\rm (ii)] If  $\g g=\mathbf{W}(n)$, then for the unique real form of $\g g$ we have
\[
\EType_{\g g,\tau}(W)=\begin{cases}
0_\R & \text{if }\lambda\big|_{\Centre\left((\g g^\R)_0^{}\right)}\sseq \R,\\
0_\C&\text{otherwise,}
\end{cases}
\] 
where $\lambda$ is the highest weight of $W$ with respect to some choice of a Borel $\g b_\g l\sseq\g l$.

\item[\rm (iii)] If $\g g$ is of type $\mathbf{P}(n)$ or $\mathbf{H}(n)$, then $\EType_{\g g,\tau}(W)$ is either $0_\C$, $0_\R$, or $4_\R$, and determined from the highest weight $\lambda$ of $W$ according to Corollary~\ref{thm:EtypeCplx} and Theorem~\ref{thm-endotypes-basic}.
\end{itemize}
\end{thm}

\begin{proof}
All of the assertions are immediate consequences of Corollary~\ref{thm:EtypeCplx} and Theorem~\ref{thm-endotypes-basic}. For (i), note that $\g l^\R$ is the Cartan subalgebra of a Lie subalgebra of $\g g^\R$ that is isomorphic to $\g{sl}(n,\R)$ and $\lambda$ is the highest weight of an irreducible $\g{sl}(n,\C)$-module. Thus,  $\lambda$ is integral, hence real valued on $\g l^\R$. For (ii), note that $(\g g^\R)_0^{}\cong \g{gl}(n,\R)=\R\oplus \g{sl}(n,\R)$ and the endotype of every 
irreducible $\g{sl}(n,\R)$ module is $0_\R$. Finally, for (iii) recall that  irreducible modules of a reductive Lie algebra can only have endotypes $0_\C$, $0_\R$, and $4_\R$. 
\end{proof}

\section{Examples of explicit computations of endotypes}

\label{sec-examples}

This section is devoted to concrete examples that demonstrate how our results
in Sections~\ref{sec:ETypeWbasic-Q} and~\ref{sec:tau-compatible-parab}
classify  endotypes of irreducible modules of Lie superalgebras of basic type. 
\begin{ex}
Let $\g g=\g{sl}_n(\C)$ and $\tau(X):=-X^*$, so that $\g g^\R=\g{su}(n)$. The Cartan involution $\theta$ is the identity map. Let $\g b=\g h\oplus\g n$ denote the standard upper triangular Borel subalgebra of $\g g$. Then $\tau(\g b)=\g b^\mathrm{op}$, so that in Theorem~\ref{thm-endotypes-basic} we have $r=0$ and $w$ is the longest element of the Weyl group (represented by an antidiagonal matrix). Furthermore, $\g h^\R$ is a compact Cartan subalgebra, and the Kostant cascade is $\beta_1,\ldots,\beta_{\lfloor \frac{n}{2}\rfloor}$ where $\beta_i=\eps_i-\eps_{n+1-i}$ for $1\leq i\leq \lfloor \frac{n}{2}\rfloor$. We denote the coroots of the simple roots of $\g b$ by 
$h_i:=h_{\eps_i-\eps_{i+1}}$.
The $\g b$-highest weight $\lambda$ of an irreducible $\g g$-module $W$ corresponds to a sequence of non-negative integers $(\lambda_1,\ldots,\lambda_{n-1})$ where $\lambda_i:=\langle \lambda,h_i\rangle$. Since $\bar \lambda_i=\lambda_i$ for all $i$, the condition $\lambda_\funcB=\oline{\lambda\circ\tau}$ is equivalent to
\begin{equation}
\label{eq:lai=lani}
\lambda_i=\lambda_{n+1-i}\quad\text{for every }
1\leq i\leq n.
\end{equation} 
Let $ E_{i,j}$ denote the $n\times n$ matrices with a 1 in the $(i,j)$ position and 0's elsewhere. Then 
$h_i=E_{i,i}-E_{i+1,i+1}$ and $h_{\beta_i}=E_{i,i}-E_{n+1-i,n+1-i}$. It follows that 
\[
\sum_{i=1}^{\lfloor \frac{n}{2}\rfloor}h_{\beta_i}=
\begin{cases}
\sum_{i=1}^m ih_i+\sum_{i=1}^{m-1}(m-i)h_{m+i}&\text{ if }n=2m,\\[4mm]
\sum_{i=1}^m ih_i+
\sum_{i=0}^{m-1}(m-i)h_{m+i+1}&\text{ if }n=2m+1.
\end{cases}
\]
By a direct calculation, if tollows that if~\eqref{eq:lai=lani} holds, then 
\[
\textstyle \lambda\left(\sum_{i=1}^{\lfloor \frac{n}{2}\rfloor}h_{\beta_i}\right)=\begin{cases}
m\lambda_m+\sum_{i=1}^{m-1}2i\lambda_i&\text{ if }n=2m,\\
\sum_{i=1}^{m-1}2i\lambda_i&\text{ if }n=2m+1.
\end{cases}
\]
From Theorem~\ref{thm:alternateclam} and Theorem~\ref{thm-endotypes-basic} it follows that:

\begin{itemize}
\item[(i)] If $\lambda_i=\lambda_{n+1-i}$ for at least one $1\leq i\leq n$, then $\EType_{\g g,\tau}(W)=0_\C$.

\item[(ii)] If $\lambda_i=\lambda_{n+1-i}$ for every $1\leq i\leq n$ and $n=2m+1$, then $\EType_{\g g,\tau}(W)=0_\R$. 
 
\item[(iii)] If $\lambda_i+\lambda_{n+1-i}=0$ for every $1\leq i\leq n$ and $n=2m$, then \[
\EType_{\g g,\tau}(W)=\begin{cases}
0_\R & \text{ if }(-1)^{m\lambda_m}=1,\\
4_\R & \text{ if }(-1)^{m\lambda_m}=-1.
\end{cases},
\] 
 \end{itemize}

\end{ex}

\begin{ex}
\label{ex-gl11}
let $\g g=\gl(1|1)$ and let $\g b$ denote the standard upper triangular Borel subalgebra of $\g g$. The split real form of $\g g$ corresponds to $\tau(X)=\bar X$. We have $\tau(\g b)=\g b$, hence $c_\lambda=1$ for every $\lambda$. This proves that for every irreducible $\g g$-module $W$ we have $\EType_{\g g,\tau}(W)=0_\R$. 

Next we consider the compact form $\g{u}(1,0|1,0)$, corresponding to \[
\tau\left(\begin{bmatrix}
x & y\\z& w
\end{bmatrix}
\right)=
\begin{bmatrix}
-\bar x & -\bar z\sqrt{-1}\\-\bar y\sqrt{-1}& -\bar w
\end{bmatrix}
.\]  Then $\tau(\g b)=\g b^\mathrm{op}$, and 
in the notation of Lemma~\ref{lem:seqisotr} we have $\Pi_{\tau(\g b)}=r_{\g g,\eps_1-\delta_1}\Pi_\g b$. The highest weight of an irreducible $\g g$-module $W$ is of the form $\lambda=a\eps_1+b\delta_1$ where $a,b\in\C$. We have
\[
\lambda_\funcB=\begin{cases}
a\eps_1+b\delta_1 & \text{ if }a+b=0,\\
(a-1)\eps_1+(b+1)\delta_1& \text{ if }a+b\neq 0. 
\end{cases}
\]
We have $\oline{\lambda\circ\tau}=-\bar a \eps_1-\bar b\delta_1$.
\begin{itemize}
\item[(i)] Suppose $a+b=0$. Then the condition $\lambda_\funcB=\oline{\lambda\circ\tau}$ is equivalent to $a=\alpha\sqrt{-1}$ and $b=-\alpha\sqrt{-1}$, where $\alpha\in\R$. In this case $D_\lambda=1$ and $c_\lambda=1$, hence $\EType_{\g g,\tau}(W)=0_\R$. If $a+b=0$ but $a$ and $b$ are not purely imaginary, then $\EType_{\g g,\tau}(W)=0_\C$.

\item[(ii)] Suppose that $a+b\neq 0$. Then the condition $\lambda_\funcB=\oline{\lambda\circ\tau}$ is equivalent to
\begin{equation}
\label{eq:a=b=12}
a=\frac12+\alpha\sqrt{-1}\ \text{ and }\ 
b=-\frac12+\beta\sqrt{-1},\text{ for }\alpha,\beta\in\R,\ \alpha\neq \beta. 
\end{equation}
If~\eqref{eq:a=b=12} does not hold, then $\EType(W)=0_\C$. 
When~\eqref{eq:a=b=12} is satisfied, we have $D_\lambda=
E_{1,2}(-\sqrt{-1}E_{2,1})$, so that $\HC_{\g b}(D_\lambda)=-\sqrt{-1}(E_{1,1}+E_{2,2})$ and $c_\lambda=\alpha+\beta$. From Theorem~\ref{thm-endotypes-basic} it follows that if $\alpha+\beta>0$ then $\EType_{\g g,\tau}(W)=6_\R$, whereas if 
$\alpha+\beta<0$ then $\EType_{\g g,\tau}(W)=2_\R$. 
\end{itemize}
Note that if we take the realization of $\g{u}(1,0|1,0)$ 
by\[
\tau\left(\begin{bmatrix}
x & y\\z& w
\end{bmatrix}
\right)=
\begin{bmatrix}
-\bar x & \bar z\sqrt{-1}\\\bar y\sqrt{-1}& -\bar w
\end{bmatrix}
,\]
then the constraint on $\alpha+\beta$ in (ii) will be flipped, i.e., the endotype will be $2_\R$ if $\alpha+\beta>0$, and $6_\R$ if $\alpha+\beta<0$.

\end{ex}

\begin{rmk}
The dependence on $\tau$ of the correspondence between highest weights and endotypes that we observed in Example~\ref{ex-gl11}  can also occur for Lie algebras. For instance, let $\g g=\g{so}(8)$ and let $W=\C^8$ denote the standard representation of $\g g$.  Then the real form $\g{so}(2,6)$ of $\g g$  has a realization as $8\times 8$ real matrices, and this proves that 
with respect to this real form the endotype of $W$ is $0_\R$. Furthermore,  the real form $\g{so}^*(8)$ can be realized as $4\times 4$ matrices over quaternions, and this proves that the endotype of $W$ with respect to  $\g{so}^*(8)$ is $4_\R$. But $\g{so}^*(8)\cong \g{so}(2,6)$. Indeed the latter isomorphism is obtained by a diagram automorphism of $D_4$ (triality). 
This shows that the description of endotypes by highest weights does not depend solely on the isomorphism class of a real form, but on the conjugacy classes of the real forms under the adjoint group. 
\end{rmk}

\begin{ex}
Let $\g g=\g{gl}(1|2)$ and $\tau(X)=-M\bar X^\mathrm{str}M^{-1}$ where $M$ is the $3\times 3$ diagonal matrix with diagonal entries $1,\sqrt{-1},\sqrt{-1}$, so that $\g g^\R=\g{u}(1,0|2,0)$. Then $\tau(\g b)=\g b^\mathrm{op}$ and with $w_\circ$ as in~\eqref{eq:w00}, 
in the notation of Lemma~\ref{lem:seqisotr} we have
\[
\Pi_{\tau(\g b)}=\tilde w_\circ \cdot
(r_{\g g,\eps_1-\delta_2}\cdot 
(r_{\g g,\eps_1-\delta_1}\cdot 
\Pi_\g b)))),\quad\tilde w_\circ:=\begin{bmatrix}
1  & 0\\0 & w_\circ\end{bmatrix}.
\]
It follows that 
\[
e_{\alpha_1}=E_{1,2}\quad,\quad
e_{\alpha_2}=E_{1,3},
\]
hence
\[
\Ad_{{\tilde w}_\circ^{-1}}(\tau(e_{\alpha_1}))=-\sqrt{-1}E_{3,1}\ ,\
\Ad_{{\tilde w}_\circ^{-1}}(\tau(e_{\alpha_2}))=\sqrt{-1} E_{2,1}.
\]
The $\g b$-highest weight $\lambda$ of an irreducible $\g g$-module $W$ is of the form $\lambda=a_1\eps_1+b_1\delta_1+b_2\delta_2$ where  $b_1-b_2$ is  a non-negative integer. For conciseness, we denote this highest weight by a triple $(a_1,b_1,b_2)$. 
Then we have  
\[
\lambda_\funcB=
\begin{cases}
(a_1-2,\; b_2+1,\; b_1+1) & \text{if } a_1+b_1\neq 0 \text{ and } a_1+b_2\neq 1,\\
(a_1-1,\; b_2,\; b_1+1) & \text{if } a_1+b_1\neq 0 \text{ and } a_1+b_2=1,\\
(a_1-1,\; b_2+1,\; b_1) & \text{if } a_1+b_1=0 \text{ and } a_1+b_2\neq 0,\\
(a_1,\; b_2,\; b_1) & \text{if } a_1+b_1=0 \text{ and } a_1+b_2=0.
\end{cases}
\]
Now $\oline{\lambda\circ\tau}$ corresponds to $(-\bar a_1,-\bar b_1,-\bar b_2)$, hence the condition $\lambda_\funcB=\oline{\lambda\circ \tau}$ only occurs in the first and fourth cases, under the following restrictions:
\begin{itemize}
\item[(i)] From the first case,  we obtain that $a_1=1+\alpha\sqrt{-1}$ and  $b_2=-1-\bar b_1$, where $\alpha\in\R$, $b_1\in\C$, and $2\Re(b_1)+1\in\mathbb Z^{\geq 0}$. In this case
there are two odd reflections (i.e., $r=2$) and
 $D_\lambda=E_{1,2}E_{1,3}E_{3,1}E_{2,1}$, so that \[
\HC_\g b(D_\lambda)=(E_{1,1}+E_{2,2})(E_{1,1}+E_{3,3}-1).
\]
The Kostant cascade of $\lund=\g{sl}_2(\C)$ is $\delta_1-\delta_2$, and the corresponding coroot is the $3\times 3$ matrix $E_{2,2}-E_{3,3}$. Thus
by Theorem~\ref{thm:alternateclam} we have 
\[
c_\lambda=(-1)^{b_1-(-1-\bar b_1)}(1+\alpha\sqrt{-1}+b_1)(\alpha\sqrt{-1}-1-\bar b_1)=
(-1)^{2\Re(b_1)+1}
(-z\bar z),
\]
where $z=1+\alpha\sqrt{-1}+b_1$. 
From $\Re(b_1)\geq -\frac12$ it follows that $\Re(z)\neq 0$, hence $-z\bar z<0$.
From Theorem~\ref{thm-endotypes-basic} it follows that:
\begin{itemize}
\item If $2\Re(b_1)$ is even then $\EType_{\g g,\tau}(W)=0_\R$.
\item If $2\Re(b_1)$ is odd then $\EType_{\g g,\tau}(W)=4_\R$.
 
\end{itemize}

\item[(ii)] From the last case, we obtain $a_1=\alpha\sqrt{-1}$, $b_1=b_2=-\alpha\sqrt{-1}$, where $\alpha\in\R$. In this case there are no odd reflections so that $D_\lambda=1$, and $\EType_{\g g,\tau}(W)=0_\R$ because $c_\lambda=(-1)^{b_1-b_2}=1$.

\item[(iii)] If $\lambda$ is not one of the highest weights described in (i) and (ii), then $\EType_{\g g,\tau}(W)=0_\C$. 
\end{itemize}

\end{ex}

\appendix

\section{Proofs of Theorems~\ref{thm:Irr=Irr} and~\ref{thm:Table}} 
\label{sec:proofsthms}

We denote the identity functor of $\catR$ by $\funcI_\catR$.
\begin{lem}
\label{lem:EFandEF}
We have the following natural isomorphisms of functors:
\begin{itemize}
\item[\rm (i)]
$\funcF\funcE\cong \funcI_\catR\oplus\funcI_\catR$.
\item[\rm (ii)]
$\funcE\funcF\cong \funcI_\catC\oplus\funcB$.
\item[\rm (iii)] $\funcB\funcE\cong \funcE$. 
\end{itemize}
\end{lem}

\begin{proof}
(i) Follows immediately from the definitions of $\funcE$ and $\funcF$. 

(ii) Let $W:=(V,\iota)\in \Obj(\catC)$ and let $\tilde\epsilon_1,\tilde\epsilon_2\in
\Hom_{\catR}(V,V\oplus V)$ be defined by 
\[
\tilde\epsilon_1:=\epsilon_1-\epsilon_2\iota\quad\text{and}\quad
\tilde\epsilon_2:=\epsilon_1+\epsilon_2\iota.
\]
Recall that 
$\funcE\funcF (W)=(V\oplus V,-\epsilon_{1}\pi_{2}+\epsilon_{2}\pi_{1})$. A direct calculation yields
\[
\tilde\epsilon_1\in \Hom_{\catC}(W,\funcE\funcF (W))\quad\text{and}\quad
\tilde\epsilon_2\in \Hom_{\catC}(\funcB (W),\funcE\funcF (W)).
\]
Next set $\tilde\pi_1:=\frac12(\pi_1+\iota\pi_2)$ and $\tilde\pi_2:=\frac12(\pi_1-\iota\pi_2)$, so that $\tilde\pi_1,\tilde\pi_2\in\Hom_{\catR}(V\oplus V,V)$. Again by direct calculations we have 
\[
\tilde\pi_1\in\Hom_\catC(\funcE\funcF (W),W)\quad\text{and}
\quad
\tilde\pi_2\in\Hom_\catC(\funcE\funcF (W),\funcB (W)).
\]
Also, $\tilde \pi_1\tilde\epsilon_1=1_W$, 
$\tilde \pi_2\tilde\epsilon_2=1_{\funcB W}$, and 
$\tilde \pi_i\tilde\epsilon_j=0$ for $i\neq j$. This proves that $\funcE\funcF (W)\cong W\oplus \funcB (W)$.  
For naturality of the latter  isomorphism note that 
for 
$\phi\in\Hom_{\catC}(W,W')$
where  $W,W'\in\Obj(\catC)$ 
we have 
\[
\tilde \pi'_1(\funcE\funcF(\phi))\tilde\epsilon_1=1_W\quad,\quad
\tilde \pi'_2(\funcE\funcF(\phi))\tilde\epsilon_2=1_{\funcB(W)}\quad,\quad
\tilde \pi'_i(\funcE\funcF(\phi))\tilde\epsilon_j=0
\text{ for }i\neq j, 
\]
where $\tilde\pi_1',\tilde\pi'_2$ are defined similar to $\tilde\pi_1,\tilde\pi_2$, but with respect to $W'$. 

(iii) It is straightforward to check that the $\catC$-isomorphisms $\funcE(V)\to\funcB\funcE(V)$ given by 
$\epsilon_2\pi_1+\epsilon_1\pi_2$ are natural.
\end{proof}

\begin{dfn}
Let $f:X\to Y$ be a monomorphism in an abelian category $\catA$. We call $f$  a \emph{proper embedding} if $X\not\cong \mathbf 0_\catA$ and $\coker(f)\neq 0$. 
\end{dfn}

\begin{rmk}
\label{rmk:WsimpleBWsimple}
Let $W\in\Obj(\catC)$ be simple. Then $\funcB(W)$ is also a simple object of $\catC$. Indeed   for any $W'\in\Obj(\catC)$ the isomorphism $\Hom_\catC(W',\funcB(W))\xrightarrow{\cong}
\Hom_\catC(\funcB(W'),W)$ given by $f\mapsto \funcB(f)$ maps proper embeddings to proper embeddings.

\end{rmk}

\begin{lem}
\label{lem:E2}
Let $V\in\Obj(\catR)$ be a simple object. Then either $\funcE (V)$ is simple or $\funcE (V)\cong W\oplus \funcB(W)$ where $W\in\Obj(\catC)$ is simple and  $\funcF(W)\cong V$.  
\end{lem}

\begin{proof}
By Lemma~\ref{lem:EFandEF}(i) we have $\funcF\funcE(V)\cong V\oplus V$, hence $\funcE(V)$ has length at most 2. 
Suppose that $\funcE (V)$ is not simple.
Then there exists a proper embedding  $f\in\Hom_{\catC}(W, \funcE (V))$ for a simple object $W\in\Obj(\catC)$.  Since $\funcF$ is exact,  $\funcF(W)$ is a proper embedding  as well. From these and the isomorphism $\funcF\funcE(V)\cong V\oplus V$ it follows that  $\funcF (W)\cong V$. Thus, Lemma~\ref{lem:EFandEF}(ii) implies that 
\[
\funcE(V)\cong\funcE\funcF(W)\cong
W\oplus \funcB (W),
\] 
and $\funcB(W)$ is a simple object because of Remark~\ref{rmk:WsimpleBWsimple}.
Finally, we have
\[
V\oplus V\cong \funcF\funcE(V)\cong \funcF(W)\oplus \funcF\funcB(W)
\cong \funcF(W)\oplus \funcF(W),
\]
from which it follows that $\funcF(W)\cong V$. 
 \end{proof}

\begin{lem}
\label{lem:F2}
Let $W\in\Obj(\catC)$ be a simple object. Then either $\funcF (W)$ is simple or $\funcF (W)\cong V\oplus V$ where $V\in\Obj(\catR)$ is simple
and  $\funcE(V)\cong W\cong \funcB(W)$.  
\end{lem}
\begin{proof}
If $\funcF(W)$ is not simple then there exists a proper embedding $f\in\Hom_{\catR}(V,\funcF(W))$ for a simple object  $V\in\Obj(\catR)$.
 Since $\funcE$ is exact, $\funcE(f)$ is also a proper embedding. However, by Lemma~\ref{lem:EFandEF}(ii) we have $\funcE\funcF(W)\cong W\oplus \funcB(W)$. It follows that either $\funcE(V)\cong W$ or $\funcE(V)\cong \funcB(W)$. Since $\funcB^2\cong \funcI_\catC$, we obtain $W\cong \funcE(V)$ or $W\cong \funcB\funcE(V)$. Since $\funcF\funcB\cong\funcF$, from either of the latter isomorphisms and Lemma~\ref{lem:EFandEF}(i)
it follows that 
$
\funcF(W)\cong \funcF\funcE(V)\cong V\oplus V
$. Finally, from Lemma~\ref{lem:EFandEF}(iii) we obtain 
$\funcE(V)\cong W\cong \funcB(W)$.
\end{proof}

\begin{lem}
\label{lem:Hom(EV,EV)}
Let $V\in\Obj(\catR)$ and set $A:=\End_{\catR}(V)$. Then $\End_{\catC}(\funcE (V))\cong A\otimes_{\R}\C$.
\end{lem}
\begin{proof}
Let $f\in \End_{\catC}(\funcE V)$. Since $f\in\End_{\catR}(V\oplus V)$, we can define $f_{i,j}:\End_{\catR}(V)$ 
for $i,j\in\{1,2\}$
by $f_{i,j}:=\pi_jf\epsilon_i$, so that
$
f=\sum_{i,j}\epsilon_j f_{i,j}\pi_i
$. 
The morphism $f$  commutes with 
$-\epsilon_1\pi_2+\epsilon_2\pi_1$ if and only if  $f_{2,2}=f_{1,1}$ and $f_{2,1}=-f_{1,2}$. We can readily verify that the map
\[
\End_{\catC}(\funcE (V))\to A\otimes_\R\C
\quad,\quad
f\mapsto f_{1,1}+\sqrt{-1}f_{1,2}
\] 
is a $\C$-linear isomorphism of algebras.
\end{proof}

\begin{lem}
\label{lem:EndRCH}
Let $V\in\Obj(\catR)$ be simple. 
\begin{itemize}
\item[\rm (i)]
If $\End_\catR(V)\cong\R$ then $W:=\funcE(V)$ is simple and $(\funcB,+)\in\Sigma_W$ but $(\funcB,-)\not\in\mathfrak S_W$.
\item[\rm (ii)]
If $\End_\catR(V)\cong \C$ or $\qH$, then $\funcE(V)\cong W\oplus \funcB(W)$ where $W:=(V,\iota)$ for $\iota\in\End_\catR(V)$ corresponding to $i\in\C$ or $\qH$, respectively. \item[\rm (iii)] For $V$ and $W$ as in (ii), if
 $\End_\catR(V)\cong \C$  then we have $(\funcB,\pm)\not\in\mathfrak S_W$. If 
  $\End_\catR(V)\cong \qH=\C\oplus \C j$, then 
$(\funcB,-)\in\mathfrak S_W$ but $(\funcB,+)\not\in\mathfrak S_W$, and the isomorphism $W\cong \funcB(W)$ is induced by $j\in\End_\catR(V)$.

\end{itemize}
\end{lem}

\begin{proof}
(i) From Lemma~\ref{lem:E2} it follows that $\funcE(V)$ is simple if and only if $\End_\catC(\funcE(V))$ is a division algebra. This and Lemma~\ref{lem:Hom(EV,EV)} imply that $W:=\funcE(V)$ is simple. 
As noted in the proof of
Lemma~\ref{lem:EFandEF}(iii), the morphism $\phi\in\End_\catR(\funcE(V))$ defined by  $\phi:=\epsilon_1\pi_2+\epsilon_2\pi_1$ is an isomorphism between $W$ and $\funcB(W)$.  Clearly 
 $\phi^2=1$, hence  $(\funcB,+)\in\mathfrak S_W$.  Lemma~\ref{lem:sig}(ii) implies that $(\funcB,-)\not\in \mathfrak S_W$.

(ii) Since $V=\funcF(W)$, the assertion follows from Lemma~\ref{lem:EFandEF}(ii). 

(iii) An isomorphism $W\cong \funcB(W)$ in $\catC$ is an isomorphism of $V$ in $\catR$ that does not commute with $\iota\in\End_\catR(V)$. In other words, from
$W\cong \funcB(W)$ it follows that
$\End_\catR(V)$ must be noncommutative. This proves the first assertion. 
Next suppose that $\End_\catR(\funcE(V))\cong \qH$ and recall that $W:=(V,\iota)$ where $\iota$ corresponds to $i\in\End_\catR(V)$. From $ij=-ji$ it follows that  the $\catR$-isomorphism $j\in\qH$ belongs to $\Hom_\catC(W,\funcB(W))$, hence $(\funcB,-)\in \mathfrak S_W$ and thus by Lemma~\ref{lem:sig} we have $(\funcB,+)\not\in\mathfrak S_W$
\end{proof}


\begin{prp}
\label{prp:IrrI-IrrII}
The map $[V]\mapsto [\funcE(V)]$ is a bijection from $\Irr_+(\catR)$ onto $\Irr_+(\catC)$.

\end{prp}

\begin{proof}
From Lemma~\ref{lem:EndRCH}(i)
it follows that 
if 
$[V]\in\Irr_+(\catR)$ then
$[\funcE(V)]\in \Irr_+(\catC)$.
To verify injectivity, it suffices to prove that for simple objects $V,V'$ of $\catR$, if $\funcE(V)\cong \funcE(V')$ then $V\cong V'$.
By Lemma~\ref{lem:EFandEF}(i) 
we have $V\oplus V\cong \funcF\funcE(V)\cong \funcF\funcE(V')\cong V'\oplus V'$. Since $V$ and $V'$ are both simple we obtain $V\cong V'$, as desired.
For surjectivity,  note that given any $W\in\Irr_+(\catC)$, by Lemma~\ref{lem:F2} we have  
$\funcF(W)=V\oplus V$ and $\funcE(V)\cong W$. Finally, Lemma~\ref{lem:EndRCH} implies that $[V]\in\Irr_+(\catR)$.     
\end{proof}

\begin{proof}[Proof of Theorem~\ref{thm:Irr=Irr}]
(i) This is a trivial consequence of the fact that the endomorphism ring of a simple object of $\catR$ is a division algebra over $\R$. 

(ii) Follows immediately from Lemma~\ref{lem:F2}. 

(iii) From $\funcF\funcB=
\funcF$ it follows that the sets $\Irr_\star(\catC)$ are $\Theta_\funcB$-stable. 
From Lemma~\ref{lem:F2} it follows that $\Irr_+(\catC)$ is $\Theta_\funcB$-fixed. 

(iv)
Assume that $[W]\in\Irr_+(\catC)$. 
Then by Proposition~\ref{prp:IrrI-IrrII} we have $W=\funcE(V)$ where $[V]\in\Irr_+(\catR)$. Furthermore $\Phi([W])=[V]$. Thus $\Phi$ is a bijection from $\Irr_+(\catC)$ onto $\Irr_+(\catR)$. 

Next assume that $[W]\in \Irr_\circ(\catC)$. Then 
$\Phi([W])=[\funcF(W)]$ and 
by Lemma~\ref{lem:EFandEF}(ii) we have 
$\funcE\funcF(W)\cong W\oplus \funcB(W)$. From $\funcB(W)\not\cong W$ it follows that 
$\End_{\catC}(\funcE\funcF(W))\cong \C\times \C$, hence from Lemma~\ref{lem:Hom(EV,EV)} it follows that $\End_\catR(\funcF(V))\cong\C$, i.e., $[V]\in\Irr_\circ(\catR)$. 

An argument similar to the case $[W]\in\Irr_\circ(\catC)$ implies that if $[W]\in\Irr_-(\catC)$ then $\Phi([W])\in\Irr_-(\catR)$. It remains to prove that the restricted map \[\Phi:\Irr_\circ(\catC)/\funcB\cup\Irr_-(\catC)\to 
\Irr_\circ(\catR)\cup\Irr_-(\catR)
\] is a bijection. For surjectivity, note that if $[V]\in\Irr_\circ(\catR)\cup\Irr_-(\catR)$ then  Lemma~\ref{lem:EndRCH} and Lemma~\ref{lem:E2} imply that 
$\funcE(V)\cong W\oplus \funcB(W)$ such that $\funcF(W)=V$, hence $\Phi([W])=[V]$. 
For injectivity, assume that 
$\Phi([W])=\Phi([W'])$ for $[W],[W']\in\Irr_\circ(\catC)\cup\Irr_-(\catC)$. Then 
\[
W\oplus \funcB(W)\cong 
\funcE\funcF(W)\cong \funcE\funcF(W')\cong W'\oplus \funcB(W').
\]
It follows that $W\cong W'$ or $W\cong \funcB(W')$, hence $[W]$ and $[W']$ are in the same $\Theta_\funcB$-orbit. 
\end{proof}

\begin{proof}[Proof of Theorem~\ref{thm:Table}]
By Lemma~\ref{lem:Hom(EV,EV)}
the $2^\mathrm{nd}$ column in Table~\ref{Table-1} determines the $3^\mathrm{rd}$ column. 
By Lemma~\ref{lem:EndRCH}  the $2^\mathrm{nd}$ column also detemines the $4^\mathrm{th}$ column. Existence of symmetries of the form $(\funcB,\pm )$
follows from Lemma~\ref{lem:EndRCH} as well. 
It remains to determine the existence of the symmetries $\Pi$ and $(\Pi\funcB,\pm)$. 
The latter symmetries can only exist when $\wEnd_\catR(V)_{\ood}\neq0$ (because of Lemma~\ref{lem:F2}), and in the rest of the proof we focus on these cases.  
Let  $\varepsilon\in\End_\catR(V)_\ood$ be an element satisfying $\varepsilon^2=\pm 1$. Then $\eps$ corresponds to $\phi_\eps\in\Hom_\catR(V,\Pi(V))$. There are two cases to consider:

\noindent\textbf{Case I:}
$\funcE(V)=W$, for a simple $W$.  Then  $\Pi\in\mathfrak S_W$ 
via $\tilde\phi_\eps:=\funcE(\phi_\eps)$.
Also, by Lemma~\ref{lem:EndRCH} we have $(\funcB,+)\in\mathfrak S_W$ and the isomorphism $W\cong \funcB(W)$ is by the morphism
$\psi:=\epsilon_1\pi_2+\epsilon_2\pi_1$. From $\funcB(\tilde\phi_\eps)\psi=\Pi(\psi)\tilde \phi_\eps$ it follows that 
$c(\Pi(\psi)\phi_\eps)=c(\psi)\mathrm{sign}(\eps^2)=\mathrm{sign}(\eps^2)$, so that 
$(\Pi\funcB,\mathrm{sign}(\eps^2))\in \mathfrak S_W$. 

\noindent\textbf{Case II:}
$\funcE(V)=W\oplus \funcB(W)$. Then  by Lemma~\ref{lem:EndRCH} we have $\C\sseq\End_\catR(V)$ and $W=(V,\iota)$ where $\iota$ corresponds to $i\in\End_\catR(V)$.  If $\eps i=i\eps$ then $\phi_\varepsilon\iota=\Pi(\iota)\phi_\varepsilon$ and thus $\phi_\eps\in\Hom_\catC(W,\Pi(W))$. Similarly, if $\eps i=-i\eps$ then $\phi_\eps\in \Hom_\catC(W,\Pi\funcB(W))$ and $c(\phi_\eps)=\eps^2$. Finally, recall that the $W\cong \funcB(W)$ if and only if $\End_\catR(V)\cong \qH$, and in the latter case the isomorphism corresponds to $j\in\End_\catR(V)_\ood$.
These facts and Lemma~\ref{lem:EndRCH} determine the existence of symmetries $\Pi$ and $(\Pi\funcB,*)$ for $*\in\{\pm 1\}$. 
\end{proof}

\section{highest weights and  odd reflections for $\g{sl}(2|2)$ and $\g{psl}(2|2)$}
\label{App-sl22}
Let $\g g=\g{sl}(2|2)$. Without loss of generality we assume that $\g h$ is the diagonal Cartan subalgebra of $\g g$. As usual,  we choose the standard spanning set $\eps_1,\eps_2,\delta_1,\delta_2$ of $\g h^*$. Note that $\eps_1+\eps_2=\delta_1+\delta_2$.   There are two Borel subalgebras $\g b^1,\g b^2$ of $\g g$ that contain the upper triangular Borel subalgebra $\g b_\eev$ of $\g g_\eev$. They correspond to the fundamental systems 
\[
\Pi_{\g b^1}=\{\eps_1-\delta_1,\delta_1-\delta_2\}\
\quad\text{ for the positive system }
\Phi^{1,+}:=\{
\eps_1-\delta_1,\delta_1-\eps_2,\eps_1-\eps_2, \delta_1-\delta_2
\},\] and 
\[\Pi_{\g b^2}=\{\delta_1-\eps_1,\eps_1-\eps_2\}
\quad\text{ for the positive system }
\Phi^{2,+}:=\{\delta_1-\eps_1, \eps_1-\eps_2, \delta_1-\eps_2,\delta_1-\delta_2
\}.\]
For an odd root $\alpha_1$,   the subalgebra $\g g_{\alpha_1}\oplus[\g g_{\alpha_1},\g g_{-\alpha_1}]\oplus\g g_{-\alpha_1}$ is isomorphic to $\g{sl}(1|1)\oplus\g{sl}(1|1)$.  We choose standard bases $e_i,f_i,h_i$, $i\in\{1,2\}$, for the $\g{sl}(1|1)$ summands such that $e_1,e_2\in\g g_{\alpha_1}$,  $f_1,f_2\in\g g_{-\alpha_1}$, $[e_i,f_j]=0$ for $i\neq j$, and  $h_i=[e_i,f_i]$. It is straightforward to verify that the choices of $e_i,f_i,h_i$ are unique up to scaling. 

Let $W$ be an irreducible $\g g$-module with $\g b^1$-highest weight $\lambda\in\g h^*$, corresponding to the highest weight vector $w_\lambda$. In what follows, we set $\alpha_1:=\eps_1-\delta_1$.
\begin{lem}
With $W$ and $w_\lambda$ as above, the have the following assertions.
\label{lem:B1}
\begin{itemize} 
\item[\rm (i)]
Let $e_{\alpha}\in \g g_{\alpha}$ where $\alpha\in\Phi^{2,+}\backslash \{-\alpha_1\}$.
Then  $e_{\alpha} w_\lambda=0$.

\item[\rm (ii)]  Let  $e_\alpha\in \g g_\alpha$, where $\alpha\in\Phi^{2,+}\backslash \{-\alpha_1\}$. Then 
$e_\alpha f_1w_\lambda=e_\alpha f_2w_\lambda=0$. 

\item[\rm (iii)]  Let 
 $e_\alpha\in \g g_\alpha$, where $\alpha\in\Phi^{2,+}\backslash \{-\alpha_1\}$. Then 
$e_\alpha f_1f_2w_\lambda=0$. 
\end{itemize}
\end{lem}

\begin{proof}
(i) follows from $\Phi^{2,+}\backslash\{-\alpha_1\}= \Phi^{1,+}\backslash \{\alpha_1\}$.
 For (ii), either $\alpha-\alpha_1$ is not a root or $\alpha-\alpha_1\in \Phi^{1,+}$. For (iii), either $\alpha-2\alpha_1$ is not a root or $\alpha-2\alpha_1\in \Phi^{1,+}$.
\end{proof}
Let $\omega_{\lambda,\alpha_1}:\g g_{\alpha_1}\times \g g_{-\alpha_1}\to\C$ denote the bilinear form defined by $\omega_{\lambda,\alpha_1}(x,y):=\lambda([x,y])$. Then  $\mathrm{rank}(\omega_{\lambda,\alpha_1})=\#\{i:\lambda(h_i)\neq 0\}$.  
\begin{prp}\label{prp:hwsl22}
The $\g b^2$-highest weight of $W$ is obtained as follows.
\begin{itemize}
\item[\rm (i)]
If $\mathrm{rank}(\omega_{\lambda,\alpha_1})=0$, then the $\g b^2$-highest weight of $W$ is $\lambda$ and the corresponding highest weight vector is $w_\lambda$. 
\item[\rm (ii)]
If $\mathrm{rank}(\omega_{\lambda,\alpha_1})=1$, then 
the $\g b^2$-highest weight of $W$ is $\lambda-\alpha_1$ and the corresponding  highest weight vector is $e_{-\alpha_1}w_\lambda$ for any $e_{-\alpha_1}\in\g g_{-\alpha_1}$ such that  $\lambda([e_{-\alpha_1},\g g_{\alpha_1}])\neq \{0\}$.
\item[\rm (iii)]
If $\mathrm{rank}(\omega_{\lambda,\alpha_1})=2$, then 
the $\g b^2$-highest weight of $W$ is $\lambda-2\alpha_1$ and the corresponding  highest weight vector is $f_1f_2w_\lambda$.

\end{itemize}
\end{prp} 
\begin{proof}
(i) By Lemma~\ref{lem:B1}(i) it suffices to show that $f_1w_\lambda=f_2w_\lambda=0$. We have 
$e_if_jw_\lambda=
-f_je_iw_\lambda+\lambda([e_i,f_j])w_\lambda=0$ for $1\leq i,j\leq 2$. From this and
Lemma~\ref{lem:B1}(ii) it follows 
that if $f_jw_\lambda\neq 0$ for some $j\in\{1,2\}$, then
$f_j w_\lambda$ is a $\g b^1$-highest weight of $W$. But this contradicts the uniqueness (up to a scalar) of $w_\lambda$.

(ii) Without loss of generality we can assume that $\lambda(h_1)=0$ and $\lambda(h_2)=1$. By the argument of (i), we  must have $f_1w_\lambda=0$. Since \[
e_2f_2w_\lambda=-f_2e_2w_\lambda+\lambda(h_2)w_\lambda=w_\lambda,
\] we must have $f_2w_\lambda\neq 0$.  By Lemma~\ref{lem:B1}(ii) and the relations $f_1f_2w_\lambda=-f_2f_1w_\lambda=0$ and 
$f_2(f_2w_\lambda)=(f_2)^2w_\lambda=0$ it follows that $f_2w_\lambda$ (or $(f_2+cf_1)w_\lambda$ for any $c\in\C$) is a $\g b^2$-highest weight vector of $W$. 

(iii) Moving the $e_i$ past the $f_j$, we obtain  $e_2e_1f_1f_2w_\lambda=\lambda(h_1)\lambda(h_2)w_\lambda\neq 0$, hence in particular $f_1f_2w_\lambda\neq 0$. 
Lemma~\ref{lem:B1}(iii) together with the relations \[
f_1(f_1f_2w_\lambda)=(f_1^2)f_2w_\lambda=0
\] 
and \[
f_2(f_1f_2w_\lambda)=-f_1(f_2^2)w_\lambda=0,
\] imply that  $f_1f_2w_\lambda$ is a $\g b^2$-highest weight vector of $W$. 
\end{proof}
The analysis of the case of $\g g=\g{psl}(2|2)$
is the same. 
Irreducible modules of  $\g{psl}(2|2)$ are the same as irreducible $\g{sl}(2|2)$-modules satisfying $\lambda(I)=0$ where $I$ spans the center of $\g{sl}(2|2)$. The latter constraint and uniqueness up to scaling of $h_1,h_2$  imply that $\lambda(h_1)=c\lambda(h_2)$ for some $c\in\C^*$. In particular, 
the assumption of Proposition~\ref{prp:hwsl22}(ii) cannot happen. Consequently,  the $\g b_\alpha$-highest weight of $W$ is determined by parts (i) and (iii) of Proposition~\ref{prp:hwsl22}.

\end{document}